\documentclass[12pt,oneside,english]{amsart}
\usepackage[T1]{fontenc}
\usepackage[latin9]{inputenc}
\usepackage{geometry}
\usepackage{color}
\usepackage{babel}
\usepackage{amstext}
\usepackage{amsthm}
\usepackage{amssymb}
\usepackage[unicode=true,
 bookmarks=true,bookmarksnumbered=false,bookmarksopen=false,
 breaklinks=false,pdfborder={0 0 1},backref=false,colorlinks=true]
 {hyperref}
\hypersetup{pdftitle={High temperature TAP upper bound for the free energy of mean field spin glasses},
 pdfauthor={David Belius}}

\makeatletter
\numberwithin{equation}{section}
\numberwithin{figure}{section}
\theoremstyle{plain}
\newtheorem{thm}{\protect\theoremname}[section]
\newenvironment{lyxcode}
	{\par\begin{list}{}{
		\setlength{\rightmargin}{\leftmargin}
		\setlength{\listparindent}{0pt}
		\raggedright
		\setlength{\itemsep}{0pt}
		\setlength{\parsep}{0pt}
		\normalfont\ttfamily}%
	 \item[]}
	{\end{list}}
\theoremstyle{remark}
\newtheorem*{acknowledgement*}{\protect\acknowledgementname}
\theoremstyle{remark}
\newtheorem{rem}[thm]{\protect\remarkname}
\theoremstyle{plain}
\newtheorem{lem}[thm]{\protect\lemmaname}
\theoremstyle{definition}
\newtheorem{defn}[thm]{\protect\definitionname}
\theoremstyle{plain}
\newtheorem{prop}[thm]{\protect\propositionname}

\usepackage{amsaddr}
\date{}

\makeatother

\usepackage[style=alphabetic,maxbibnames=99]{biblatex}
\providecommand{\acknowledgementname}{Acknowledgement}
\providecommand{\definitionname}{Definition}
\providecommand{\lemmaname}{Lemma}
\providecommand{\propositionname}{Proposition}
\providecommand{\remarkname}{Remark}
\providecommand{\theoremname}{Theorem}

\begin{document}
\title[High temperature TAP upper bound]{High temperature TAP upper bound for the free energy of mean field
spin glasses}
\author{David Belius}
\email{david.belius@cantab.net}
\address{Department of Mathematics and Computer Science, University of Basel,
Switzerland.}
\thanks{Research supported by SNSF grant 176918. }
\begin{abstract}
This work proves an upper bound for the free energy of the Sherrington-Kirkpatrick
model and its generalizations in terms of the Thouless-Anderson-Palmer
(TAP) energy. The result applies to models with spherical or Ising
spins and any mixed $p$-spin Hamiltonian with external field or with
a non-linear spike term. The bound is expected to be tight to leading
order at high temperature, and is non-trivial in the presence of an
external field. For the proof a geometric microcanonical method is
employed, in which one covers the spin space with sets, each of which
is centered at a magnetization vector $m$ and whose contribution
to the partition function is bounded in terms of the TAP energy at
$m$.
\end{abstract}

\maketitle

\section{Introduction}

This article proves an upper bound for the free energy of spherical
or Ising mixed $p$-spin spin glass models in terms of the maximum
of their Thouless-Anderson-Palmer (TAP) energy. This is a step towards
the computation of the free energy of spin glass models completely
within a geometric, microcanonical TAP framework. Though its goal
is similar, as explained below this approach is fundamentally different
from other recent TAP approaches in the mathematical literature.

To formally state the results consider for $N\ge1$ an inner product
on $\mathbb{R}^{N}$ given by $\left\langle a,b\right\rangle =\frac{1}{N}\sum a_{i}b_{i}$,
so that the corresponding norm $\|\cdot\|=\sqrt{\left\langle \cdot,\cdot\right\rangle }$
satisfies $\|a\|^{2}=\frac{1}{N}\sum_{i=1}^{N}a_{i}^{2}$. Let $B_{N}\left(r\right)=\left\{ \sigma\in\mathbb{R}^{N}:\|\sigma\|\le r\right\} $
and $B_{N}^{\circ}\left(r\right)=\left\{ \sigma\in\mathbb{R}^{N}:\|\sigma\|<r\right\} $
denote the closed resp. open ball of radius $r$ and let $B_{N}=B_{N}\left(1\right)$
and $B_{N}^{\circ}=B_{N}^{\circ}\left(1\right)$. Due to our convention
$\left\{ -1,1\right\} ^{N}\subset B_{N-1}$.

Let $H_{N}\left(\sigma\right)$ be any mixed $p$-spin Hamiltonian
on $B_{N}$, that is a centered Gaussian process indexed by the \emph{spin
vectors} $\sigma\in B_{N}$ with covariance
\begin{equation}
\mathbb{E}\left[H_{N}\left(\sigma\right)H_{N}\left(\sigma'\right)\right]=N\xi\left(\left\langle \sigma,\sigma'\right\rangle \right)\text{\,for }\sigma,\sigma'\in B_{N},\label{eq: covar}
\end{equation}
for a power series $\xi\left(x\right)=\sum_{p\ge0}a_{p}x^{p}$ with
$a_{p}\ge0$ and $\xi\left(1\right)<\infty$. The inverse temperature
is denoted by $\beta\ge0$.

Let $K\in\left\{ 1,\ldots,N\right\} $ and $U$ be a linear subspace
of $\mathbb{R}^{N}$ of dimension $K$, and let $P^{U}$ denote projection
onto $U$. Let $f_{N}:B_{N}\to\mathbb{R}$ be a function representing
a generalized external field such that $f_{N}\left(\sigma\right)=f_{N}\left(P^{U}\sigma\right)$
for all $\sigma\in B_{N}$, and write
\begin{equation}
H_{N}^{f}\left(\sigma\right)=H_{N}\left(\sigma\right)+f_{N}\left(\sigma\right)\text{ for }\sigma\in B_{N}.\label{eq: HNf}
\end{equation}
With $f_{N}\left(\sigma\right)=h\sum_{i=1}^{N}\sigma_{i}$ (and $K=1$)
we obtain a standard linear external field and with $f_{N}\left(x\right)=\frac{h}{N}(\sum_{i=1}^{N}\sigma_{i})^{2}$
a quadratic spike as studied for instance in \cite{benarousLandscapeSpikedTensor2018}.

Let the \emph{Onsager term} of the TAP energy be given by
\begin{equation}
{\rm On}\left(q\right)=\xi\left(1\right)-\left(1-q\right)\xi'\left(q\right)-\xi\left(q\right),q\in\left[0,1\right].\label{eq: Onsager term def}
\end{equation}
The TAP energy of the mixed $p$-spin model with Ising spins is 
\begin{equation}
H_{{\rm TAP}}^{{\rm Ising}}\left(m\right)=\beta H_{N}^{f}\left(m\right)+I_{{\rm Ising}}\left(m\right)+\frac{\beta^{2}}{2}{\rm On}\left(\|m\|^{2}\right),\label{eq: HTAP ising def}
\end{equation}
for the entropy term
\begin{equation}
I_{{\rm Ising}}\left(m\right)=-\sum_{i=1}^{N}J\left(m_{i}\right),\label{eq: ising entropy I}
\end{equation}
where $J$ is the binary entropy function given by
\begin{equation}
J\left(m\right)=\frac{1+m}{2}\log\left(1+m\right)+\frac{1-m}{2}\log\left(1-m\right).\label{eq: binary entropy def J}
\end{equation}
Our main result for the Ising spin model is the following.
\begin{thm}[Ising TAP Upper Bound]
\label{thm: main_thm_ising}For all $\delta,L\in\left(0,\infty\right)$
there is a constant $c_{1}=c_{1}\left(\delta,L\right)$ such that
the following holds. Let $N\ge1$ and $E$ be the uniform distribution
on $\left\{ -1,1\right\} ^{N}$. Assume that $0\le\beta,\sqrt{\xi'''\left(1\right)}\le L$.
Let $U$ have dimension $1\le K\le c_{1}N/\log N$ and $f_{N}:B_{N}\to\mathbb{R}$
be Lipschitz on $B_{N}$ with respect to $\|\cdot\|$ with Lipschitz
constant at most $LN$. Then
\begin{equation}
\mathbb{P}\left(\log E\left[\exp\left(\beta H_{N}^{f}\left(\sigma\right)\right)\right]\le\sup_{m\in\left(-1,1\right)^{N}}H_{{\rm TAP}}^{{\rm Ising}}\left(m\right)+\delta N\right)\ge1-c_{1}^{-1}e^{-c_{1}N}.\label{eq: Ising TAP UB}
\end{equation}
\end{thm}

For the spherical mixed $p$-spin model the TAP energy is given by
\begin{equation}
H_{{\rm TAP}}^{\text{sph}}\left(m\right)=\beta H_{N}^{f}\left(m\right)+I_{{\rm sph}}\left(m\right)+\frac{\beta^{2}}{2}{\rm On}\left(\|m\|^{2}\right),\label{eq: HTAP spherical def}
\end{equation}
for the entropy term
\begin{equation}
I_{{\rm sph}}\left(m\right)=\frac{N}{2}\log\left(1-\|m\|^{2}\right),\label{eq: spherical entropy I}
\end{equation}
and our main result is the following. Let $S_{N-1}\left(r\right)=\left\{ \sigma\in\mathbb{R}^{N}:\|\sigma\|=r\right\} $
denote the sphere of radius $r$, and write $S_{N-1}=S_{N-1}\left(1\right)$.\\
\noindent\begin{minipage}[t]{1\columnwidth}%
\begin{thm}[Spherical TAP Upper Bound]
\label{thm: main_thm_spherical}For all $\delta,L\in\left(0,\infty\right),$
there is a constant $c_{2}=c_{2}\left(\delta,L\right)$ such that
the following holds. Let $N\ge1$ and $E$ be the the uniform distribution
on $S_{N-1}$. Assume that $0\le\beta,\sqrt{\xi'''\left(1\right)}\le L$.
Let $U$ have dimension $1\le K\le c_{2}N/\log N$ and $f_{N}:B_{N}\to\mathbb{R}$
be Lipschitz on $B_{N}$ with respect to $\|\cdot\|$ with Lipschitz
constant at most $LN$. Then
\begin{equation}
\mathbb{P}\left(\log E\left[\exp\left(\beta H_{N}^{f}\left(\sigma\right)\right)\right]\le\sup_{m\in B_{N}}H_{{\rm TAP}}^{{\rm sph}}\left(m\right)+\delta N\right)\ge1-c_{2}^{-1}e^{-c_{2}N}.\label{eq: spherical main result}
\end{equation}
\end{thm}

\end{minipage}
\begin{lyxcode}
\end{lyxcode}
The bounds (\ref{eq: Ising TAP UB}) and (\ref{eq: spherical main result})
for the free energy are expected to be tight to leading order at high
temperature. In the absence of an external field the bounds are no
stronger than the annealed upper bound and are thus trivial, but in
the presence of an external field obtaining such upper bounds is a
difficult problem. In future work the author plans to rigorously compute
the maximal TAP energy for some of these models (building on results
such as \cite{AuffingerBenArousCernyComplexityofSpinGlasses,fyodorovHighDimensionalRandomFields2013,SubagComplexityOfSphericalPspinmodelsA2ndmomentapproach,zhoufanTAPFreeEnergy2021,beliusTrivialityGeometryMixed2022}),
which combined with the this article will yield concrete upper bounds
for the free energy.

An important question is whether a similar upper bound that is tight
at all temperatures can be proven, in addition to a matching lower
bound. Such bounds must involve conditions which rule out some $m$'s,
including at least the Plefka condition \cite{TAPSolutionOfSolvableModelOfASpinGlass,PlefkaConvergenceCondOftheTAPequations}
of the physics literature. For the special case of the spherical
$2$-spin model with linear external field (i.e. $E$ is uniform on
$S_{N-1}$, $\xi\left(x\right)=x^{2}$ and $f_{N}\left(\sigma\right)=h\sum_{i=1}^{N}\sigma_{i}$)
the work \cite{BeliusKistler2spin} shows that in fact for all $\beta,h\ge0$
\[
\log E\left[\exp\left(\beta H_{N}^{f}\left(\sigma\right)\right)\right]=\sup_{m:\beta\left(1-\|m\|^{2}\right)\le\frac{1}{\sqrt{2}}}H_{{\rm TAP}}^{{\rm sph}}\left(m\right)+o\left(N\right),
\]
where the condition on $m$ is precisely Plefka's condition for this
model. Thus it shows that matching upper and lower bounds hold at
any temperature once the Plefka condition is added to the sup. An
extension of this result to the present setting would pave the way
for the computation of the free energy of a larger class of spin glass
models, including at low temperature, using a geometric TAP approach.

Theorems \ref{thm: main_thm_ising}-\ref{thm: main_thm_spherical}
follow from a more general bound for a general spin reference measure
$E$ on the sphere $S_{N-1}$, which need not be a product measure.
To state it define the general entropy function
\begin{equation}
I_{E,\delta}\left(m\right)=\inf_{\lambda\in\mathbb{R}^{N},\|\lambda\|=1}\log E\left[\left\{ \sigma\in S_{N-1}:\left\langle \lambda,\sigma-m\right\rangle \ge-\delta\right\} \right]\in[-\infty,0],\label{eq: IN}
\end{equation}
for any probability measure $E$ on $S_{N-1}$, and a general TAP
energy by
\begin{equation}
H_{{\rm TAP}}^{E,\delta}\left(m\right)=\beta H_{N}^{f}\left(m\right)+I_{E,\delta}\left(m\right)+\frac{\beta^{2}}{2}{\rm On}\left(\|m\|^{2}\right).\label{eq: HTAP measure on subset of sphere}
\end{equation}
Our general bound is the following.

\noindent\begin{minipage}[t]{1\columnwidth}%
\begin{thm}[General TAP Upper Bound]
\label{thm: main_thm_general}For any $\delta\in\left(0,1\right),\beta\ge0,1\le K\le N,L\in\left(0,\infty\right)$
there is a constant $\kappa=\kappa\left(\delta,\beta,K,L\right)$
such that the following holds. Let $N\ge1$, and $E$ be any probability
measure on $S_{N-1}$. Assume $\sqrt{\xi'''\left(1\right)}\le L$.
Let $U$ have dimension $K$ and $f_{N}:B_{N}\to\mathbb{R}$ be Lipschitz
on $B_{N}$ with respect to $\|\cdot\|$ with Lipschitz constant at
most $LN$. Then it holds that
\begin{equation}
\mathbb{P}\left(\log E\left[\exp\left(\beta H_{N}^{f}\left(\sigma\right)\right)\right]\le\frac{1}{N}\sup_{m\in B_{N}}H_{{\rm TAP}}^{E,\delta}\left(m\right)+\delta\right)\ge1-\kappa e^{-\frac{\delta}{2}N}.\label{eq: general UB}
\end{equation}
More explicitly, the constant $\kappa$ satisfies $\kappa\le\bar{\kappa}^{\bar{\kappa}}$
for $\bar{\kappa}=c\max\left(K,\left(\beta L/\delta\right)^{8}\right)$
and a universal constant $c$.
\end{thm}

\end{minipage}

\medskip{}
Note that the only difference between (\ref{eq: HTAP ising def}),
(\ref{eq: HTAP spherical def}) and (\ref{eq: HTAP measure on subset of sphere})
is the corresponding entropy term $I_{{\rm Ising}},I_{{\rm sph}},I_{E,\delta}$.
If $E$ is the uniform measure on $\left\{ -1,1\right\} ^{N}$, as
in Theorem \ref{thm: main_thm_ising}, it turns out that $I_{E,\delta}$$\left(m\right)$
is $-\infty$ if $m$ is sufficiently far from $\left(-1,1\right)^{N}$,
and otherwise it can be bounded above in terms of $I_{{\rm Ising}}\left(m\right)$.
Similarly if $E$ is the uniform measure on the sphere $S_{N-1}$,
as in Theorem \ref{thm: main_thm_spherical}, it turns out that $I_{E,\delta}\left(m\right)$
can be bounded above in terms of $I_{{\rm sph}}\left(m\right)$. These
bound will be proved and used to derive Theorems \ref{thm: main_thm_spherical}-\ref{thm: main_thm_general}
from the general Theorem \ref{thm: main_thm_general} in Sections
\ref{sec: Ising SK proof}-\ref{sec: Spherical SK proof}.

The proof of Theorem \ref{thm: main_thm_general} is based on covering
the sphere with sets, each centered on a ``magnetization'' vector
$m\in\mathbb{R}^{N}$, on which one can expand the Hamiltonian $H_{N}\left(\sigma\right)$
around $m$, giving rise to a recentered ``effective'' Hamiltonian
$H_{N}^{m}\left(\hat{\sigma}\right)$ (see (\ref{eq: recent hamilt sketch})
and (\ref{eq: recentered hamil})-(\ref{eq: decompostion})) for which
the \emph{external field effectively vanishes}. The integral of the
Gibbs factor $\exp(\beta H_{N}^{f}\left(\sigma\right))$ over the
set of the cover centered at $m$ can be bounded above by $\exp(H_{{\rm TAP}}^{E,\delta}(m)+o(N))$
using an annealed upper bound (Markov inequality). The number of sets
in the cover will be seen to grow slowly with $N$, and therefore
a union bound over all the sets will be enough to obtain the upper
bound (\ref{eq: general UB}). The approach can be seen as a TAP-informed
sophisticated moment method. A more detailed sketch is provided in
Section \ref{sec: sketch}.

\subsection{Related work and historical remarks}

Mean field-spin glasses were introduced in \cite{SKSolvableModelOfASpinGlass}
as toy models of the properties of exotic magnetic alloys, and they
\cite{KosterlitzThoulessJonesSphericalModelofASpinGlass,CrisantiSommersTheSphericalPspinInteractionSGModel,talagrandMultipleLevelsSymmetry2000,crisantiSphericalSpinglassModel2004,TalagrandFEOftheSphericalMeanFieldModel}
and related models have since become paradigmatic examples of complex
systems \cite{MezardParisiVirasoro-SpinGlassTheoryandBeyond,kimCoveringCubesRandom1998,DingSlySunProofoftheSatisfiabilityConjectureforLargek,dingCapacityLowerBound2019,mezard2009information}.
Their investigation can alternatively be thought of as the study of
the extrema of highly correlated high dimensional random fields.

The \emph{partition function }is the integral 
\[
Z_{N}=E\left[\exp\left(\beta H_{N}^{f}\left(\sigma\right)\right)\right],
\]
over the spins $\sigma$ against a reference measure $E$ (the uniform
distribution on $\left\{ -1,1\right\} ^{N}$ for the Ising spin model,
and the uniform measure on the sphere $S_{N-1}$ for the spherical
model). The \emph{free energy} is the exponent
\[
F_{N}=\frac{1}{N}\log Z_{N}.
\]
The spin vector $\sigma$ under the \emph{Gibbs measure $\mathcal{G}\left(A\right)=E[1_{A}\exp(\beta H_{N}^{f}(\sigma))]/Z_{N}$}
models for instance the aforementioned exotic materials. Computing
the free energy for large $N$ is a first step towards determining
the behavior of $\sigma$ under the Gibbs measure, which is the ultimate
goal of studying these models.

The TAP energy was introduced in \cite{TAPSolutionOfSolvableModelOfASpinGlass}
for the purpose of solving the Sherrington-Kirkpatrick model. Presumably
the original motivation was to devise a framework to - among other
things - compute the free energy. In physics, the computation of the
free energy was however achieved by the replica symmetry breaking
ansatz and the replica method of Parisi \cite{ParisiSeqofApproxSolsToSK,parisi1979infinite},
and in mathematics by the interpolation method of Guerra \cite{GuerraBrokenReplicaSymmetryBounds}
together with the methods of Talagrand, Aizenman-Sim-Starr and Panchenko
\cite{aizenmanExtendedVariationalPrinciple2003,TalagrandTheParisiFormula,talagrandFreeEnergySpherical2006,panchenkoParisiFormulaMixed2014,chenAizenmanSimsStarrSchemeParisi2013,PanchenkoTheSKModel},
which are very different approaches. The study of the TAP energy has
played a complementary role in the analysis of the model in physics
\cite{mezard1987spin,brayMetastableStatesSpin1980,dedominicisWeightedAveragesOrder1983,grossSimplestSpinGlass1984,kurchanBarriersMetastableStates1993,CrisantiSommersTAPApproachtoSphericalPspinSGModels,cavagnaFormalEquivalenceTAP2003}
and mathematics \cite[Section 1.7]{ChatterjeeSpinGlassesandSteinsMethod,auffingerThoulessAndersonPalmer2019,auffingerSpinDistributionsGeneric2019,chenTAPFreeEnergy2018,chenGeneralizedTAPFree2021,talagrandMeanFieldModels2011a},
rather than being fully developed as a stand-alone solution of it
(see however recent work mentioned below).

Initial steps towards the development of the TAP approach as a stand-alone
solution were taken in \cite{kistlerPrivateCommunication2016,BeliusKistler2spin}.
As mentioned above, the article \cite{BeliusKistler2spin} computed
the free energy of the $2$-spin spherical model at all temperatures
and linear external fields in terms of the TAP energy. This article
takes a further step, by giving an upper bound for the free energy
in terms of the TAP energy in a much more general setting.

An alternative approach to computing the free energy using the TAP
energy was initiated by Subag \cite{SubagFreeEnergyLandscapesinSphericalSpinGlasses,chenGeneralizedTAPFree2021,chenGeneralizedTAPFree2022,subagFreeEnergySpherical2021}.
This involves properties of the limiting Gibbs measure such as the
concept of ``multisamplable overlap'' and is therefore very different
from the approach of the current article which is microcanonical and
works on the level of spin configurations for finite $N$. Indeed,
note that Theorems \ref{thm: main_thm_ising}-\ref{thm: main_thm_general}
are all quantitative finite $N$ statements.

An earlier stream of work also initiated by Subag computes the free
energy of certain spherical spin glasses at low enough temperature
\cite{subagGeometryGibbsMeasure2017,benarousGeometryTemperatureChaos2020}.
These, and also \cite{arousShatteringMetastabilitySpin2021}, use
a moment method and are more similar in spirit to the present approach.
The eventual goal of the present research project is to compute the
free energy for all models and all temperatures using a TAP-informed
moment method.

Another TAP approach involving an iterative solution of the TAP equations
(critical point equations of the TAP energy) was initiated by Bolthausen
\cite{BolthausenAnIterativeConstructionOfSoloftheTAPequations,bolthausenMoritaTypeProof2018,brenneckeNoteReplicaSymmetric2021}.
The iterative construction of a cover of the sphere in the present
article bears some similarity to this iterative solution of the TAP
equations. Those works also use a moment method. The way iteration
and moment method are used is however quite different; the aforementioned
articles construct one sequence of iterates that converge to the conjecturally
unique TAP solution at high temperature, while we here construct a
hierarchy of iterates whose associated sets cover the whole sphere.
Morally speaking, one of the iterates in our hierarchy should be the
iterate that Bolthausen's algorithm produces.

A further difference compared to the aforementioned works is that
here bounding the free energy in terms of the TAP energy is neatly
decoupled from the study of the behavior of the TAP energy itself
(indeed, the present work deals only with the former and leaves the
latter for later research).

\subsection{Overview of article}

In Section \ref{sec: sketch} we give a detailed sketch of the proofs
of Theorems \ref{thm: main_thm_ising}-\ref{thm: main_thm_general},
motivating the construction of the cover of $S_{N-1}$. In Section
\ref{sec: law of recent} we formally introduce the recentered Hamiltonian
and study its law. In Section \ref{sec: magnetization construciton}
we give the iterative construction of magnetizations, which we then
use in Section \ref{sec: construction of cover} to construct the
cover of $S_{N-1}$. Then in Section \ref{sec: proof of gen UB} the
construction is used to prove the general TAP upper bound Theorem
\ref{thm: main_thm_general}. In Section \ref{sec: Ising SK proof}
the Ising upper bound Theorem \ref{thm: main_thm_ising} is derived
from the general result, and in Section \ref{sec: Spherical SK proof}
the spherical upper bound Theorem \ref{thm: main_thm_spherical} is
similarly derived. The appendix contains some basic results about
the Hamiltonian that follow from the classical theory of Gaussian
processes.

We use $c$ to denote unspecified positive constants, whose numerical
value may be different each time the notation $c$ is used, even within
the same formula. The standard inner product is denoted by $a\cdot b$
and the standard norm by $\left|\cdot\right|$, so that $a\cdot b=\sum_{i=1}^{N}a_{i}b_{i}=N\left\langle a,b\right\rangle $
and $\left|a\right|=\sqrt{\sum_{i=1}^{N}a_{i}^{2}}=\sqrt{N}\|a\|$
for $a,b\in\mathbb{R}^{N}$.
\begin{acknowledgement*}
The author is grateful to Erwin Bolthausen, Jiri Cerny, Francesco
Concetti, Giuseppe Genovese and Shuta Nakajima for their close reading
of a draft of this article and valuable comments for improvement.
\end{acknowledgement*}

\section{\label{sec: sketch}Sketch of proof}

As the construction of the magnetizations and cover in Sections \ref{sec: magnetization construciton}-\ref{sec: construction of cover}
is quite involved, this section gives a detailed sketch of the proofs
which motivates it.

The goal is to use annealed upper bounds (i.e. the Markov inequality)
to obtain a bound for the free energy that is tight to leading order,
even in the presence of an external field.

\subsection{\label{subsec: sketch sphere}Sphere with linear external field}

First let us sketch a direct proof of the bound with a linear external
field, that is with
\begin{equation}
H_{N}^{f}\left(\sigma\right)=H_{N}\left(\sigma\right)+Nh\left\langle \sigma,u_{1}\right\rangle ,\label{eq: HNf sketch}
\end{equation}
for a $\|\cdot\|$-unit vector $u_{1}$, and for the spherical model
where $E$ denotes the uniform measure on the sphere $S_{N-1}$ (cf.
Theorem \ref{thm: main_thm_spherical}).

The standard annealed upper bound for the free energy $F_{N}$ is
obtained from
\begin{equation}
Z_{N}\le\mathbb{E}\left[E\left[\exp\left(\beta H_{N}^{f}\left(\sigma\right)\right)\right]\right]e^{o\left(N\right)},\label{eq: annealed bound}
\end{equation}
and is a simple consequence of the Markov inequality. If $h=0$ the
integral on the RHS equals 
\begin{equation}
\exp\left(\frac{1}{2}\text{Var}\left(\beta H_{N}^{f}\left(\sigma\right)\right)\right)=\exp\left(N\frac{\beta^{2}}{2}\xi\left(1\right)\right).\label{eq: RHS 2}
\end{equation}

If the covariance is $\xi\left(x\right)=\sum_{p\ge0}a_{p}x^{p}$ with
$a_{0}=a_{1}=0$ and $h=0$ then (\ref{eq: annealed bound}) gives
a bound for the free energy that is tight to leading order at high
temperature (for small enough $\beta$ this can be verified by proving
a matching lower bound using a simple second moment method). By contrast,
if at least one of $a_{0},a_{1},h$ are non-zero then the law of the
Hamiltonian $H_{N}^{f}\left(\sigma\right)$ is that of $\mbox{\ensuremath{A_{0}}+\ensuremath{A_{1}\cdot\sigma}+\ensuremath{\tilde{H}_{N}\left(\sigma\right)}}$
where $A_{0}\sim\mathcal{N}\left(0,Na_{0}\right),A_{1}\sim\mathcal{N}\left(hu_{1},Na_{1}I\right)$
and $\tilde{H}_{N}\left(\sigma\right)$ are independent, and $\tilde{H}_{N}\left(\sigma\right)$
is a Hamiltonian with covariance function $\tilde{\xi}\left(x\right)=\sum_{p\ge2}a_{p}x^{p}$.
A non-vanishing global shift $A_{0}$ or a non-vanishing (random)
external field $A_{1}\cdot\sigma$ will both individually cause the
annealed upper bound (\ref{eq: annealed bound}) to overestimate the
free energy to leading order.

In this sketch we are interested in rectifying this to obtain a bound
that is tight - at least for some $\beta$ - when $\xi$ only has
terms of order $2$ and higher, but $h>0$. Roughly speaking we do
this by covering the sphere with a finite number of regions where
the \emph{effective} external field vanishes.

Define the partition function restricted to a region $A$ by
\begin{equation}
Z_{N}\left(A\right)=E\left[1_{A}\exp\left(\beta H_{N}\left(\sigma\right)+N\beta h\left\langle \sigma,u_{1}\right\rangle \right)\right].\label{eq: rest part func}
\end{equation}
For the ``equator''
\begin{equation}
\mathcal{E}=\left\{ \sigma:\left|\left\langle \sigma,u_{1}\right\rangle \right|\le\eta\right\} ,\label{eq: first equator}
\end{equation}
we have the bound
\begin{equation}
Z_{N}\left(\mathcal{E}\right)\le E\left[\exp\left(\beta H_{N}\left(\sigma\right)\right)\right]e^{N\beta h\eta}.\label{eq: eliminate ext field}
\end{equation}
Applying the annealed bound as in (\ref{eq: annealed bound})-(\ref{eq: RHS 2})
to the integral on the RHS yields
\begin{equation}
\begin{array}{rcccl}
E\left[\exp\left(\beta H_{N}\left(\sigma\right)\right)\right] & \le & \exp\left(N\frac{\beta^{2}}{2}\xi\left(1\right)+o\left(N\right)\right) & \overset{\eqref{eq: Onsager term def}}{=} & \exp\left(N\frac{\beta^{2}}{2}\text{On}\left(0\right)+o\left(N\right)\right)\\
 &  &  & \overset{\eqref{eq: HTAP spherical def}}{=} & \exp\left(H_{{\rm TAP}}^{{\rm sph}}\left(0\right)+o\left(N\right)\right),
\end{array}\label{eq: anneal on first equator}
\end{equation}
so we obtain from (\ref{eq: eliminate ext field})
\begin{equation}
Z_{N}\left(\mathcal{E}\right)\le\exp\left(H_{{\rm TAP}}^{{\rm sph}}\left(0\right)+N\beta h\eta+o\left(N\right)\right).\label{eq: equator bound}
\end{equation}
This will give a bound for $\frac{1}{N}\log Z_{N}\left(\mathcal{E}\right)$
that is at high temperature tight to leading order in the limits $N\to\infty$
and then $\eta\downarrow0$, since the covariance $\xi$ of $H_{N}\left(\sigma\right)$
only has terms of order $2$ and higher.

Eq. (\ref{eq: equator bound}) bounds the contribution of the equator
$\mathcal{E}$ to the partition function in terms of the TAP energy
at $m=0$. To get a bound for the actual partition function $Z_{N}=Z_{N}\left(S_{N-1}\right)$
we need to also bound the contribution of $\mathcal{E}^{c}$. It is
natural to decompose $\mathcal{E}^{c}$ according to the value of
$\left\langle \sigma,u_{1}\right\rangle $ using the sets
\begin{equation}
D_{\left(\alpha_{1}\right)}=\left\{ \sigma\in S_{N-1}:\left\langle \sigma,u_{1}\right\rangle \in(\left|\alpha_{1}\right|,\left|\alpha_{1}\right|+\varepsilon]\times\text{sign}\left(\alpha_{1}\right)\right\} \text{\,for }\alpha_{1}\in\left(-1,1\right),\label{eq: dalpha1 def}
\end{equation}
(the somewhat unusual expression on the RHS is used because it's convenient
to have $\left|\left\langle \sigma,u_{1}\right\rangle \right|>\left|\alpha_{1}\right|$
for $\sigma\in D_{\left(\alpha_{1}\right)}$). Defining the $\varepsilon$-spaced
grid 
\begin{equation}
I_{\varepsilon,\eta}=\left(\varepsilon\mathbb{Z}\right)\cap\left(-1,1\right)\backslash\left[-\frac{\eta}{2},\frac{\eta}{2}\right],\label{eq: I eps eta grid}
\end{equation}
we have for $\varepsilon\le\eta/2$
\begin{equation}
S_{N-1}=\mathcal{E}\cup\left({\displaystyle \bigcup_{\alpha_{1}\in I_{\varepsilon,\eta}}}D_{\left(\alpha_{1}\right)}\right).\label{eq: sphere decomp with dalpha1}
\end{equation}
On each $D_{\left(\alpha_{1}\right)}$ the external field term is
essentially constant (for small $\varepsilon$), so from (\ref{eq: rest part func})
we can bound
\begin{equation}
Z_{N}\left(D_{\left(\alpha_{1}\right)}\right)\le E\left[1_{D_{\left(\alpha_{1}\right)}}\exp\left(\beta H_{N}\left(\sigma\right)\right)\right]e^{N\beta h\alpha_{1}+N\beta h\varepsilon}.\label{eq: Dalpha1}
\end{equation}
One may approximate the set $D_{\left(\alpha_{1}\right)}$ by the
set
\begin{equation}
\tilde{D}=\left\{ \sigma:\left\langle \sigma,u_{1}\right\rangle =\alpha_{1}\right\} \subset D_{\left(\alpha_{1}\right)}.\label{eq: D tilde}
\end{equation}
Let $E^{\tilde{D}}$ denote the uniform measure on $\tilde{D}$. Using
that the Hamiltonian has Lipschitz constant with respect to $\|\cdot\|$
of order $N$ with high probability (see (\ref{eq: lipschitz})) the
right-hand side of (\ref{eq: Dalpha1}) can be shown to equal
\begin{equation}
E\left[D_{\left(\alpha_{1}\right)}\right]E^{\tilde{D}}\left[\exp\left(\beta H_{N}\left(\sigma\right)\right)\right]e^{N\beta h\alpha_{1}+O\left(\varepsilon N\right)}.\label{eq: approx}
\end{equation}

It is natural to now apply the annealed upper bound as in (\ref{eq: annealed bound})-(\ref{eq: RHS 2})
to $E^{\tilde{D}}\left[\exp\left(\beta H_{N}\left(\sigma\right)\right)\right]$.
Unfortunately, this will not give a tight bound, essentially because
of the presence of an effective external field, as we describe below.
To see this at the level of the covariance of the Hamiltonian, let
\begin{equation}
m_{\left(\alpha_{1}\right)}=\alpha_{1}u_{1},\label{eq: m alpha1}
\end{equation}
be the ``center'' of the sets $\tilde{D},D_{\left(\alpha_{1}\right)}$
and consider for $\sigma\in\tilde{D}$ the change of variables 
\begin{equation}
\sigma=m_{\left(\alpha_{1}\right)}+\hat{\sigma}\text{ for }\hat{\sigma}\in\text{span}(m_{\left(\alpha_{1}\right)})^{\bot}\cap S_{N-1}\left(\sqrt{1-\alpha_{1}^{2}}\right).\label{eq: change of var in slice}
\end{equation}
An easy computation shows that the process $\hat{\sigma}\to H_{N}(m_{\left(\alpha_{1}\right)}+\hat{\sigma})$
has covariance function
\begin{equation}
z\to\xi\left(\alpha_{1}^{2}+z\right)=\xi\left(\alpha_{1}^{2}\right)+\xi'\left(\alpha_{1}^{2}\right)z+\frac{1}{2}\xi''\left(\alpha_{1}^{2}\right)z^{2}+\ldots,\label{eq: xi z}
\end{equation}
which is a power series that for $\alpha_{1}\ne0$ has terms of order
$0$ and $1$ in $z$. Therefore by the discussion after (\ref{eq: RHS 2})
an annealed upper bound for $E^{\tilde{D}}\left[\exp\left(\beta H_{N}\left(\sigma\right)\right)\right]$
can not be tight. The origin of these problematic terms of (\ref{eq: xi z})
can be understood by defining for any $m$ with $\|m\|<1$ a \emph{recentered
Hamiltonian} $H_{N}^{m}\left(\hat{\sigma}\right)$ by the expansion
\begin{equation}
H_{N}\left(m+\hat{\sigma}\right)=H_{N}\left(m\right)+\nabla H_{N}\left(m\right)\cdot\hat{\sigma}+H_{N}^{m}\left(\hat{\sigma}\right).\label{eq: recent hamilt sketch}
\end{equation}
It turns out that for fixed $m$
\[
H_{N}\left(m\right),\left(\nabla H_{N}\left(m\right)\cdot\hat{\sigma}\right)_{\hat{\sigma}:\hat{\sigma}\cdot m=0},\left(H_{N}^{m}\left(\hat{\sigma}\right)\right)_{\hat{\sigma}:\hat{\sigma}\cdot m=0}
\]
are independent, and $\left(H_{N}^{m}\left(\hat{\sigma}\right)\right)_{\hat{\sigma}:\hat{\sigma}\cdot m=0}$
has covariance function 
\begin{equation}
z\to\xi_{\|m\|^{2}}\left(z\right),\text{\,where }\xi_{q}\left(z\right)=\xi\left(q+z\right)-z\xi'\left(q\right)-\xi\left(q\right),\label{eq: recent cover-1}
\end{equation}
(see Lemma \ref{lem: law of recentering}). In the covariance function
$\xi_{q}\left(z\right)$ the first and second order terms that appear
in (\ref{eq: xi z}) are removed by construction, so $H_{N}^{m_{\left(\alpha_{1}\right)}}\left(\hat{\sigma}\right)$
is a Hamiltonian for which an annealed upper bound can be tight in
the absence of an external field. However $H_{N}\left(\sigma\right)$
on the set $\tilde{D}$ consists as seen in (\ref{eq: recent hamilt sketch})
of not only the well-behaved $H_{N}^{m_{\left(\alpha_{1}\right)}}\left(\hat{\sigma}\right)$,
but also of a random mean $H_{N}\left(m_{\left(\alpha_{1}\right)}\right)$
and a random external field $\nabla H_{N}\left(m_{\left(\alpha_{1}\right)}\right)$.
With
\[
h_{\text{eff}}=\nabla H_{N}\left(m_{\left(\alpha_{1}\right)}\right),
\]
we get from (\ref{eq: recent hamilt sketch})
\[
E^{\tilde{D}}\left[\exp\left(\beta H_{N}\left(\sigma\right)\right)\right]=\exp\left(\beta H_{N}\left(m_{\left(\alpha_{1}\right)}\right)\right)E^{\tilde{D}}\left[\exp\left(\beta H_{N}^{m_{\left(\alpha_{1}\right)}}\left(\hat{\sigma}\right)+N\beta\left\langle h_{\text{eff}},\hat{\sigma}\right\rangle \right)\right].
\]
The presence of an external field in the integral on the RHS again
confirms that an annealed upper bound can not be tight. But if we
limit ourselves to the \emph{equator inside} $D_{\left(\alpha_{1}\right)}$,
namely
\begin{equation}
\mathcal{E}'=\left\{ \sigma\in D_{\left(\alpha_{1}\right)}:\left|\left\langle \hat{\sigma},h_{\text{eff}}\right\rangle \right|\le\eta\right\} ,\label{eq: e tilde equator}
\end{equation}
we can similarly to in (\ref{eq: eliminate ext field}) eliminate
the external field term via 
\begin{equation}
E^{\tilde{D}}\left[1_{\mathcal{E}'}\exp\left(\beta H_{N}^{m_{\left(\alpha_{1}\right)}}\left(\hat{\sigma}\right)+N\beta\left\langle h_{\text{eff}},\hat{\sigma}\right\rangle \right)\right]\overset{\eqref{eq: e tilde equator}}{\le}E^{\tilde{D}}\left[\exp\left(\beta H_{N}^{m_{\left(\alpha_{1}\right)}}\left(\hat{\sigma}\right)\right)\right]e^{N\beta\eta}.\label{eq: UB on heff eq}
\end{equation}
After this the annealed upper bound
\begin{equation}
\begin{array}{lcl}
E^{\tilde{D}}\left[\exp\left(\beta H_{N}^{m_{\left(\alpha_{1}\right)}}\left(\hat{\sigma}\right)\right)\right] & \le & \exp\left(\frac{\beta^{2}}{2}\text{Var}\left(H_{N}^{m_{\left(\alpha_{1}\right)}}\left(\hat{\sigma}\right)\right)+o\left(N\right)\right)\\
 & = & \exp\left(N\frac{\beta^{2}}{2}\xi_{q_{\left(\alpha_{1}\right)}}\left(1-q_{\left(\alpha_{1}\right)}\right)+o\left(N\right)\right),
\end{array}\label{eq: annealed UB}
\end{equation}
will be tight to leading order (for small $\beta$), where the $\hat{\sigma}$
on the RHS of the first line is an arbitrary $\hat{\sigma}\in\text{span}\left(m_{\left(\alpha_{1}\right)}\right)^{\bot}\cap S_{N-1}(\sqrt{1-\alpha_{1}^{2}})$
and
\begin{equation}
q_{\left(\alpha_{1}\right)}=\alpha_{1}^{2}=\|m_{\left(\alpha_{1}\right)}\|^{2}.\label{eq: q def}
\end{equation}

For the continuation of the construction is however more convenient
to define the equator not with respect to $h_{\text{eff}}$ as in
(\ref{eq: e tilde equator}) but with respect to a $\|\cdot\|$-unit
vector $u_{\left(\alpha_{1}\right),2}$ (different for each $\alpha_{1}$)
that is perpendicular to $u_{1}$ such that
\begin{equation}
\text{span}\left(u_{1},u_{\left(\alpha_{1}\right),2}\right)=\text{span}\left(u_{1},h_{\text{eff}}\right)=\text{span}\left(u_{1},\nabla H_{N}\left(m_{\left(\alpha_{1}\right)}\right)\right),\label{eq: span u1 u2}
\end{equation}
i.e.
\begin{equation}
\mathcal{E}_{\left(\alpha_{1}\right)}=\left\{ \sigma\in D_{\left(\alpha_{1}\right)}:\left|\left\langle \hat{\sigma},u_{\left(\alpha_{1}\right),2}\right\rangle \right|\le\eta\right\} .\label{eq: equator alpha1 def}
\end{equation}
For $\mathcal{E}_{\left(\alpha_{1}\right)}$ one also has as in (\ref{eq: UB on heff eq})
\begin{equation}
E^{\tilde{D}}\left[1_{\mathcal{E}_{\left(\alpha_{1}\right)}}\exp\left(\beta H_{N}^{m_{\left(\alpha_{1}\right)}}\left(\hat{\sigma}\right)+N\left\langle h_{\text{eff}},\hat{\sigma}\right\rangle \right)\right]\overset{\eqref{eq: equator alpha1 def}}{\le}E^{\tilde{D}}\left[\exp\left(\beta H_{N}^{m_{\left(\alpha_{1}\right)}}\left(\hat{\sigma}\right)\right)\right]e^{c\eta N},\label{eq: equator cont err}
\end{equation}
for a constant $c$ depending on $\beta$,  which with approximations
like (\ref{eq: D tilde}), (\ref{eq: approx}) and (\ref{eq: annealed UB})
with $\mathcal{E}_{\left(\alpha_{1}\right)}$ in place of $D_{\left(\alpha_{1}\right)}$
and a set $\tilde{\mathcal{E}}_{\left(\alpha_{1}\right)}=\left\{ \sigma:\left\langle \sigma,u_{1}\right\rangle =\alpha_{1},\left\langle \sigma,u_{\left(\alpha_{1}\right),2}\right\rangle =0\right\} $
in place of $\tilde{D}$ will be seen to give
\begin{equation}
\begin{array}{l}
Z_{N}\left(\mathcal{E}_{\left(\alpha_{1}\right)}\right)\le\exp\left(\beta H_{N}\left(m_{\left(\alpha_{1}\right)}\right)+N\beta h\alpha_{1}\right)\\
\quad\quad\quad\times E\left[\mathcal{E}_{\left(\alpha_{1}\right)}\right]\exp\left(N\frac{\beta^{2}}{2}\xi_{q_{\left(\alpha_{1}\right)}}\left(1-q_{\left(\alpha_{1}\right)}\right)+c\left(\varepsilon+\eta\right)N\right).
\end{array}\label{eq: ZN bound equator alpha1}
\end{equation}
By (\ref{eq: m alpha1}) and (\ref{eq: HNf sketch}) the first term
equals $\exp(\beta H_{N}^{f}(m_{(\alpha_{1})}))$. Turning to the
``entropy'' term $E[\mathcal{E}_{(\alpha_{1})}]$, note that $\mathcal{E}_{\left(\alpha_{1}\right)}$
is approximately a sphere of dimension $N-3$ of radius $\sqrt{1-\alpha_{1}^{2}}$
(the aforementioned $\tilde{\mathcal{E}}_{\left(\alpha_{1}\right)}$
is exactly this). Essentially since the surface area of spheres in
dimension $M$ scale like $r^{M}$, where $r$ is the radius, it turns
out that 
\begin{equation}
\begin{array}{ccccl}
E\left[\mathcal{E}_{\left(\alpha_{1}\right)}\right] & \le & \left(1-\alpha_{1}^{2}\right)^{\frac{N}{2}+o\left(N\right)} & \overset{\eqref{eq: q def}}{=} & \exp\left(\frac{N}{2}\log\left(1-\|m_{\left(\alpha_{1}\right)}\|^{2}\right)+o\left(N\right)\right)\\
 &  &  & \overset{\eqref{eq: spherical entropy I}}{=} & \exp\left(I_{{\rm sph}}\left(m_{\left(\alpha_{1}\right)}\right)+o\left(N\right)\right).
\end{array}\label{eq: equator entropy bound}
\end{equation}
Note finally that by (\ref{eq: Onsager term def}) and (\ref{eq: recent cover-1})
\begin{equation}
{\rm On}\left(q\right)=\xi_{q}\left(1-q\right)\text{\,for all }q\in\left[0,1\right],\label{eq: Onsager and recent xi}
\end{equation}
so the last factor on the right-hand side of (\ref{eq: ZN bound equator alpha1})
equals $\exp(\frac{\beta^{2}}{2}{\rm On}(\|m_{\left(\alpha_{1}\right)}\|^{2}))$.
In this fashion one can obtain from (\ref{eq: ZN bound equator alpha1})
that
\begin{equation}
\begin{array}{lcl}
Z_{N}\left(\mathcal{E}_{\left(\alpha_{1}\right)}\right) & \le & \exp\left(H_{N}^{f}\left(m_{\left(\alpha_{1}\right)}\right)+I_{{\rm sph}}\left(m_{\left(\alpha_{1}\right)}\right)+N\frac{\beta^{2}}{2}{\rm On}\left(\|m_{\left(\alpha_{1}\right)}\|^{2}\right)\right)e^{c\left(\eta+\varepsilon\right)N+o\left(N\right)}\\
 & \overset{\eqref{eq: HTAP spherical def}}{=} & \exp\left(H_{{\rm TAP}}^{{\rm sph}}\left(m_{\left(\alpha_{1}\right)}\right)+c\left(\eta+\varepsilon\right)N\right),
\end{array}\label{eq: equator tilde}
\end{equation}
which is a bound on the contribution of the equator $\mathcal{E}_{\left(\alpha_{1}\right)}$
inside $D_{\left(\alpha_{1}\right)}$ to the partition function in
terms of the TAP energy at $m=m_{\left(\alpha_{1}\right)}$. Note
that the entropy term $\frac{N}{2}\log(1-\|m\|^{2})$ of $H_{{\rm TAP}}^{{\rm sph}}$
arose from the measure of the equator $\mathcal{E}_{\left(\alpha_{1}\right)}$
under the reference measure in (\ref{eq: equator entropy bound}),
and the Onsager term $\frac{\beta^{2}}{2}\text{On}\left(\|m\|^{2}\right)$
from the annealed upper bound for the partition function of the recentered
Hamiltonian $H_{N}^{m_{\left(\alpha_{1}\right)}}$ on $\mathcal{E}_{\left(\alpha_{1}\right)}$
as in (\ref{eq: annealed UB}).

Combining (\ref{eq: equator bound}) and (\ref{eq: equator tilde})
one arrives at the bound
\begin{equation}
Z_{N}\left(\mathcal{E}\cup\left(\bigcup_{\alpha_{1}\in I_{\varepsilon,\eta}}\mathcal{E}_{\left(\alpha_{1}\right)}\right)\right)\le\left\{ \exp\left(H_{{\rm TAP}}^{{\rm sph}}\left(0\right)\right)+\sum_{\alpha_{1}\in I_{\varepsilon,\eta}}\exp\left(H_{{\rm TAP}}^{{\rm sph}}\left(m_{\left(\alpha_{1}\right)}\right)\right)\right\} e^{c\left(\eta+\varepsilon\right)N}.\label{eq: part func bount after one iter}
\end{equation}
 This is an improvement on (\ref{eq: equator bound}), but to obtain
the desired bound on $Z_{N}=Z_{N}\left(S_{N-1}\right)$ we still need
to bound the contribution of the complement of the region $\mathcal{E}\cup(\bigcup_{\alpha_{1}\in I_{\varepsilon,\eta}}\mathcal{E}_{(\alpha_{1})})$.

Similarly to how we decomposed $\mathcal{E}^{c}$ using sets $D_{\left(\alpha_{1}\right)}$
in (\ref{eq: dalpha1 def})-(\ref{eq: sphere decomp with dalpha1}),
this complement can be decomposed into sets 
\begin{equation}
D_{\left(\alpha_{1},\alpha_{2}\right)}=\left\{ \sigma\in D_{\left(\alpha_{1}\right)}:\left\langle \sigma,u_{\left(\alpha_{1}\right),2}\right\rangle \in(\left|\alpha_{2}\right|,\left|\alpha_{2}\right|+\varepsilon]\times\text{sign}\left(\alpha_{2}\right)\right\} ,\label{eq: Dalpha2 def}
\end{equation}
so that
\begin{equation}
S_{N-1}=\mathcal{E}\cup\left(\bigcup_{\alpha_{1}\in I_{\varepsilon,\eta}}\mathcal{E}_{\left(\alpha_{1}\right)}\right)\cup\left(\bigcup_{\alpha_{1},\alpha_{2}\in I_{\varepsilon,\eta}}D_{\left(\alpha_{1},\alpha_{2}\right)}\right).\label{eq: SN covered after two steps}
\end{equation}
Letting
\begin{equation}
m_{\left(\alpha_{1},\alpha_{2}\right)}=\alpha_{1}u_{1}+\alpha_{2}u_{\left(\alpha_{1}\right),2}=m_{\left(\alpha_{1}\right)}+\alpha_{2}u_{\left(\alpha_{1}\right),2},\label{eq: m alpha 2}
\end{equation}
(cf. (\ref{eq: m alpha1})) and using the change of variables $\sigma=m_{\left(\alpha_{1},\alpha_{2}\right)}+\hat{\sigma}$
one can decompose the Hamiltonian on $D_{\left(\alpha_{1},\alpha_{2}\right)}$
as
\begin{equation}
H_{N}\left(\sigma\right)=H_{N}\left(m_{\left(\alpha_{1},\alpha_{2}\right)}\right)+\nabla H_{N}\left(m_{\left(\alpha_{1},\alpha_{2}\right)}\right)\cdot\hat{\sigma}+H_{N}^{m_{\left(\alpha_{1},\alpha_{2}\right)}}\left(\hat{\sigma}\right),\label{eq: decomp alpha1 alpha2}
\end{equation}
(cf. (\ref{eq: change of var in slice}) and (\ref{eq: recent hamilt sketch})).
Here $H_{N}^{m_{\left(\alpha_{1},\alpha_{2}\right)}}\left(\hat{\sigma}\right),\hat{\sigma}\in\text{span}\left(u_{1},u_{\left(\alpha_{1}\right),2}\right)\cap B_{N},$
is a Hamiltonian with covariance function without terms of order $0$
or $1$ but the presence of the effective external field $\nabla H_{N}\left(m_{\left(\alpha_{1},\alpha_{2}\right)}\right)$
again means that an annealed upper bound on $Z_{N}\left(D_{\left(\alpha_{1},\alpha_{2}\right)}\right)$
will not be tight for all $\alpha_{2}$. To improve the situation
we may define $u_{\left(\alpha_{1},\alpha_{2}\right),3}$ as a unit
vector perpendicular to $u_{1},u_{\left(\alpha_{1}\right),2}$ such
that
\begin{equation}
\text{span}\left(u_{1},u_{\left(\alpha_{1}\right),2},u_{\left(\alpha_{1},\alpha_{2}\right),3}\right)=\text{span}\left(u_{1},u_{\left(\alpha_{1}\right),2},\nabla H_{N}\left(m_{\left(\alpha_{1},\alpha_{2}\right)}\right)\right),\label{eq: span u1 u2 u3}
\end{equation}
(cf. (\ref{eq: span u1 u2})) and define the equator
\begin{equation}
\mathcal{E}_{\left(\alpha_{1},\alpha_{2}\right)}=\left\{ \sigma\in D_{\left(\alpha_{1},\alpha_{2}\right)}:\left|\left\langle \sigma-m_{\left(\alpha_{1},\alpha_{2}\right)},u_{\left(\alpha_{1},\alpha_{2}\right),3}\right\rangle \right|\le\eta\right\} \label{eq: equator alpha 1 alpha 2 def}
\end{equation}
inside $D_{\left(\alpha_{1},\alpha_{2}\right)}$, similarly to (\ref{eq: equator alpha1 def}).
On $\mathcal{E}_{\left(\alpha_{1},\alpha_{2}\right)}$ one can bound
the effective external field term in (\ref{eq: decomp alpha1 alpha2})
by $Nc\eta$, and apply the same annealed bounds as above in (\ref{eq: annealed UB})
on each region $\mathcal{E}_{\left(\alpha_{1},\alpha_{2}\right)}$.
Since also $\mathcal{E}_{\left(\alpha_{1},\alpha_{2}\right)}$ is
essentially a lower dimensional sphere as in (\ref{eq: equator entropy bound})
\begin{equation}
E\left[\mathcal{E}_{\left(\alpha_{1},\alpha_{2}\right)}\right]\le\exp\left(\frac{N}{2}\log\left(1-\|m_{\left(\alpha_{1},\alpha_{2}\right)}\|^{2}\right)\right)e^{o\left(N\right)}=\exp\left(I_{{\rm sph}}\left(m_{\left(\alpha_{1},\alpha_{2}\right)}\right)\right)e^{o\left(N\right)}.\label{eq: equator alpha 2 entropy bound}
\end{equation}
In this way one can obtain
\begin{equation}
\begin{array}{ccl}
Z_{N}\left(\mathcal{E}_{\left(\alpha_{1},\alpha_{2}\right)}\right) & \le & \exp\left(H_{N}^{f}\left(m_{\left(\alpha_{1},\alpha_{2}\right)}\right)+I_{{\rm sph}}\left(m_{\left(\alpha_{1},\alpha_{2}\right)}\right)+N\frac{\beta^{2}}{2}{\rm On}\left(\|m_{\left(\alpha_{1},\alpha_{2}\right)}\|^{2}\right)\right)e^{c\left(\eta+\varepsilon\right)N}\\
 & \overset{\eqref{eq: HTAP spherical def}}{=} & \exp\left(H_{{\rm TAP}}^{{\rm sph}}\left(m_{\left(\alpha_{1},\alpha_{2}\right)}\right)+c\left(\eta+\varepsilon\right)N\right),
\end{array}\label{eq: ZN equator alpha 2 bound}
\end{equation}
(cf. (\ref{eq: equator tilde})), and from this one can improve on
the bound (\ref{eq: part func bount after one iter}) to get
\begin{equation}
\begin{array}{l}
Z_{N}\left(\mathcal{E}\cup\left({\displaystyle \bigcup_{\alpha_{1}\in I_{\varepsilon,\eta}}}\mathcal{E}_{\left(\alpha_{1}\right)}\right)\cup\left({\displaystyle \bigcup_{\alpha_{1},\alpha_{2}\in I_{\varepsilon,\eta}}}\mathcal{E}_{\left(\alpha_{1},\alpha_{2}\right)}\right)\right)\le e^{c\left(\eta+\varepsilon\right)N}\times\\
\quad\left\{ \exp\left(H_{{\rm TAP}}^{{\rm sph}}\left(0\right)\right)+{\displaystyle \sum_{\alpha_{1}\in I_{\varepsilon,\eta}}}\exp\left(H_{{\rm TAP}}^{{\rm sph}}\left(m_{\left(\alpha_{1}\right)}\right)\right)+{\displaystyle \sum_{\alpha_{1},\alpha_{2}\in I_{\varepsilon,\eta}}}\exp\left(H_{{\rm TAP}}^{{\rm sph}}\left(m_{\left(\alpha_{1},\alpha_{2}\right)}\right)\right)\right\} .
\end{array}\label{eq: bound after two iter}
\end{equation}
We can naturally continue this construction for any number $M$
of iterations, giving rise for each $k=1,\ldots,M$ and $\alpha\in I_{\varepsilon,\eta}^{k}$
to
\begin{itemize}
\item a direction $u_{\alpha,k+1}$, such that with the notation $u_{\alpha,l}=u_{\left(\alpha_{1},\ldots,\alpha_{l-1}\right),l}$
and $m_{\alpha,l}=m_{\left(\alpha_{1},\ldots,,\alpha_{l}\right)}$
we have that
\begin{equation}
m_{\alpha,k+1}=\alpha_{1}u_{1}+\alpha_{2}u_{\alpha,2}+\ldots+\alpha_{k+1}u_{\alpha,k+1}=m_{\alpha,k}+\alpha_{k+1}u_{\alpha,k+1},\label{eq: m alpha l increment}
\end{equation}
(cf. (\ref{eq: m alpha1}) and (\ref{eq: m alpha 2})) and
\begin{equation}
\begin{array}{c}
u_{\alpha,1},\ldots,u_{\alpha,k+1}\text{\,is an orthonormal basis of }\\
\text{span}\left(u_{1},\nabla H_{N}\left(m_{\alpha,1}\right),\ldots,\nabla H_{N}\left(m_{\alpha,k}\right)\right),
\end{array}\label{eq: basis}
\end{equation}
(cf. (\ref{eq: span u1 u2}) and (\ref{eq: span u1 u2 u3})),
\item a set
\begin{equation}
D_{\alpha}=\left\{ \sigma\in S_{N-1}:\left\langle \sigma,u_{\alpha,l}\right\rangle \in(\left|\alpha_{l}\right|,\left|\alpha_{l}\right|+\varepsilon]\times\text{sign}\left(\alpha_{l}\right),l=1,\ldots,k\right\} ,\label{eq: Dalpha def sketch}
\end{equation}
(cf. (\ref{eq: dalpha1 def}) and (\ref{eq: Dalpha2 def}))
\item and an equator (cf. (\ref{eq: first equator}) (\ref{eq: equator alpha1 def})
and (\ref{eq: equator alpha 1 alpha 2 def}))
\begin{equation}
\mathcal{E}_{\alpha}=\left\{ \sigma\in D_{\alpha}:\left|\left\langle \sigma-m_{\alpha},u_{\alpha,k+1}\right\rangle \right|\le\eta\right\} ,\label{eq: Ealpha def sketch}
\end{equation}
\end{itemize}
such that similarly to (\ref{eq: equator entropy bound}) and (\ref{eq: equator alpha 2 entropy bound})
\begin{equation}
E\left[\mathcal{E}_{\alpha}\right]\le\exp\left(\frac{N}{2}\log\left(1-\|m_{\alpha}\|^{2}\right)+o\left(N\right)\right)=\exp\left(I_{{\rm sph}}\left(m_{\alpha}\right)+o\left(N\right)\right),\label{eq: spherical entropy}
\end{equation}
(for $M$ growing slowly with $N$) and similarly to (\ref{eq: equator tilde})
and (\ref{eq: ZN equator alpha 2 bound})
\begin{equation}
Z_{N}\left(\mathcal{E}_{\alpha}\right)\le\exp\left(H_{{\rm TAP}}^{{\rm sph}}\left(m_{\alpha}\right)+c\left(M\varepsilon+\eta\right)N\right),\label{eq: energy of equator}
\end{equation}
(the formal definitions appear in Definitions \ref{def: space of incr},
\ref{def: def of m}, \ref{def: def of sets}). In (\ref{eq: energy of equator})
it is crucial that while the error involving $\varepsilon$ compounds
at most $M$ times since it comes from ``continuity'' errors as
in (\ref{eq: approx}) in $k\le M$ dimensions (see (\ref{eq: Dalpha def sketch})),
the error involving $\eta$ comes only from one dimension as in (\ref{eq: equator cont err})
(see (\ref{eq: Ealpha def sketch})) and is therefore not multiplied
by $M$ (see Lemma \ref{lem: to thin slice}).

Furthermore using the notations $m_{\left(\right)}=0$ and $\mathcal{E}_{\left(\right)}=\mathcal{E}$
one has 
\begin{equation}
S_{N-1}=\left(\cup_{k=0}^{M}\bigcup_{\alpha\in I_{\varepsilon,\eta}^{k}}\mathcal{E}_{\alpha}\right)\cup\left(\bigcup_{\alpha\in I_{\varepsilon,\eta}^{M+1}}D_{\alpha}\right),\label{eq: sphere covered with D M plus 1}
\end{equation}
(cf. (\ref{eq: sphere decomp with dalpha1}), (\ref{eq: SN covered after two steps}))
and as a simple consequence of (\ref{eq: energy of equator}) the
bound
\begin{equation}
Z_{N}\left(\cup_{k=0}^{M}\bigcup_{\alpha\in I_{\varepsilon,\eta}^{k}}\mathcal{E}_{\alpha}\right)\le e^{c\left(M\varepsilon+\eta\right)N}\sum_{k=0}^{M}\sum_{\alpha\in I_{\varepsilon,\eta}^{k}}\exp\left(H_{{\rm TAP}}^{{\rm sph}}\left(m_{\alpha}\right)\right),\label{eq: bound after K}
\end{equation}
for the contribution of the first set on the RHS of (\ref{eq: sphere covered with D M plus 1})
to the partition function. Seemingly, the problem remains that we
have no bound for the contribution $Z_{N}(\cup_{\alpha\in I_{\varepsilon,\eta}^{M+1}}D_{\alpha})$
of the second set on the RHS of (\ref{eq: sphere covered with D M plus 1}).
But we now argue that for large enough $M$, the sets $D_{\alpha},\alpha\in I_{\varepsilon,\eta}^{M+1}$
are in fact empty. This is because if $\alpha\in I_{\varepsilon,\eta}^{M+1}$
then in the construction we have the recursively defined basis from
(\ref{eq: basis}) such that if $\sigma\in D_{\alpha}$ then $\left|\left\langle \sigma,u_{\alpha,k}\right\rangle \right|\ge\alpha_{k}\ge\frac{\eta}{2}$
for $k=1,\ldots,M$ (recall (\ref{eq: Dalpha def sketch}) and (\ref{eq: I eps eta grid})).
This implies that $\|\sigma\|^{2}\ge M\frac{\eta^{2}}{4}$. Recall
that $D_{\alpha}\subset S_{N-1}$, so any $\sigma\in D_{\alpha}$
satisfies $\|\sigma\|^{2}=1$. Thus if $M=c\eta^{-2}$ for $c$ large
enough so that $M\frac{\eta^{2}}{4}>1$ then we have $D_{\alpha}=\emptyset$,
and in fact
\begin{equation}
S_{N-1}=\cup_{0\le k\le c\eta^{-2}}\bigcup_{\alpha\in I_{\varepsilon,\eta}^{k}}\mathcal{E}_{\alpha}.\label{eq: sketch incl}
\end{equation}
It will then follow from (\ref{eq: bound after K}) that
\begin{equation}
Z_{N}\left(S_{N-1}\right)\le e^{c\left(\eta^{-2}\varepsilon+\eta\right)N}\sum_{0\le k\le c\eta^{-2}}\sum_{\alpha\in I_{\varepsilon,\eta}^{k}}\exp\left(H_{{\rm TAP}}^{{\rm sph}}\left(m_{\alpha}\right)\right).\label{eq: ZN final bound as sum}
\end{equation}
As the number of summands is bounded in $N$ this will imply that
\begin{equation}
Z_{N}\left(S_{N-1}\right)\le\exp\left(\sup_{m}H_{{\rm TAP}}^{{\rm sph}}\left(m\right)+c\left(\eta^{-2}\varepsilon+\eta\right)N+o\left(N\right)\right).\label{eq: sketch final}
\end{equation}
Since first $\eta$ and then $\varepsilon$ can be made arbitrarily
small this is the desired bound, cf. (\ref{eq: spherical main result}).
This sketch can be formalized to yield a proof of Theorem \ref{thm: main_thm_spherical}
in the case of linear external field.

\subsection{\label{subsec: general spike}General spike term}

It is straight-forward to adapt the argument to a general Lipschitz
spike term $f_{N}\left(\sigma\right)=f_{N}\left(P^{U}\sigma\right)$
for a linear subspace $U$ of dimension $K$. One starts the iteration
with a set of $\|\cdot\|$-orthonormal initial directions $u_{1},\ldots,u_{K}$
whose span is $U$, rather than just one initial direction $u_{1}$,
and in the first step decomposes $S_{N-1}$ into sets 
\begin{equation}
D_{\left(\alpha_{1}\right)}=\left\{ \sigma\in S_{N-1}:\left\langle \sigma,u_{l}\right\rangle \in(\left|\alpha_{1,l}\right|,\left|\alpha_{1,l}\right|+\varepsilon]\times\text{sign}\left(\alpha_{l}\right),l=1,\ldots,K\right\} ,\label{eq: Dalpha1 spike}
\end{equation}
for $\alpha_{1}=\left(\alpha_{1,1},\ldots,\alpha_{1,K}\right)\in\left(-1,1\right)^{K}$
instead of (\ref{eq: dalpha1 def}), and defines 
\[
m_{\left(\alpha_{1}\right)}=\alpha_{1,1}u_{1}+\ldots+\alpha_{1,K}u_{K},
\]
instead of (\ref{eq: m alpha1}). Instead of (\ref{eq: Dalpha1})
one now has
\[
Z_{N}\left(D_{\left(\alpha_{1}\right)}\right)\le E\left[1_{D_{\left(\alpha_{1}\right)}}\exp\left(\beta H_{N}\left(\sigma\right)\right)\right]e^{\beta f_{N}\left(m_{\left(\alpha_{1}\right)}\right)+c\varepsilon KN},
\]
by the Lipschitz assumption on $f_{N}$. The subsequent directions
$u_{\left(\alpha_{1}\right),2},u_{\left(\alpha_{1}\right),3},\ldots$
are constructed one at a time just as in Section \ref{subsec: sketch sphere}
except that they are chosen orthogonal to all of $u_{1},\ldots,u_{K}$
and not just $u_{1}$, and the equators $\mathcal{E}_{\alpha}$ are
constructed by the same formula (\ref{eq: Ealpha def sketch}) for
$k\ge1$. In this way one obtains
\[
S_{N-1}=\cup_{1\le k\le c\eta^{-2}}\bigcup_{\alpha\in I_{\varepsilon,\eta}^{K}\times I_{\varepsilon,\eta}^{k-1}}\mathcal{E}_{\alpha},
\]
rather than (\ref{eq: sketch incl}) and 
\begin{equation}
Z_{N}\left(S_{N-1}\right)\le e^{c\left(\eta^{-2}\varepsilon+\eta\right)N}\sum_{1\le k\le c\eta^{-2}}\sum_{\alpha\in I_{\varepsilon,\eta}^{K}\times I_{\varepsilon,\eta}^{k-1}}\exp\left(H_{{\rm TAP}}^{{\rm sph}}\left(m_{\alpha}\right)\right),\label{eq: final with non-linear spike}
\end{equation}
rather than (\ref{eq: ZN final bound as sum}), which implies (\ref{eq: sketch final})
also for a general spike term. This sketch can be formalized to yield
a proof of Theorem \ref{thm: main_thm_spherical} for such a general
spike.

\subsection{\label{subsec: ising}Ising (or general) reference measure $E$}

In the sketch above the TAP energy for the spherical model arose essentially
because the entropy estimate (\ref{eq: equator entropy bound}) holds
for the spherical reference measure $E$. If $E$ is e.g. the Ising
reference measure (uniform measure on $\left\{ -1,1\right\} ^{N}$)
then to obtain an upper bound like (\ref{eq: equator tilde}) with
$H_{{\rm TAP}}^{{\rm Ising}}$ instead of $H_{{\rm TAP}}^{{\rm sph}}$
one should instead have that $E\left[\mathcal{E}_{\alpha}\right]$
is at most
\[
\exp\left(I_{{\rm Ising}}\left(m\right)\right)=\exp\left(-\sum_{i=1}^{N}J\left(m_{i}\right)+o\left(N\right)\right),
\]
where $J$ is the binary entropy from (\ref{eq: binary entropy def J}).
Such an upper bound however does not hold for sets $\mathcal{E}_{\alpha}$
as defined in e.g. (\ref{eq: equator alpha1 def}) of the sketch above:
a simple demonstration of this phenomenon is the fact that 
\[
E\left[\left\{ \sigma:\left\langle \sigma-m,m\right\rangle \approx0\right\} \right]\gg\exp\left(-\sum_{i=1}^{N}J\left(m_{i}\right)\right),
\]
for $E$ the uniform measure on $\left\{ -1,1\right\} ^{N}$, except
when $m_{1}=\ldots=m_{N}$. However, it can be shown (under appropriate
technical conditions) that there exists a unit vector $\lambda=\lambda_{m}$
such that
\begin{equation}
E\left[\left\{ \sigma:\left\langle \sigma-m,\lambda\right\rangle \approx0\right\} \right]\le\exp\left(-\sum_{i=1}^{N}J\left(m_{i}\right)+o\left(N\right)\right),\label{eq: halfplane}
\end{equation}
(see Lemma \ref{lem: ising entropy lemma}). This suggests the way
to extend the sketch in Sections \ref{subsec: sketch sphere}-\ref{subsec: general spike}
above to Ising reference measures: in each step of the iteration one
must decompose not only as in (\ref{eq: m alpha 2}) and (\ref{eq: m alpha l increment})
in one new direction $u_{\alpha,k+1}=u_{\alpha,k+1,1}$ whose span
with the previous vectors includes $\nabla H_{N}\left(m_{\alpha,k}\right)$;
rather one must include a second direction $u_{\alpha,k+1,2}$ whose
span with the others includes also $\lambda_{m_{\alpha}}$. Thus we
let each $\alpha_{k},k\ge2,$ be a pair of numbers in $I_{\varepsilon,\eta}^{2}$
that dictate an increment in $\text{span}(u_{\alpha,k,1},u_{\alpha,k,2})$
so that
\[
m_{\alpha,k+1}=m_{\alpha,k}+\alpha_{k,1}u_{\alpha,k+1,1}+\alpha_{k,2}u_{\alpha,k+1,2},
\]
and define the sets $D_{\alpha}$ for $\alpha\in I_{\varepsilon,\eta}^{K}\times(I_{\varepsilon,\eta}^{2})^{k-1}$
with respect to both the directions as
\[
D_{\alpha}=\left\{ \sigma\in D_{\left(\alpha_{1}\right)}:\left\langle \sigma,u_{\alpha,l,j}\right\rangle \in(\left|\alpha_{l,j}\right|,\left|\alpha_{l,j}\right|+\varepsilon]\times\text{sign}\left(\alpha_{l,j}\right),l=2,\ldots,k,j=1,2\right\} ,
\]
(for $D_{\left(\alpha_{1}\right)}$ as in (\ref{eq: Dalpha1 spike}))
and similarly for the the equators
\[
\mathcal{E}_{\alpha}=\left\{ \sigma\in D_{\alpha}:\left|\left\langle \sigma-m_{\alpha},u_{\alpha,k+1,j}\right\rangle \right|\le\eta\text{\,for }j=1,2\right\} 
\]
(cf. (\ref{eq: Ealpha def sketch})). The equator $\mathcal{E}_{\alpha}$
is then contained in $\left\{ \sigma:\left\langle \sigma-m_{\alpha},\lambda_{m_{\alpha}}\right\rangle \approx0\right\} $,
which means that by (\ref{eq: halfplane}) 
\[
E\left[\mathcal{E}_{\alpha}\right]\le\exp\left(-\sum_{i=1}^{N}J\left(\left(m_{\alpha}\right)_{i}\right)+o\left(N\right)\right)=\exp\left(I_{{\rm Ising}}\left(m_{\alpha}\right)+o\left(N\right)\right)
\]
instead of (\ref{eq: spherical entropy}), and then
\[
\begin{array}{ccl}
Z_{N}\left(\mathcal{E}_{\alpha}\right) & \le & \exp\left(\beta H_{N}^{f}\left(m_{\alpha}\right)+I_{{\rm Ising}}\left(m_{\alpha}\right)+\frac{\beta^{2}}{2}{\rm On}\left(\|m_{\alpha}\|^{2}\right)+c\left(M\varepsilon+\eta\right)N\right)\\
 & \overset{\eqref{eq: HTAP ising def}}{=} & \exp\left(H_{{\rm TAP}}^{\text{Ising}}\left(m_{\alpha}\right)+c\left(M\varepsilon+\eta\right)N\right),
\end{array}
\]
instead of (\ref{eq: energy of equator}). In this way the sketch
in Sections \ref{subsec: sketch sphere}-\ref{subsec: general spike}
above can be modified to yield a proof of the inequality
\[
Z_{N}\left(S_{N-1}\right)\le e^{c\left(\eta^{-2}\varepsilon+\eta\right)N}\sum_{1\le k\le c\eta^{-2}}\sum_{\alpha\in I_{\varepsilon,\eta}^{K}\times I_{\varepsilon,\eta}^{2\left(k-1\right)}}\exp\left(H_{{\rm TAP}}^{\text{Ising}}\left(m_{\alpha}\right)\right),
\]
instead of (\ref{eq: final with non-linear spike}). This then implies
(\ref{eq: sketch final}) with $H_{{\rm TAP}}^{\text{Ising}}$ instead
of $H_{{\rm TAP}}^{{\rm sph}}$. This sketch can be formalized into
a direct proof of Theorem \ref{thm: main_thm_ising}.

Alternatively - and this is the approach taken below - the sketch
can be turned into a formal proof of the TAP upper bound Theorem \ref{thm: main_thm_general}
for a general reference measure $E$ by defining as in (\ref{eq: IN})
the entropy term $I_{E,\delta}\left(m\right)$ of the general TAP
energy essentially as the LHS of (\ref{eq: halfplane}) minimized
over $\lambda$. From this general result the spherical and Ising
bounds can be derived by bounding $I_{E,\delta}$ above by $I_{{\rm sph}}$
and $I_{{\rm Ising}}$ respectively.

\section{\label{sec: law of recent}The recentered Hamiltonian and its law}

In this section we formally define the recentered Hamiltonian and
study its law.

Recall that $H_{N}\left(\sigma\right),\sigma\in B_{N}$ is a Gaussian
process with covariance given by (\ref{eq: covar}) for a power series
$\xi$ with non-negative coefficients and $\xi\left(1\right)<\infty$.
By Lemma \ref{lem: energy and first deriv covar} in the appendix
the process $H_{N}\left(\sigma\right),\sigma\in B_{N},$ exists and
is almost surely differentiable on $B_{N}^{o}$.

For any $m\in B_{N}$ and $\hat{\sigma}\in\mathbb{R}^{N}$ with $\|\hat{\sigma}\|^{2}\le1-\|m\|^{2}$
we define the recentered Hamiltonian by
\begin{equation}
H_{N}^{m}\left(\hat{\sigma}\right)=H_{N}\left(m+\hat{\sigma}\right)-\nabla H_{N}\left(m\right)\cdot\hat{\sigma}-H_{N}\left(m\right),\label{eq: recentered hamil}
\end{equation}
so that
\begin{equation}
H_{N}\left(m+\hat{\sigma}\right)=H_{N}\left(m\right)+\nabla H_{N}\left(m\right)\cdot\hat{\sigma}+H_{N}^{m}\left(\hat{\sigma}\right).\label{eq: decompostion}
\end{equation}
We furthermore define the effective external field
\begin{equation}
h_{{\rm eff}}\left(m\right)=\nabla H_{N}\left(m\right).\label{eq: heff def}
\end{equation}
Note that we then have
\begin{equation}
H_{N}^{f}\left(m+\hat{\sigma}\right)=H_{N}^{f}\left(m\right)+h_{{\rm eff}}\left(m\right)\cdot\hat{\sigma}+H_{N}^{m}\left(\hat{\sigma}\right)+\left(f_{N}\left(m+\hat{\sigma}\right)-f_{N}\left(m\right)\right),\label{eq: decomp with ext field}
\end{equation}
and since $f_{N}\left(\sigma\right)=f_{N}\left(P^{U}\sigma\right)$
we have for $\hat{\sigma}\in U^{\bot}$
\[
H_{N}^{f}\left(m+\hat{\sigma}\right)=H_{N}^{f}\left(m\right)+h_{{\rm eff}}\left(m\right)\cdot\hat{\sigma}+H_{N}^{m}\left(\hat{\sigma}\right).
\]

\begin{rem}
If $f_{N}$ is differentiable its recentering $f_{N}^{m}$ and the
recentering $(H_{N}^{f})^{m}$ of $H_{N}^{f}$ are well-defined. Then
an elegant alternative way to write this decomposition is to define
\[
h_{{\rm eff}}\left(m\right)=\nabla H_{N}^{f}\left(m\right)=\nabla H_{N}\left(m\right)+\nabla f_{N}\left(m\right),
\]
so that (\ref{eq: decomp with ext field}) can be replaced by 
\begin{equation}
H_{N}^{f}\left(m+\hat{\sigma}\right)=H_{N}^{f}\left(m\right)+h_{{\rm eff}}\left(m\right)\cdot\hat{\sigma}+\left(H_{N}^{f}\right)^{m}\left(\hat{\sigma}\right).\label{eq: decomp with f deriv}
\end{equation}
Ultimately in the proof we will apply (\ref{eq: decomp with ext field})
only when $P^{U}\hat{\sigma}$ is small enough to make $f_{N}\left(m+\hat{\sigma}\right)-f_{N}\left(m\right)$,
$f_{N}^{m}\left(\hat{\sigma}\right)$  and $\nabla f_{N}\left(m\right)\cdot\hat{\sigma}$
all negligible. As the benefit of writing (\ref{eq: decomp with f deriv})
is purely aesthetic, and it has the drawback of requiring the assumption
that $f_{N}$ is differentiable rather than Lipschitz we use (\ref{eq: decomp with ext field})
instead. \qed
\end{rem}

If $r>0$ and $\Sigma\subset B_{N}\left(r\right)$ we say that a centered
Gaussian process $\left(g\left(\sigma\right)\right)_{\sigma\in\Sigma}$
has covariance function $\xi:\left[-r^{2},r^{2}\right]\to[0,\infty)$
if
\begin{equation}
\mathbb{E}\left[g\left(\sigma\right)g\left(\sigma'\right)\right]=N\xi\left(\left\langle \sigma,\sigma'\right\rangle \right)\text{\,for all }\sigma,\sigma'\in\Sigma.\label{eq: g covar}
\end{equation}
If $V\subset\mathbb{R}^{N}$ is a linear space, $\Sigma=V\cap B_{N}\left(r\right)$
and $g$ is differentiable on $V\cap B_{N}^{\circ}\left(r\right)$
we define the recentering of $g^{m}$ (generalizing (\ref{eq: recentered hamil}))
by
\begin{equation}
g^{m}\left(\hat{\sigma}\right)=g\left(m+\hat{\sigma}\right)-\nabla g\left(m\right)\cdot\hat{\sigma}-g\left(m\right)\text{\,for }\hat{\sigma}\in\Sigma-m,\label{eq: recent general}
\end{equation}
where the gradient $\nabla g\left(m\right)$ is understood to be a
vector in $V$. If $\xi:\left[-r^{2},r^{2}\right]\to[0,\infty)$ is
a covariance function and $q\in\left[0,r^{2}\right]$ we define
\begin{equation}
\begin{array}{l}
\xi_{q}:\left[-\left(r^{2}-q\right),r^{2}-q\right]\to[0,\infty),\\
\xi_{q}\left(z\right)=\xi\left(q+z\right)-\xi'\left(q\right)z-\xi\left(q\right).
\end{array}\label{eq: xi recentered}
\end{equation}

The next lemma shows that the components of the decomposition in (\ref{eq: recent general})
are independent on the ``slice'' perpendicular to $m$, and gives
the law of the recentered process $g^{m}$ on the slice. It follows
by computing covariances of $g$ and its partial derivatives, and
the definition (\ref{eq: recent general}). Let $\partial_{r}$ denote
the radial derivative.
\begin{lem}
\label{lem: law of recentering}Let $N\ge1,r>0,V\subset\mathbb{R}^{N}$
be a linear space and let $\left(g\left(\sigma\right)\right)_{\sigma\in\Sigma}$
for $\Sigma=V\cap B_{N}\left(r\right)$ be a centered Gaussian field
with covariance function $\xi$, differentiable on $V\cap B_{N}^{\circ}\left(r\right)$.
For any $m\in\Sigma$ let 
\begin{equation}
\tilde{\Sigma}=\left\{ \hat{\sigma}\in\Sigma-m:\hat{\sigma}\cdot m=0\right\} =\tilde{V}\cap B_{N}\left(\sqrt{r^{2}-\|m\|^{2}}\right),\label{eq: sigma tilde}
\end{equation}
for the linear space $\tilde{V}=V\cap\left\{ \hat{\sigma}:\hat{\sigma}\cdot m=0\right\} $.
Then the three objects
\begin{equation}
\bigg(g\left(m\right),\partial_{r}g\left(m\right)\bigg),\quad P^{\text{span}\left(m\right)^{\bot}}\nabla g\left(m\right),\quad\left(g^{m}\left(\hat{\sigma}\right)\right)_{\hat{\sigma}\in\tilde{\Sigma}}\text{ are independent},\label{eq: three procs}
\end{equation}
and $\left(g^{m}\left(\hat{\sigma}\right)\right)_{\hat{\sigma}\in\tilde{\Sigma}}$
is a centered Gaussian process with covariance function $\xi_{\|m\|^{2}}$.
\end{lem}

\begin{proof}
The second equality in (\ref{eq: sigma tilde}) is elementary. Since
the covariance in (\ref{eq: g covar}) depends only on the inner product
we can by rotating $\mathbb{R}^{N}$ assume w.l.o.g. that $V$ is
the span of the first $\dim\left(V\right)$ basis vectors and $m$
is a multiple of $\left(1,\ldots,0\right)$. We then use the covariance
formulas for the partial derivatives of $g$ from Lemma \ref{lem: energy and first deriv covar}.
Since $m_{i}=0$ for $i\ge2$ one obtains from (\ref{eq: grad grad covar})
that $\partial_{i}g,i=2,\ldots,\dim\left(V\right)$ are IID with variance
$\xi'\left(\|m\|^{2}\right)$, and that $\partial_{r}g\left(m\right)=\partial_{1}g\left(m\right)$
is independent of $\partial_{i}g,i=2,\ldots,\dim\left(V\right)$.
Also for $\hat{\sigma}$ with $\hat{\sigma}\cdot m=0$ we obtain from
(\ref{eq: energy grad covar})
\begin{equation}
\mathbb{E}\left[g\left(m+\hat{\sigma}\right)\partial_{i}g\left(m\right)\right]=\left(\left|m\right|\delta_{i1}+\hat{\sigma}_{i}\right)\xi'\left(\|m\|^{2}\right)\text{ for }i=1,\ldots,\dim\left(V\right).\label{eq: h dih covar}
\end{equation}
This implies that $g\left(m\right)$ is independent of $P^{\text{span}\left(m\right)^{\bot}}\nabla g\left(m\right)$
by considering the case $\hat{\sigma}=0$. Thus we have shown that
$\left(g\left(m\right),\partial_{r}g\left(m\right)\right)$ and $P^{\text{span}\left(m\right)^{\bot}}\nabla g\left(m\right)$
are independent. Also any $\hat{\sigma}$ perpendicular to $m$
\[
\begin{array}{ccl}
\mathbb{E}\left[g^{m}\left(\hat{\sigma}\right)g\left(m\right)\right] & \overset{\eqref{eq: recent general}}{=} & \mathbb{E}\left[\left(g\left(m+\hat{\sigma}\right)-\nabla g\left(m\right)\cdot\hat{\sigma}-g\left(m\right)\right)g\left(m\right)\right]\\
 & \overset{\eqref{eq: h dih covar},\eqref{eq: g covar}}{=} & \xi\left(\|m\|^{2}\right)-\xi\left(\|m\|^{2}\right)=0,
\end{array}
\]
and similarly for $i=1,\ldots,\dim\left(V\right)$
\[
\mathbb{E}\left[g^{m}\left(\hat{\sigma}\right)\partial_{i}g\left(m\right)\right]\overset{\eqref{eq: recent general},\eqref{eq: h dih covar}}{=}\left(\left|m\right|\delta_{i1}+\hat{\sigma}_{i}\right)\xi'\left(\|m\|^{2}\right)-\hat{\sigma}_{i}\xi'\left(\|m\|^{2}\right)-\left|m\right|\delta_{i1}\xi'\left(\|m\|^{2}\right)=0.
\]
Thus the three objects in (\ref{eq: three procs}) are indeed independent.

Lastly one gets exploiting the independencies that for all $\hat{\sigma},\hat{\sigma}'$
perpendicular to $m$,
\[
\begin{array}{ccl}
\mathbb{E}\left[g^{m}\left(\hat{\sigma}\right)g^{m}\left(\hat{\sigma}'\right)\right] & \overset{\eqref{eq: recent general},\eqref{eq: three procs}}{=} & \mathbb{E}\left[g\left(m+\hat{\sigma}\right)g^{m}\left(\hat{\sigma}'\right)\right]\\
 & \overset{\eqref{eq: g covar},\eqref{eq: recent general},\eqref{eq: h dih covar}}{=} & N\xi\left(\|m\|^{2}+\left\langle \hat{\sigma},\hat{\sigma}'\right\rangle \right)-N\xi'\left(\|m\|^{2}\right)\left\langle \hat{\sigma},\hat{\sigma}'\right\rangle -N\xi\left(\|m\|^{2}\right)\\
 & \overset{\eqref{eq: xi recentered}}{=} & N\xi_{\|m\|^{2}}\left(\left\langle \hat{\sigma},\hat{\sigma}'\right\rangle \right).
\end{array}
\]
which shows that $\left(g^{m}\left(\hat{\sigma}\right)\right)_{\hat{\sigma}\in\tilde{\Sigma}}$
is a centered Gaussian process with covariance function $\xi_{\|m\|^{2}}$
as claimed.
\end{proof}
Later we will use the following natural relations for recentered processes
which follow directly from the definition (\ref{eq: recent general})
and hold for any $N\ge1,r>0$ and any process ${g:V\cap B_{N}\left(r\right)\to\mathbb{R}}$,
differentiable on $V\cap B_{N}^{\circ}\left(r\right)$, and any $a,b\in V\cap B_{N}^{\circ}\left(r\right),\hat{\sigma}\in V\cap B_{N}\left(r\right)$
such that $a+b\in V\cap B_{N}^{\circ}\left(r\right),a+b+\hat{\sigma}\in V\cap B_{N}\left(r\right)$:
\begin{equation}
\nabla g^{a}\left(b\right)=\nabla g\left(a+b\right)-\nabla g\left(a\right),\label{eq: recent grad}
\end{equation}
(where the gradient is understood to be a vector in $V$) and
\begin{equation}
\left(g^{a}\right)^{b}\overset{\eqref{eq: recent grad}}{=}g^{a+b}.\label{eq: recent ab a plus b}
\end{equation}
Similarly it follows from (\ref{eq: xi recentered}) that 
\begin{equation}
\left(\xi_{q}\right)_{q'}=\xi_{q+q'}\text{\,for all }\xi,q,q'.\label{eq: xi recent of recent}
\end{equation}

\section{\label{sec: magnetization construciton}Iterative construction of
magnetizations}

In this section we iteratively construct magnetizations in $B_{N}^{\circ}$
which will be used in the next section be used to construct a cover
of $S_{N-1}$.

Let 
\begin{equation}
I_{\varepsilon}=\varepsilon\mathbb{Z}\cap\left(-1,1\right),\label{eq: grid in set}
\end{equation}
be a grid of evenly spaced points at distance $\varepsilon$ covering
$\left(-1,1\right)$. For $\varepsilon>0$ we ``round down''
numbers in $\mathbb{R}$ to numbers in $\varepsilon\mathbb{Z}$ with
the operation
\begin{equation}
\lfloor x\rfloor_{\varepsilon}=\begin{cases}
\text{the }y\in\varepsilon\mathbb{Z}\text{ s.t. }x\in(y,y+\varepsilon] & \text{\,if }x>0,\\
0 & \text{ if }x=0,\\
\text{the }y\in\varepsilon\mathbb{Z}\text{ s.t. }x\in[y-\varepsilon,y) & \text{\,if }x<0.
\end{cases}\label{eq: rounding}
\end{equation}
Note that then
\begin{equation}
\left|x-\lfloor x\rfloor_{\varepsilon}\right|\le\varepsilon,\left|x\right|\ge\left|\lfloor x\rfloor_{\varepsilon}\right|\text{ for all }x\in\mathbb{R}\text{ and }\left|x\right|>\left|\lfloor x\rfloor_{\varepsilon}\right|\text{ for all }x\in\mathbb{R}\backslash\left\{ 0\right\} .\label{eq: rounding distance}
\end{equation}
Next we define the space of increments that will be used in the construction
of the cover. Recall that $K$ is the dimension of the linear space
$U\subset\mathbb{R}^{N}$.
\begin{defn}[Space of increments]
\label{def: space of incr}Define for each $1\le k\le\frac{N-K}{2}+1$
and $\varepsilon>0$
\begin{equation}
\mathcal{A}_{k}=\left\{ \alpha\in I_{\varepsilon}^{K}\times\left(I_{\varepsilon}^{2}\right)^{k-1}:\left|\alpha\right|<1\right\} ,\label{eq: def of Ak}
\end{equation}
(where $\left|\alpha\right|^{2}=\sum_{l=1}^{k}\left|\alpha_{k}\right|^{2}=\sum_{l=1}^{k}\sum_{j=1}^{A_{l}}\alpha_{k,j}^{2}$
for
\[
A_{1}=K\text{ and }A_{l}=2\text{ for }l\ge2),
\]
with the convention that $\left(I_{\varepsilon}^{2}\right)^{0}=\left\{ \left(\right)\right\} $,
where $\left(\right)$ is the sequence of length $0$. Also let
\[
\mathcal{A}=\bigcup_{1\le k\le\frac{N-K}{2}+1}\mathcal{A}_{k}.
\]
\qed
\end{defn}

In the decomposition of the sphere we will decompose in a direction
that is roughly speaking the normal of the hyperplane passing though
$m$ that has minimal measure under $E$ (cf. (\ref{eq: IN}), Section
\ref{subsec: ising}).We now define this normal $\lambda_{m}^{\delta}$
for any $m$.
\begin{defn}[Minimal entropy hyperplane around $m$]
\label{def: lambda}For any $m\in B_{N}$, $\lambda\in\mathbb{R}^{N}$
and $\delta>0$ let
\[
r_{\delta}\left(m,\lambda\right)=\log E\left[\left\langle \lambda,\sigma-m\right\rangle \ge-\delta\right],
\]
and 
\begin{equation}
\lambda_{m}^{\delta}\in\left\{ \lambda\in S_{N-1}:r_{\delta}\left(m,\lambda\right)\le\inf_{\tilde{\lambda}\in\mathbb{R}^{N}:\|\tilde{\lambda}\|=1}r_{\delta}\left(m,\tilde{\lambda}\right)+\delta\right\} ,\label{eq: lambda def}
\end{equation}
be picked according to some arbitrary $m$-measurable rule.\qed

Note that
\begin{equation}
\log E\left[\left\langle \lambda_{m}^{\delta},\sigma-m\right\rangle \ge-\delta\right]\le\inf_{\lambda:\|\lambda\|=1}r_{\delta}\left(m,\lambda\right)+\delta\overset{\eqref{eq: IN}}{\le}I_{E,\delta}\left(m\right)+\delta.\label{eq: lambda delta m and I E delta}
\end{equation}
\end{defn}

\begin{rem}
\label{rem: entropy min}For natural measures $E$ (like the uniform
measure on $\left\{ -1,1\right\} ^{N}$ or $S_{N-1}$) the infimum
in $\inf_{\tilde{\lambda}:\|\tilde{\lambda}\|=1}r_{\delta}(m,\tilde{\lambda})$
is achieved, and we could simply define $\lambda_{m}^{\delta}$ as
a minimizer. Furthermore for $E$ uniform on $S_{N-1}$ and $m\ne0$
the infimum is uniquely attained and in this case the parameter $\delta$
is unnecessary and one could define
\begin{equation}
\lambda_{m}^{\delta}=\lambda_{m}=\frac{m}{\|m\|}.\label{eq: spherical entropy min}
\end{equation}
In general however the infimum may not be achieved, and hence we define
$\lambda_{m}^{\delta}$ in (\ref{eq: lambda def}) as a vector that
almost achieves it.\qed
\end{rem}

The next definition gives an iterative construction of a magnetization
vector $m_{\alpha}$ from a vector $\alpha\in\mathcal{A}_{k}$ of
increment magnitudes lying in $I_{\varepsilon}$. It also constructs
an associated subspace $\overline{V}{}_{\alpha}$ of $\hat{\sigma}$
for which $\hat{\sigma}\cdot m_{\alpha}=0$. The crucial aspect of
the construction is that the next set of increments for the magnetization
covers the whole span of the gradient of the Hamiltonian and the normal
of the hyperplane of smallest entropy, both evaluated at the current
point of the iteration. Furthermore the increments are orthogonal
to the previous increments, and the subspace $\overline{V}{}_{\alpha}$
consists of vectors orthogonal to all previous increments and to the
gradient and minimal entropy hyperplane normal of $m_{\alpha}$.

Let
\begin{equation}
u_{1},\ldots,u_{K}\in\mathbb{R}^{N}\text{ be an }\left\langle \cdot,\cdot\right\rangle \text{-orthonormal basis of }U,\label{eq: U1 space}
\end{equation}
(chosen according to some arbitrary rule).
\begin{defn}[Magnetization $m_{\alpha}$ and basis $u_{\alpha,l,j}$ associated
to $\alpha\in\mathcal{A}_{k}$]
\label{def: def of m}$ $\\
Let $\varepsilon,\delta>0$, $1\le k\le\frac{N-K}{2}+1$ and $\alpha\in\mathcal{A}_{k}$
.  We define $\left(m_{\alpha,l}\right)_{1\le l\le k}$ and $\left(u_{\alpha,l}\right)_{1\le l\le k+1}$,
where $u_{\alpha,1}=\left(u_{\alpha,1,1},\ldots,u_{\alpha,1,K}\right)$
and $u_{\alpha,l}=\left(u_{\alpha,l,1},u_{\alpha,l,2}\right)$ for
$l\ge2$ and $u_{\alpha,l,j}\in\mathbb{R}^{N}$, via a recursion.
It starts with
\begin{equation}
u_{\alpha,1,j}=u_{j}\text{ for }j=1,\ldots,K,\label{eq: first us}
\end{equation}
and
\begin{equation}
U_{\alpha,1}=\text{span}\left(u_{\alpha,1,1},\ldots,u_{\alpha,1,K}\right)\text{ and }V_{\alpha,1}=U_{\alpha,1}^{\bot},\label{U alpha 1 def}
\end{equation}
and
\begin{equation}
m_{\alpha,1}=\sum_{j=1}^{K}\alpha_{1,j}u_{\alpha,1,j}.\label{eq: m0 u0 def}
\end{equation}

For the first step, consider the space
\begin{equation}
\text{span}\left(\nabla H_{N}\left(m_{\alpha,1}\right),\lambda_{m_{\alpha,1}}^{\delta}\right)\cap V_{\alpha,1}.\label{eq: space first step}
\end{equation}
Let $u_{\alpha,2,1},u_{\alpha,2,2}$ be an $\left\langle \cdot,\cdot\right\rangle $-orthonormal
basis of this space, or (if its dimension is less than two) of an
arbitrary two-dimensional subspace of $V_{\alpha,1}$ containing it,
and let
\begin{equation}
U_{\alpha,2}=\text{span}\left(U_{\alpha,1},u_{\alpha,2,1},u_{\alpha,2,2}\right)\text{ and }V_{\alpha,l}=U_{\alpha,l}^{\bot}.
\end{equation}
Then let
\begin{equation}
m_{\alpha,2}=m_{\alpha,1}+\alpha_{2,1}u_{\alpha,2,1}+\alpha_{2,2}u_{\alpha,2,2}.\label{eq: m def first step}
\end{equation}

For $3\le l\le k$ consider the space
\begin{equation}
\text{span}\left(\nabla H_{N}\left(m_{\alpha,l-1}\right),\lambda_{m_{\alpha,l-1}}^{\delta}\right)\cap V_{\alpha,l-1}.\label{eq: space-1}
\end{equation}
Let $u_{\alpha,l,1},u_{\alpha,l,2}$ be an $\left\langle \cdot,\cdot\right\rangle $-orthonormal
basis of a two-dimensional subspace of $V_{\alpha,l-1}$ containing
(\ref{eq: space-1}), and
\begin{equation}
U_{\alpha,l}=\text{span}\left(U_{\alpha,1},u_{\alpha,r,1},u_{\alpha,r,2}:2\le r\le l\right)\text{ and }V_{\alpha,l}=U_{\alpha,l}^{\bot}.\label{eq: U alpha l def}
\end{equation}
 Next, let
\begin{equation}
m_{\alpha,l}=m_{\alpha,l-1}+\alpha_{l,1}u_{\alpha,l,1}+\alpha_{l,2}u_{\alpha,l,2},\label{eq: m def}
\end{equation}
This recursively defines $\left(m_{\alpha,k}\right)_{1\le l\le k}$
and $\left(u_{\alpha,l}\right)_{1\le l\le k}$.

Define
\begin{equation}
m_{\alpha}=m_{\alpha,k},\label{eq: m alpha def}
\end{equation}
and
\begin{equation}
U_{\alpha}=U_{\alpha,k}\text{ and }V_{\alpha}=V_{\alpha,k}=U_{\alpha}^{\bot}.\label{eq: U alpha def}
\end{equation}
Finally define $u_{\alpha,k+1,1},u_{\alpha,k+1,2}$ as an $\left\langle \cdot,\cdot\right\rangle $-orthonormal
basis of a two-dimensional subspace of $V_{\alpha}$ containing 
\begin{equation}
\text{span}\left(\nabla H_{N}\left(m_{\alpha}\right),\lambda_{m_{\alpha}}^{\delta}\right)\cap V_{\alpha},\label{eq: space 2}
\end{equation}
and let $u_{\alpha,k+1}=\left(u_{\alpha,k+1,1},u_{\alpha,k+1,2}\right)$
and
\begin{equation}
\overline{U}_{\alpha}=U_{\alpha,k+1}=\text{span}\left(U_{\alpha},u_{\alpha,k+1,1},u_{\alpha,k+1,2}\right)\text{ and }\overline{V}_{\alpha}=V_{\alpha,k+1}=\overline{U}_{\alpha}^{\bot}.\label{eq: UV def}
\end{equation}
\qed
\end{defn}

\begin{rem}
$ $
\begin{enumerate}
\item For pedagogical reasons we carried out the first step of the iteration
in (\ref{eq: space first step})-(\ref{eq: m def first step}) explicitly,
but note that it's precisely the general step from (\ref{eq: space-1})-(\ref{eq: m def})
for $l=2$. Thus (\ref{eq: space-1})-(\ref{eq: m def}) in fact hold
for $2\le l\le k$, and not just $3\le l\le k$.
\item Whenever a basis of a space $\text{span}\left(a,b\right)$ for random
vectors $a,b$ is chosen above, it is understood to be chosen by a
rule which is measurable with respect to $a,b$.\qed
\end{enumerate}
\end{rem}

We collect here some simple consequences of this construction which
will be used in the proofs that follow. Recalling (\ref{eq: U1 space}),
(\ref{eq: space-1}) and the sentence below it, note that
\begin{equation}
u_{\alpha,l,j}\text{ for }1\le l\le k+1,1\le j\le A_{l},\text{ are }\left\langle \cdot,\cdot\right\rangle \text{-orthonormal for each }\alpha\in\mathcal{A}_{k},\label{eq: us orthonormal}
\end{equation}
and from also (\ref{eq: U alpha def})
\begin{equation}
U_{\alpha}=\text{span}\left(u_{\alpha,l,j}:1\le l\le k,1\le j\le A_{l}\right),\quad\quad V_{\alpha}=U_{\alpha}^{\bot}.\label{eq: Ualpha dep on u}
\end{equation}
We have for all $\alpha\in\mathcal{A}_{k}$ and $1\le l\le k$ that
\begin{equation}
m_{\alpha,l}\overset{\eqref{eq: m0 u0 def},\eqref{eq: m def}}{=}\sum_{r=1}^{l}\sum_{j=1}^{A_{l}}\alpha_{r,j}u_{\alpha,r,j}\overset{\eqref{eq: def of Ak},\eqref{eq: us orthonormal}}{\in}B_{N}^{\circ},\label{eq: m alpha l expansion}
\end{equation}
so that
\begin{equation}
m_{\alpha,l}\in U_{\alpha,l}=V_{\alpha,l}^{\bot}\text{ for }l=1,\ldots,k,\label{eq: m inc intermediate l}
\end{equation}
and in particular (recall (\ref{eq: m alpha def}))
\begin{equation}
m_{\alpha}=\sum_{l=1}^{k}\sum_{j=1}^{A_{l}}\alpha_{l,j}u_{\alpha,l,j}\in B_{N}^{\circ},\label{eq: m alpha expansion}
\end{equation}
and (recall (\ref{eq: U alpha def}))
\begin{equation}
m_{\alpha}\in U_{\alpha,k}=U_{\alpha}.\label{eq: m in U alpha k}
\end{equation}
Also
\begin{equation}
\nabla H_{N}\left(m_{\alpha,l}\right),\lambda_{m_{\alpha,l}}^{\delta}\in U_{\alpha,l+1}=V_{\alpha,l+1}^{\bot}\text{ for }l=1,\ldots,k,\label{eq: nabla H lambda m in U}
\end{equation}
(see (\ref{eq: space-1}) and below it, (\ref{eq: space 2}) and above
it and (\ref{eq: Ualpha dep on u})) and in particular 
\begin{equation}
h_{{\rm eff}}\left(m_{\alpha}\right)=\nabla H_{N}\left(m_{\alpha}\right),\lambda_{m_{\alpha}}^{\delta}\in\overline{U}{}_{\alpha}=\overline{V}{}_{\alpha}^{\bot},\label{eq: heff lambda m in U}
\end{equation}
where (see (\ref{eq: UV def}))
\begin{equation}
\begin{array}{ccccccc}
\overline{U}{}_{\alpha} & = & U_{\alpha,k+1} & = & \text{span}\left(u_{\alpha,l,j}:1\le l\le k+1,1\le j\le A_{l}\right) & \supset & U_{\alpha},\\
\overline{V}_{\alpha} & = & V_{\alpha,k+1} & = & U_{\alpha,k+1}^{\bot} & \subset & V_{\alpha},
\end{array}\label{eq: Ualpha k in terms of u}
\end{equation}
and the inclusions are strict. Also
\begin{equation}
\mathbb{R}^{N}=\overline{U}{}_{\alpha}\oplus\overline{V}{}_{\alpha}.\label{eq: orthog decomp overline}
\end{equation}
Furthermore, note from (\ref{eq: m def}) that for each $l\le k$
the vector $m_{\alpha,l}$ depends on $\alpha\in\mathcal{A}_{k}$
only through $\left(\alpha_{l}\right)_{1\le r\le l}$ and $u_{\alpha,l}$
depends on $\alpha$ only through $\left(\alpha_{l}\right)_{1\le r\le l-1}$.
Thus the bases are ``nested'' in the sense that for all $1\le l\le\left(k+1\right)\wedge\left(k'+1\right)$
\begin{equation}
\begin{array}{c}
\text{if }\alpha\in\mathcal{\mathcal{A}}_{k},\alpha'\in\mathcal{\mathcal{A}}_{k'}\text{ with }\alpha_{l}=\alpha_{l}^{'}\text{\,for }1\le r\le l-1,\\
\text{then }u_{\alpha,l,j}=u_{\alpha',l,j}\text{ for }1\le j\le A_{l}.
\end{array}\label{eq: nested}
\end{equation}
Lastly (from below (\ref{eq: space-1}) and above (\ref{eq: space 2}))
for $l=2,\ldots,k+1$
\begin{equation}
u_{\alpha,l}\text{ is measurable wrt. to }m_{\alpha,l-1}\text{\,and }P^{V_{\alpha,l-1}}\nabla H_{N}\left(m_{\alpha,l-1}\right).\label{eq: next u meas}
\end{equation}

\begin{rem}
\begin{enumerate}
\item Note that if we were treating only the case when $E$ is the uniform
measure on the sphere then we could replace $\lambda_{m}^{\delta}$
with $\lambda_{m}=\frac{m}{\left|m\right|}$ in the construction,
cf. (\ref{eq: spherical entropy min}). Since $m_{\alpha,l}\in U_{\alpha,l}$
the spaces (\ref{eq: space-1}), (\ref{eq: space 2}) then always
have dimension at most $1$. Thus in the spherical case the above
construction could be simplified by omitting $\alpha_{l,2}$ for all
$l\ge2$ and defining $\mathcal{A}_{k}$ as a subset of $I_{\varepsilon}^{K}\times I_{\varepsilon}^{k-1}$
rather than as in (\ref{eq: def of Ak}) (cf. the sketch in Sections
\ref{subsec: sketch sphere}, \ref{subsec: general spike}).
\item Though we find it more convenient not to use this fact, it is easy
to see that $\nabla H_{N}\left(m_{\alpha,l}\right)\in U_{\alpha,l}$
occurs with probability $0$ (except when $m_{\alpha,l}=0$), so that
with probability one the spaces (\ref{eq: space-1}), (\ref{eq: space 2})
have dimension at least $1$.\qed
\end{enumerate}
\end{rem}

Define
\begin{equation}
q_{\alpha}=\|m_{\alpha}\|^{2}\overset{\eqref{eq: us orthonormal},\eqref{eq: m alpha expansion}}{=}\left|\alpha\right|^{2}.\label{eq: q alpha def}
\end{equation}
Another important feature of the construction is that conditionally
on $u_{\alpha,l},l=1,\ldots,{k+1}$, the law of $H_{N}^{m_{\alpha}}\left(\hat{\sigma}\right)$
restricted to $\overline{V}_{\alpha}$ is that of a centered Gaussian
process with covariance function $\xi_{q_{\alpha}}$. For $\alpha\in\mathcal{A}_{1}$
this follows from Lemma \ref{lem: law of recentering} with $g=H_{N}$
and $m=m_{\alpha}=m_{\alpha,1}$. For $\alpha\in\mathcal{A}_{k},k\ge2,$
it essentially follows by repeatedly applying Lemma \ref{lem: law of recentering},
as we now show. Let
\begin{equation}
\mathcal{R}_{\alpha}=\sigma\left(P^{V_{\alpha,l}}\nabla H_{N}\left(m_{\alpha,l}\right),l=1,\ldots,k;u_{l},l=2,\ldots,k+1\right).\label{eq: R alpha def}
\end{equation}
Note that under $\mathbb{P}\left(\cdot|\mathcal{R}_{\alpha}\right)$
the objects $\overline{V}_{\alpha}$, $m_{\alpha}$ are deterministic
(recall (\ref{eq: m alpha expansion}), (\ref{eq: Ualpha k in terms of u})).
\begin{lem}
\label{lem: law of recentering alpha}For $1\le k\le\frac{N-K}{2}+1$
and $\alpha\in\mathcal{A}_{k}$ almost surely the $\mathbb{P}\left(\cdot|\mathcal{R}_{\alpha}\right)$-law
of $\left(H_{N}^{m_{\alpha}}\left(\hat{\sigma}\right)\right)_{\hat{\sigma}\in\overline{V}_{\alpha}\cap B_{N}\left(\sqrt{1-q_{\alpha}}\right)}$
is that of a centered Gaussian process with covariance function $\xi_{q_{\alpha}}$.
\end{lem}

\begin{proof}
Let for $l=0,\ldots,k$,
\[
\mathcal{R}_{\alpha,l}=\sigma\left(P^{V_{\alpha,r}}\nabla H_{N}\left(m_{\alpha,r}\right),r=1,\ldots,l;u_{r},r=2,\ldots,l+1\right),
\]
and
\[
\Sigma_{l}=V_{\alpha,l+1}\cap B_{N}\left(\sqrt{1-q_{\alpha,l}}\right)\text{ where }q_{\alpha,l}=\|m_{\alpha,l}\|^{2},
\]
with the convention $m_{\alpha,0}=0$ and $q_{\alpha,0}=0$. We prove
by induction that for $l=0,\ldots,k$
\begin{equation}
\text{the }\mathbb{P}\left(\cdot|\mathcal{R}_{\alpha,l}\right)\text{-law of }H_{N}^{m_{\alpha,l}}\text{\,is a CGPD\ensuremath{\left(\Sigma_{l},\xi_{q_{\alpha,l}},q_{\alpha,l}\right)}}\label{eq: induction step}
\end{equation}
where ``CGPD$(\Sigma,\xi,q)$'' is a shorthand for ``centered Gaussian
process with index set $\Sigma\subset B_{N}\left(\sqrt{1-q}\right)$
and covariance function $\xi$, differentiable on $\Sigma\cap B_{N}^{\circ}\left(\sqrt{1-q}\right)$''.
The case $l=k$ implies the claim of the lemma.

The case $l=0$ of (\ref{eq: induction step}) holds trivially (note
that $H_{N}^{0}\left(\sigma\right)=H_{N}\left(\sigma\right)$ for
all $\sigma$, that $\mathcal{R}_{\alpha,0}$ is the trivial $\sigma$-algebra
and recall that $V_{\alpha,1}$ is deterministic).

To prove the induction step we will use that for $s=0,\ldots,k-1$
\begin{equation}
\nabla H_{N}\left(m_{\alpha,s+1}\right)-\nabla H_{N}\left(m_{\alpha,s}\right)=\nabla H_{N}^{m_{\alpha,s}}\left(m_{\alpha,s+1}-m_{\alpha,s}\right),\label{eq: deriv}
\end{equation}
which follows from (\ref{eq: recent grad})
\begin{equation}
H_{N}^{m_{\alpha,s+1}}\left(\sigma\right)=\left(H_{N}^{m_{\alpha,s}}\right)^{m_{\alpha,s+1}-m_{\alpha,s}}\left(\sigma\right)\text{ for all }\sigma\in B_{N}\left(\sqrt{1-q_{\alpha,s+1}}\right),\label{eq: process}
\end{equation}
which follows from (\ref{eq: recent ab a plus b}), both with $V=\mathbb{R}^{N},g=H_{N},a=m_{\alpha,s},b=m_{\alpha,s+1}-m_{\alpha,s}$.

Now assume that (\ref{eq: induction step}) holds for $l=s$. Note
that $m_{\alpha,s},u_{\alpha,s+1},m_{\alpha,s+1}V_{\alpha,s+1},\Sigma_{s}$
are deterministic under $\mathbb{P}(\cdot|\mathcal{R}_{\alpha,s})$.
Consider the process $g\left(\hat{\sigma}\right)=H_{N}^{m_{\alpha,s}}\left(\hat{\sigma}\right)$
for $\hat{\sigma}\in\Sigma_{l}$ under $\mathbb{P}(\cdot|\mathcal{R}_{\alpha,s})$.
By applying Lemma \ref{lem: law of recentering} to it with $r=\sqrt{1-q_{\alpha,s}}$,
$V=V_{\alpha,s+1}$ and
\[
m=m_{\alpha,s+1}-m_{\alpha,s}\overset{\eqref{eq: m def}}{\in}\text{span}\left(u_{\alpha,s+1,1},u_{\alpha,s+1,2}\right),
\]
 under $\mathbb{P}\left(\cdot|\mathcal{R}_{\alpha,s}\right)$ and
using that
\[
q_{\alpha,s}+\|m\|^{2}\overset{\eqref{eq: us orthonormal},\eqref{eq: m alpha l expansion}}{=}q_{\alpha,s+1}\text{ and }\left(\xi_{q_{\alpha,s}}\right)_{\|m\|^{2}}\overset{\eqref{eq: xi recent of recent}}{=}\xi_{q_{\alpha,s}+\|m\|^{2}}=\xi_{q_{\alpha,s+1}},
\]
we obtain that the process $(g^{m}\left(\hat{\sigma}\right),\hat{\sigma}\in\tilde{\Sigma})$
for 
\[
\tilde{\Sigma}=V_{\alpha,s+1}\cap\left\{ \hat{\sigma}:\hat{\sigma}\cdot m=0\right\} \cap B_{N}\left(\sqrt{1-q_{\alpha,s+1}}\right),
\]
is independent of $\nabla g\left(m\right)=P^{V_{\alpha,s+1}}\nabla H_{N}^{m_{\alpha,s}}\left(m\right)$
and is a CGPD$(\xi_{q_{\alpha,s+1}},\tilde{\Sigma},q_{\alpha,s+1})$.
By (\ref{eq: deriv}) we have that $\nabla g\left(m\right)$ and $P^{V_{\alpha,s+1}}\nabla H_{N}\left(m_{\alpha,s+1}\right)$
are deterministically related under $\mathbb{P}\left(\cdot|\mathcal{R}_{\alpha,s}\right)$,
so under this measure $(g^{m}\left(\hat{\sigma}\right),\hat{\sigma}\in\tilde{\Sigma})$
and $P^{V_{\alpha,s+1}}\nabla H_{N}\left(m_{\alpha,s+1}\right)$ are
independent. By (\ref{eq: next u meas}) the process $(g^{m}\left(\hat{\sigma}\right),\hat{\sigma}\in\tilde{\Sigma})$
is then independent of $\left(P^{V_{\alpha,s+1}}\nabla H_{N}\left(m_{\alpha,s+1}\right),u_{\alpha,s+2}\right)$
under $\mathbb{P}\left(\cdot|\mathcal{R}_{\alpha,s}\right)$. Thus
$(g^{m}\left(\hat{\sigma}\right),\hat{\sigma}\in\tilde{\Sigma})$
is a CGPD$(\xi_{q_{\alpha,s+1}},\tilde{\Sigma},q_{\alpha,s+1})$
also under $\mathbb{P}\left(\cdot|\mathcal{R}_{s+1}\right)$. Next
under this measure $u_{\alpha,s+2},V_{\alpha,s+2},\Sigma_{s+1}$ are
deterministic, and $\tilde{\Sigma}\subset V_{\alpha,s+1}\cap\{\hat{\sigma}:\hat{\sigma}\cdot u_{\alpha,s+1,j}=0,j=1,2\}\cap B_{N}(\sqrt{1-q_{\alpha,s+1}})=\Sigma_{s+1}$
(recall (\ref{eq: U alpha l def}) and thus $(g^{m}\left(\hat{\sigma}\right),\hat{\sigma}\in\Sigma_{s+1})$
is a CGPD$(\xi_{q_{\alpha,s+1}},\Sigma_{s+1},q_{\alpha,s+1})$  under
$\mathbb{P}\left(\cdot|\mathcal{R}_{s+1}\right)$. Finally by (\ref{eq: process})
this in fact means that $(H_{N}^{m_{\alpha,s+1}}\left(\hat{\sigma}\right),\hat{\sigma}\in\Sigma_{s+1})$
is CGPD$(\xi_{q_{\alpha,s+1}},\Sigma_{s+1},q_{\alpha,s+1})$  under
$\mathbb{P}\left(\cdot|\mathcal{R}_{s+1}\right)$. Thus (\ref{eq: induction step})
holds for $l=s+1$. This completes the proof of (\ref{eq: induction step})
by induction.
\end{proof}

\section{\label{sec: construction of cover}Construction of a cover of $S_{N-1}$}

In this section we construct ``equator'' sets $\mathcal{E}_{\alpha}$
that form a cover of $S_{N-1}$, each associated to a magnetization
vector $m_{\alpha}$ from the previous section. A crucial part of
the definition is the condition $\left|\left\langle \sigma,u_{\alpha,k+1,j}\right\rangle \right|\le\eta$,
which makes the effective external field on $\mathcal{E}_{\alpha}$
almost vanish, as will be proven in Proposition \ref{prop: effective external field B-1}
below, and also makes the entropy of the sets $\mathcal{E}_{\alpha}$
bounded by the entropy function $I_{E,\delta}$, as proven in Proposition
\ref{prop: entropy of Bs} below.
\begin{defn}
\label{def: def of sets}(Regions $\mathcal{E}_{\alpha}\subset D_{\alpha}$
of $S_{N-1}$ associated to $m_{\alpha}$) Fix $\varepsilon>0$ and
$\eta>0$. For each $1\le k\le\frac{N-K}{2}+1$ and $\alpha\in\mathcal{A}_{k}$
define the (random) sets
\begin{equation}
D_{\alpha}=\left\{ \sigma\in S_{N-1}:\lfloor\left\langle \sigma,u_{\alpha,l,j}\right\rangle \rfloor_{\varepsilon}=\alpha_{l,j}\text{\,for }1\le l\le k,1\le j\le A_{l}\right\} ,\label{eq: Dalpha def}
\end{equation}
and
\begin{equation}
\mathcal{E}_{\alpha}=\left\{ \sigma\in D_{\alpha}:\left|\left\langle \sigma,u_{\alpha,k+1,j}\right\rangle \right|\le\eta\text{ for }j=1,2\right\} .\label{eq: Ealpha def}
\end{equation}
\qed
\end{defn}

The next definition essentially constructs the cover, by specifying
the vectors $\alpha$ of the sets $\mathcal{E}_{\alpha}$ that should
be included.
\begin{defn}
(Index set $\mathcal{A}_{\varepsilon,\eta}$ of cover of $S_{N-1}$)
For any $0<\varepsilon,\eta<1$ let
\begin{equation}
\mathcal{A}_{\varepsilon,\eta}=\bigcup_{1\le k\le5\eta^{-2}}\mathcal{A}_{k}.\label{eq: Aeps  ddef}
\end{equation}
\qed
\end{defn}

Note from (\ref{eq: def of Ak}) and the fact that $\left|I_{\varepsilon}\right|\le2\varepsilon^{-1}$
(recall (\ref{eq: grid in set})) that $\mathcal{A}_{k}$ are finite
sets of cardinality bounded in $N$, and for $\varepsilon,\eta\in\left(0,1\right)$,
\begin{equation}
\left|\mathcal{A}_{k}\right|\le\left(2\varepsilon^{-1}\right)^{K+2\left(k-1\right)}\text{\,and }\left|\mathcal{A}_{\varepsilon,\eta}\right|\le\left(2\varepsilon^{-1}\right)^{K+10\eta^{-2}}.\label{eq: size of cover}
\end{equation}

We now show that $\left(\mathcal{E}_{\alpha}\right)_{\alpha\in\mathcal{A}_{\varepsilon,\eta}}$
is indeed a cover for $S_{N-1}$.
\begin{prop}
($\left(\mathcal{E}_{\alpha}\right)_{\alpha\in\mathcal{A}_{\varepsilon,\eta}}$
is cover of $S_{N-1}$) It holds for $0<\eta<1,0<\varepsilon\le\eta/2$
and all $\delta>0$ that \label{prop: cover}
\begin{equation}
S_{N-1}=\bigcup_{\alpha\in\mathcal{A}_{\varepsilon,\eta}}\mathcal{E}_{\alpha}\text{ almost surely.}\label{eq: cover inclusion}
\end{equation}
\end{prop}

\begin{proof}
Let $\sigma\in S_{N-1}$. Define $\alpha_{1}\in I_{\varepsilon}^{K}$
by
\begin{equation}
\alpha_{1,1}=\lfloor\left\langle \sigma,u_{1}\right\rangle \rfloor_{\varepsilon},\ldots,\alpha_{1,K}=\lfloor\left\langle \sigma,u_{K}\right\rangle \rfloor_{\varepsilon},\label{eq: first step}
\end{equation}
and define the random sequence $\alpha_{l},2\le l\le\frac{N-K}{2}+1,$
recursively via
\begin{equation}
\alpha_{l,1}=\lfloor\left\langle \sigma,u_{\left(\alpha_{1},\ldots,\alpha_{l-1}\right),l,1}\right\rangle \rfloor_{\varepsilon}\text{ and }\alpha_{l,2}=\lfloor\left\langle \sigma,u_{\left(\alpha_{1},\ldots,\alpha_{l-1}\right),l,2}\right\rangle \rfloor_{\varepsilon},\label{eq: alpha from sigma}
\end{equation}
where the $u_{\left(\alpha_{1},\ldots,\alpha_{l-1}\right),l}$ are
constructed in Definition \ref{def: def of m}. Let 
\begin{equation}
k\text{ be the smallest positive integer such that }\left|\alpha_{k+1,1}\right|,\left|\alpha_{k+1,2}\right|\le\frac{\eta}{2},\label{eq: k smallest int}
\end{equation}
and let $\alpha=\left(\alpha_{1},\ldots,\alpha_{k}\right)$. By the
``nesting'' property (\ref{eq: nested}) of the basis $u_{\alpha,l,j}$
we have
\[
u_{\left(\alpha_{1},\ldots,\alpha_{l-1}\right),l,j}=u_{\alpha,l,j}\text{\,for all }2\le l\le k+1,1\le j\le2,
\]
which together with (\ref{eq: first us}) means that (\ref{eq: first step})-(\ref{eq: alpha from sigma})
imply
\begin{equation}
\lfloor\left\langle \sigma,u_{\alpha,l,j}\right\rangle \rfloor_{\varepsilon}=\alpha_{l,j}\text{\,for all }1\le l\le k+1,1\le j\le A_{l}.\label{eq: sigma rounded down}
\end{equation}
Thus $\alpha_{l,j}\in I_{\varepsilon}$ for all $l,j$ and since the
$u_{\alpha,l,j}$ are orthonormal (recall (\ref{eq: us orthonormal}))
\begin{equation}
\sum_{l=1}^{k}\sum_{j=1}^{A_{l}}\alpha_{l,j}^{2}\overset{\eqref{eq: sigma rounded down},\eqref{eq: rounding distance}}{<}\sum_{l=1}^{k}\sum_{j=1}^{A_{l}}\left\langle \sigma,u_{\alpha,l,j}\right\rangle ^{2}\le\|\sigma\|^{2}=1.\label{eq: alphas 2 norm}
\end{equation}
This implies that $\alpha\in\mathcal{A}_{k}$ (see its definition
(\ref{eq: def of Ak})), and then (\ref{eq: sigma rounded down})
implies that $\sigma\in D_{\alpha}$ (see its definition (\ref{eq: Dalpha def})).

Furthermore, since $\left|\alpha_{k+1,1}\right|,\left|\alpha_{k+1,2}\right|\le\eta$
the equality (\ref{eq: sigma rounded down}) implies that $\left|\left\langle \sigma,u_{\alpha,k+1}\right\rangle \right|\le\eta/2+\varepsilon\le\eta$
so that $\sigma\in\mathcal{E}_{\alpha}$ by (\ref{eq: Ealpha def}).

We now show that also $\alpha\in\mathcal{A}_{\varepsilon,\eta}$,
thus completing the proof of (\ref{eq: cover inclusion}). The construction
of $\alpha$ (recall (\ref{eq: k smallest int})) implies that
\begin{equation}
\left|\alpha_{l,1}\right|>\frac{\eta}{2}\text{ or }\left|\alpha_{l,2}\right|>\frac{\eta}{2}\text{\, for }l=2,\ldots,k.\label{eq: al not close to zero}
\end{equation}
Thus we have 
\[
\left(k-1\right)\frac{\eta^{2}}{4}\overset{\eqref{eq: al not close to zero}}{\le}\sum_{l=2}^{k}\sum_{j=1}^{A_{l}}\alpha_{l,1}^{2}\overset{\eqref{eq: alphas 2 norm}}{\le}1,
\]
showing that
\begin{equation}
k\le1+4\eta^{-2}\le5\eta^{-2},\label{eq: length of al}
\end{equation}
and thus that $\alpha\in\mathcal{A}_{\varepsilon,\eta}$ (see (\ref{eq: Aeps  ddef})).
\end{proof}
\begin{rem}
The proposition could be somewhat strengthened, since one can prove
that it also holds with
\[
\left\{ \alpha\in\mathcal{A}_{\varepsilon,\eta}:\alpha_{l}\notin\left[-\frac{\eta}{2},\frac{\eta}{2}\right]^{2}\text{ for }2\le l\le5\eta^{-2}\right\} ,
\]
in place of $\mathcal{A}_{\varepsilon,\eta}$. Indeed, for instance
for any $\alpha\in\mathcal{A}_{1}$ the set $\mathcal{E}_{\alpha}$
is itself contained in $\cup_{\alpha'\in\mathcal{A}_{k}}\mathcal{E}_{\alpha'}$
for any $k\ge2$, so the spin configurations in $\mathcal{E}_{\alpha}$
are contained in the RHS of (\ref{eq: cover inclusion}) many times
over. In the sketch of Section \ref{subsec: sketch sphere} this double-counting
is avoided (see (\ref{eq: I eps eta grid}), (\ref{eq: sketch incl})),
but in the proof we do not bother with this as it would slightly complicate
the definition of $\mathcal{A}_{\varepsilon,\eta}$ and does not otherwise
simplify the proof or strengthen the result. \qed
\end{rem}

\section{\label{sec: proof of gen UB}Proof of general TAP upper bound}

In this section we complete the proof of the general TAP upper bound
Theorem \ref{thm: main_thm_general} using the construction from the
previous two sections.

Recall that $E$ is the measure in the statement of Theorem \ref{thm: main_thm_general}.
Proposition \ref{prop: cover} implies
\begin{equation}
E\left[\exp\left(\beta H_{N}^{f}\left(\sigma\right)\right)\right]\le\sum_{\alpha\in\mathcal{A}_{\varepsilon,\eta}}E\left[1_{\mathcal{E}_{\alpha}}\exp\left(\beta H_{N}^{f}\left(\sigma\right)\right)\right],\label{eq: foreshadow}
\end{equation}
which we will use in the proof of Theorem \ref{thm: main_thm_general}.
Define for all $\alpha\in\mathcal{A}$ such that $E\left[\mathcal{E}_{\alpha}\right]>0$
the (random) measures
\begin{equation}
E^{m_{\alpha}}\left[A\right]=\frac{E\left[A\cap\mathcal{E}_{\alpha}\right]}{E\left[\mathcal{E}_{\alpha}\right]},\text{ }A\subset S_{N-1}\text{ measurable},\label{eq: E equator def}
\end{equation}
so that (\ref{eq: foreshadow}) can be written as
\begin{equation}
E\left[\exp\left(\beta H_{N}^{f}\left(\sigma\right)\right)\right]\le\sum_{\alpha\in\mathcal{A}_{\varepsilon,\eta}:E\left[\mathcal{E}_{\alpha}\right]>0}E\left[\mathcal{E}_{\alpha}\right]E^{m_{\alpha}}\left[\exp\left(\beta H_{N}^{f}\left(\sigma\right)\right)\right].\label{eq: foreshadow 2}
\end{equation}
In what follows we will give bounds for $E[\mathcal{E}_{\alpha}]$
and $E^{m_{\alpha}}[\exp(\beta H_{N}^{f}(\sigma))]$. We first consider
$E[\mathcal{E}_{\alpha}]$. To bound it we use the next lemma, which
shows that if $\sigma\in\mathcal{E}_{\alpha}$ then the increment
$\hat{\sigma}=\sigma-m_{\alpha}$ is almost orthogonal to $\overline{U}{}_{\alpha}$,
i.e. almost lies in $\overline{V}{}_{\alpha}$. It will also be used
to show that the effective external field on $\mathcal{E}_{\alpha}$
almost vanishes in Proposition \ref{prop: effective external field B-1}.
\begin{lem}
\label{lem: Dalpha is thick slice}For any $0<\eta\le K^{-1/2}$ and
$0<\varepsilon\le\eta^{2}$ it holds almost surely for all $\alpha\in\mathcal{A}_{\varepsilon,\eta}$
that 
\begin{equation}
\sigma\in\mathcal{E}_{\alpha}\implies\|P^{\overline{U}{}_{\alpha}}\hat{\sigma}\|=\|P^{\overline{U}{}_{\alpha}}\sigma-m_{\alpha}\|=\|P^{\overline{V}{}_{\alpha}}\hat{\sigma}-\hat{\sigma}\|=\|P^{\overline{V}{}_{\alpha}}\sigma-\hat{\sigma}\|\le4\eta,\label{eq: slice distance}
\end{equation}
where $\hat{\sigma}=\sigma-m_{\alpha}$.
\end{lem}

\begin{proof}
The equality of the first four expressions is elementary and holds
for any $\sigma\in\mathbb{R}^{N}$ since $P^{\overline{V}{}_{\alpha}}\hat{\sigma}+P^{\overline{U}{}_{\alpha}}\hat{\sigma}=\hat{\sigma}$
for all $\hat{\sigma}$ by (\ref{eq: orthog decomp overline}), and
$m_{\alpha}\in U_{\alpha}\subset\overline{U}{}_{\alpha}$ (see (\ref{eq: m in U alpha k}),
(\ref{eq: Ualpha k in terms of u})).

To show the inequality fix $1\le k\le5\eta^{-2}$ and $\alpha\in\mathcal{A}_{k}$.
By (\ref{eq: us orthonormal}), (\ref{eq: Ualpha k in terms of u})
we have for any $\sigma\in\mathbb{R}^{N}$
\[
\|P^{\overline{U}{}_{\alpha}}\hat{\sigma}\|^{2}=\sum_{l=1}^{k+1}\sum_{j=1}^{A_{l}}\left\langle \hat{\sigma},u_{\alpha,l,j}\right\rangle ^{2}=\sum_{l=1}^{k+1}\sum_{j=1}^{A_{l}}\left\langle \sigma-m_{\alpha},u_{\alpha,l,j}\right\rangle ^{2}.
\]
By (\ref{eq: us orthonormal}), (\ref{eq: m alpha expansion}) we
have $\left\langle m_{\alpha},u_{\alpha,l,j}\right\rangle =\alpha_{l,j}$
for $1\le l\le k,1\le j\le A_{l}$, and $\left\langle m_{\alpha},u_{\alpha,k+1,j}\right\rangle =0$
for $j=1,2$. Thus the right-hand side can be written as 
\[
\sum_{l=1}^{k}\sum_{j=1}^{A_{l}}\left(\left\langle \sigma,u_{\alpha,l,j}\right\rangle -\alpha_{l,j}\right)^{2}+\left\{ \left\langle \sigma,u_{\alpha,k+1,1}\right\rangle ^{2}+\left\langle \sigma,u_{\alpha,k+1,2}\right\rangle ^{2}\right\} .
\]
By the definition (\ref{eq: Dalpha def}) of $D_{\alpha}$ and (\ref{eq: rounding distance}),
and the definition (\ref{eq: Ealpha def}) of $\mathcal{E}_{\alpha}$,
this is at most
\[
\left(K+2\left(k-1\right)\right)\varepsilon^{2}+2\eta^{2}.
\]
Using $k\le5\eta^{-2},K\le\eta^{-2}$ and $\varepsilon\le\eta^{2}$
this can be bounded by $16\eta^{2}$, giving the claim.
\end{proof}
For $\alpha\in\mathcal{A}_{k}$ we refer to the set 
\[
\left(m_{\alpha}+\overline{V}_{\alpha}\right)\cap B_{N}=m_{\alpha}+B_{N}\left(\sqrt{1-\|m_{\alpha}\|^{2}}\right)\cap\overline{V}_{\alpha},
\]
as a ``slice'' of the ball centered at $m_{\alpha}$ of $\text{dim}\,\overline{V}_{\alpha}=N-K-2k$.
Since $\sigma\in\mathcal{E}_{\alpha}$ satisfies $\|P^{\overline{U}_{\alpha}}\sigma-m_{\alpha}\|\le4\eta$
by the previous lemma we have

\[
\mathcal{E}_{\alpha}\subset\bigcup_{m\in\overline{U}{}_{\alpha},\|m-m_{\alpha}\|\le4\eta}\left(m+\overline{V}_{\alpha}\right)\cap B_{N}\text{ for }\alpha\in\mathcal{A}_{\varepsilon,\eta},
\]
where the right-hand side can be thought of as a ``thick slice''
centered at $m_{\alpha}$.

The next result uses the previous lemma and the fact that $\overline{U}_{\alpha}$
contains the normal of an (almost) minimal entropy hyperplane around
$m_{\alpha}$ to bound $E[\mathcal{E}_{\alpha}]$ in terms of $I_{E,\delta}\left(m_{\alpha}\right)$.
\begin{prop}
\label{prop: entropy of Bs}For $I_{E,\delta}$ as in (\ref{eq: IN})
and any $\delta>0,0<\eta\le\min\left(K^{-1/2},\delta/4\right)$ and
$0<\varepsilon\le\eta^{2}$ it holds that
\[
E\left[\mathcal{E}_{\alpha}\right]\le e^{I_{E,\delta}\left(m_{\alpha}\right)+\delta}\text{\,for all }\alpha\in\mathcal{A}_{\varepsilon,\eta}\text{ almost surely.}
\]
\textup{}
\end{prop}

\begin{proof}
Fix $\alpha\in\mathcal{A}_{\varepsilon,\eta}$. Since $\lambda_{m_{\alpha}}^{\delta}\in\overline{U}_{\alpha}=\overline{V}_{\alpha}^{\bot}$
(recall (\ref{eq: heff lambda m in U})) and $\|\lambda_{m_{\alpha}}^{\delta}\|=1$
(recall Definition \ref{def: lambda}) we have for any $\sigma$
\[
\left|\left\langle \lambda_{m_{\alpha}}^{\delta},\sigma-m_{\alpha}\right\rangle \right|=\left|\left\langle \lambda_{m_{\alpha}}^{\delta},\hat{\sigma}\right\rangle \right|=\left|\left\langle \lambda_{m_{\alpha}}^{\delta},P^{\overline{V}_{\alpha}}\hat{\sigma}-\hat{\sigma}\right\rangle \right|\le\|P^{\overline{V}_{\alpha}}\hat{\sigma}-\hat{\sigma}\|,
\]
and by Lemma \ref{lem: Dalpha is thick slice}
\[
\|P^{\overline{V}_{\alpha}}\hat{\sigma}-\hat{\sigma}\|\le4\eta\le\delta,
\]
for any $\sigma\in\mathcal{E}_{\alpha}$ if $\eta\le\delta/4$, so
that
\[
\left\langle \lambda_{m_{\alpha}}^{\delta},\sigma-m_{\alpha}\right\rangle \ge-\delta\text{\,for all }\sigma\in\mathcal{E}_{\alpha}.
\]
Therefore 
\[
E\left[\mathcal{E}_{\alpha}\right]\le E\left[\left\langle \lambda_{m_{\alpha}}^{\delta},\sigma-m_{\alpha}\right\rangle \ge-\delta\right]\overset{\eqref{eq: lambda delta m and I E delta}}{\le}\exp\left(I_{E,\delta}\left(m_{\alpha}\right)+\delta\right).
\]
\end{proof}
The previous proposition will allow us to deal with the term $E[\mathcal{E}_{\alpha}]$
in (\ref{eq: foreshadow 2}). We now turn to the other term on the
RHS of (\ref{eq: foreshadow 2}), namely the normalized partition
function $E^{m_{\alpha}}[\exp(\beta H_{N}^{f}(\sigma))]$. We first
show that the effective external field on $\mathcal{E}_{\alpha}$
can be neglected. This is one of the most important properties of
the cover $\left(\mathcal{E}_{\alpha}\right)_{\alpha\in\mathcal{A}_{\varepsilon,\eta}}$,
together with the fact that it consists of a small number of sets.
\begin{prop}[Effective external field on $\mathcal{E}_{\alpha}$ vanishes for $\alpha\in\mathcal{A}_{\varepsilon,\eta}$]
\label{prop: effective external field B-1}$ $\\
For $0<\eta\le K^{-1/2}$ and $0<\varepsilon\le\eta^{2}$ it holds
a.s. for all $\alpha\in\mathcal{A}_{\varepsilon,\eta}$ that
\begin{equation}
\sigma\in\mathcal{E}_{\alpha}\implies\left|\left\langle \hat{\sigma},h_{{\rm eff}}\left(m_{\alpha}\right)\right\rangle \right|\le4\eta\|h_{{\rm eff}}\left(m_{\alpha}\right)\|\text{ where }\hat{\sigma}=\sigma-m_{\alpha}.\label{eq: eff ext field vanish}
\end{equation}
\end{prop}

\begin{proof}
Since $\overline{U}{}_{\alpha}\oplus\overline{V}{}_{\alpha}$ is
an orthogonal decomposition of $\mathbb{R}^{N}$ (see (\ref{eq: orthog decomp overline}))
\[
\left|\left\langle \hat{\sigma},h_{{\rm eff}}\left(m_{\alpha}\right)\right\rangle \right|\le\left|\left\langle P^{\overline{V}{}_{\alpha}}\hat{\sigma},h_{{\rm eff}}\left(m_{\alpha}\right)\right\rangle \right|+\|P^{\overline{U}{}_{\alpha}}\hat{\sigma}\|\|h_{{\rm eff}}\left(m_{\alpha}\right)\|.
\]
The first term on the RHS vanishes since $h_{{\rm eff}}\left(m_{\alpha}\right)\in\overline{U}{}_{\alpha}$
(recall (\ref{eq: heff lambda m in U})), and by Lemma \ref{lem: Dalpha is thick slice}
we have $\|P^{\overline{U}{}_{\alpha}}\hat{\sigma}\|\le4\eta$ so
(\ref{eq: eff ext field vanish}) follows.
\end{proof}
In the remainder of the proof we will assume that $\xi'''\left(1\right)<\infty$
and work on the event
\begin{equation}
\mathcal{L}_{N}=\left\{ \sup_{m\in B_{N}^{\circ}}\|\nabla H_{N}\left(m\right)\|\le c_{\xi},\frac{\left|H_{N}\left(m\right)-H_{N}\left(\tilde{m}\right)\right|}{N}\le c_{\xi}\|m-m'\|\forall m,m'\in B_{N}^{\circ}\right\} ,\label{eq: lipschitz event}
\end{equation}
for $c_{\xi}=c\sqrt{\xi'''\left(1\right)}$, which by Lemma \ref{lem: grad estimate}
and $\xi'\left(1\right)\le\xi''\left(1\right)\le\xi'''\left(1\right)$
satisfies
\begin{equation}
\mathbb{P}\left(\mathcal{L}_{N}\right)\le e^{-N}\text{\,for all }N\ge1,\label{eq: the event prob}
\end{equation}
for a large enough universal $c$. We now use the decomposition (\ref{eq: decomp with ext field})
to show that the normalized partition function $E^{m_{\alpha}}[\exp(\beta H_{N}^{f}(\sigma))]$
can be bounded in terms of the partition function of the recentered
Hamiltonian $H_{N}^{m_{\alpha}}$ on $\mathcal{E}_{\alpha}$.
\begin{prop}[Decomposing Hamiltonian on thick slice]
\label{prop: using decomposition}Let $\beta\ge0,L>0$. If $\sqrt{\xi'''\left(1\right)}\le L$
and if $f_{N}:B_{N}\to\mathbb{R}$ satisfies $f_{N}\left(\sigma\right)=f_{N}\left(P^{U}\sigma\right)$
for all $\sigma$ and is Lipschitz with respect to $\|\cdot\|$ with
Lipschitz constant at most $LN$, it holds for $0<\eta\le K^{-1/2}$
and $0<\varepsilon\le\eta^{2}$ that
\begin{equation}
E^{m_{\alpha}}\left[\exp\left(\beta H_{N}^{f}\left(\sigma\right)\right)\right]\le e^{\beta H_{N}^{f}\left(m_{\alpha}\right)}E^{m_{\alpha}}\left[\exp\left(\beta H_{N}^{m_{\alpha}}\left(\hat{\sigma}\right)\right)\right]e^{c\eta\beta LN}\text{ for all }\alpha\in\mathcal{A}_{\varepsilon,\eta},\label{eq: remove eff ext field}
\end{equation}
on the event (\ref{eq: lipschitz event}).
\end{prop}

\begin{proof}
Fix $\alpha\in\mathcal{A}_{\varepsilon,\eta}$. We use the decomposition
(\ref{eq: decomp with ext field}) with $\sigma=m_{\alpha}+\hat{\sigma}$
, i.e.
\begin{equation}
H_{N}^{f}\left(m_{\alpha}+\hat{\sigma}\right)=H_{N}^{f}\left(m_{\alpha}\right)+h_{{\rm eff}}\left(m_{\alpha}\right)\cdot\hat{\sigma}+\left(f_{N}\left(m_{\alpha}+\hat{\sigma}\right)-f_{N}\left(m_{\alpha}\right)\right)+H_{N}^{m_{\alpha}}\left(\hat{\sigma}\right).\label{eq: decomp}
\end{equation}
Since $f_{N}\left(x\right)$ depends only on $P^{U}x$, the Lipschitz
assumption on $f_{N}$ implies that\\
${\left|f_{N}\left(m_{\alpha}+\hat{\sigma}\right)-f_{N}\left(m_{\alpha}\right)\right|\le LN\|P^{U}\hat{\sigma}\|}$
for all $\hat{\sigma}$. We have for $\sigma\in\mathcal{E_{\alpha}\subset D_{\alpha}}$
\[
\|P^{U_{1}}\hat{\sigma}\|^{2}=\sum_{i=1}^{K}\left(\left\langle \sigma,u_{i}\right\rangle -\left\langle m_{\alpha,1},u_{i}\right\rangle \right)^{2}\overset{\eqref{eq: first us},\eqref{eq: m alpha expansion}}{=}\sum_{i=1}^{K}\left(\left\langle \sigma,u_{i}\right\rangle -\alpha_{1,i}\right)^{2}\overset{\eqref{eq: rounding distance}.\eqref{eq: Dalpha def}}{\le}\varepsilon^{2}K
\]
so since $\varepsilon\le\eta^{2}$ and $\eta\le K^{-1/2}$ 
\begin{equation}
\left|f_{N}\left(m_{\alpha}+\hat{\sigma}\right)-f_{N}\left(m_{\alpha}\right)\right|\le NL\eta\text{ for all }\alpha\in\mathcal{A}\text{ and }\hat{\sigma}\in\mathcal{E}_{\alpha}-m_{\alpha}.\label{eq: f bound}
\end{equation}
Also $\|h_{{\rm eff}}\left(m_{\alpha}\right)\|=\|\nabla H_{N}\left(m_{\alpha}\right)\|\le c_{\xi}\le cL$
on the event (\ref{eq: lipschitz event}) so that by Proposition \ref{prop: effective external field B-1}
we obtain that on that event
\begin{equation}
\left|h_{{\rm eff}}\left(m_{\alpha}\right)\cdot\hat{\sigma}\right|=N\left|\left\langle h_{{\rm eff}}\left(m_{\alpha}\right),\hat{\sigma}\right\rangle \right|\le NcL\eta\text{ for all }\alpha\in\mathcal{A}_{\varepsilon,\eta}\text{ and }\hat{\sigma}\in\mathcal{E}_{\alpha}-m_{\alpha}.\label{eq: eff ext field}
\end{equation}
The claim (\ref{eq: remove eff ext field}) follows from (\ref{eq: decomp}),
(\ref{eq: f bound}) and (\ref{eq: eff ext field}).
\end{proof}
We now aim to bound the normalized partition function $E^{m_{\alpha}}[\exp(\beta H_{N}^{m_{\alpha}}(\hat{\sigma}))]$
of the recentered Hamiltonian from the RHS of (\ref{eq: remove eff ext field}).
To do so we will first move from the ``thick slices'' $\mathcal{E}_{\alpha}$
to the ``thin slices'' $m_{\alpha}+\Sigma_{\alpha}$, where
\begin{equation}
\Sigma_{\alpha}=B_{N}\left(\sqrt{1-\|m_{\alpha}\|^{2}}\right)\cap\overline{V}{}_{\alpha}\text{ for }\alpha\in\mathcal{A}.\label{eq: sigma alpha def}
\end{equation}
Note that for $\sigma\in m_{\alpha}+\Sigma_{\alpha}$ we have $\sigma=m_{\alpha}+\hat{\sigma}$
with $\hat{\sigma}\in\overline{V}{}_{\alpha}$, the latter being only
approximately true for a $\sigma\in\mathcal{E}_{\alpha}$.

For $\sigma\in\mathcal{E}_{\alpha}$ we define
\begin{equation}
\hat{\tau}_{\alpha}\left(\sigma\right)=\sqrt{1-\|m_{\alpha}\|^{2}}\frac{P^{\overline{V}{}_{\alpha}}\sigma}{\|P^{\overline{V}{}_{\alpha}}\sigma\|}\in\Sigma_{\alpha},\label{eq: tau alpha def}
\end{equation}
(with the convention $\frac{P^{\overline{V}{}_{\alpha}}\sigma}{\|P^{\overline{V}{}_{\alpha}}\sigma\|}=0$
if $P^{\overline{V}{}_{\alpha}}\sigma=0$) to be a point in $\Sigma_{\alpha}$
that approximates $\hat{\sigma}$ well, as shown by the next lemma.
\begin{lem}[Projecting to ``thin slice'']
\label{lem: projecting back to sphere}For $0<\eta\le K^{-1/2}$
and $0<\varepsilon\le\eta^{2}$ we have almost surely for all $\alpha\in\mathcal{A}_{\varepsilon,\eta}$
and $\sigma\in\mathcal{E}_{\alpha}$ that

a) 
\begin{equation}
\left|\|P^{\overline{V}{}_{\alpha}}\sigma\|-\sqrt{1-\|m_{\alpha}\|^{2}}\right|\le8\eta^{1/4}.\label{eq: claim}
\end{equation}

b) 
\[
\|\sigma-\left(m_{\alpha}+\hat{\tau}_{\alpha}\left(\sigma\right)\right)\|=\|\hat{\sigma}-\hat{\tau}_{\alpha}\left(\sigma\right)\|\le12\eta^{1/4}.
\]
\end{lem}

\begin{proof}
a) Since $\overline{U}{}_{\alpha}\oplus\overline{V}{}_{\alpha}$ is
an orthogonal decomposition of $\mathbb{R}^{N}$ (see (\ref{eq: orthog decomp overline}))
\begin{equation}
\|P^{\overline{V}{}_{\alpha}}\sigma\|=\sqrt{1-\|P^{\overline{U}{}_{\alpha}}\sigma\|^{2}}.\label{eq: PV PU}
\end{equation}
Since $\|m_{\alpha}\|,\|\sigma\|\le1$ it holds that
\[
\left|\|m_{\alpha}\|^{2}-\|P^{\overline{U}{}_{\alpha}}\sigma\|^{2}\right|=\left|\|m_{\alpha}\|+\|P^{\overline{U}{}_{\alpha}}\sigma\|\right|\left|\|m_{\alpha}\|-\|P^{\overline{U}{}_{\alpha}}\sigma\|\right|\le2\left|\|m_{\alpha}\|-\|P^{\overline{U}{}_{\alpha}}\sigma\|\right|,
\]
so that by Lemma \ref{lem: Dalpha is thick slice}
\begin{equation}
\left|\|m_{\alpha}\|^{2}-\|P^{\overline{U}{}_{\alpha}}\sigma\|^{2}\right|\le8\eta,\label{eq: bound squares}
\end{equation}
for $\sigma\in\mathcal{E}_{\alpha}$. Now if $\|m_{\alpha}\|^{2}\ge1-\eta^{1/2}$,
then
\begin{equation}
\begin{array}{ccl}
\left|\|P^{\overline{V}{}_{\alpha}}\sigma\|-\sqrt{1-\left|m_{\alpha}\right|^{2}}\right| & \le & \|P^{\overline{V}{}_{\alpha}}\sigma\|+\sqrt{1-\left|m_{\alpha}\right|^{2}}\\
 & \overset{\eqref{eq: PV PU},\eqref{eq: bound squares}}{\le} & \sqrt{1-\|m_{\alpha}\|^{2}+8\eta}+\sqrt{1-\|m_{\alpha}\|^{2}}\\
 & \le & \sqrt{\eta^{1/2}+8\eta}+\sqrt{\eta^{1/2}}\\
 & \le & 4\eta^{1/4}.
\end{array}\label{eq: bound 1}
\end{equation}

If on the other hand $\|m_{\alpha}\|^{2}\le1-\eta^{1/2}$ then 
\begin{equation}
\begin{array}{ccl}
\|P^{\overline{V}{}_{\alpha}}\sigma\|-\sqrt{1-\|m_{\alpha}\|^{2}} & \overset{\eqref{eq: PV PU}}{=} & \sqrt{1-\|P^{\overline{U}{}_{\alpha}}\sigma\|^{2}}-\sqrt{1-\|m_{\alpha}\|^{2}}\\
 & = & \sqrt{1-\|m_{\alpha}\|^{2}}\left(\sqrt{1+\frac{\|m_{\alpha}\|^{2}-\|P^{\overline{U}{}_{\alpha}}\sigma\|^{2}}{1-\|m_{\alpha}\|^{2}}}-1\right).
\end{array}\label{eq: taking out}
\end{equation}
and
\[
\left|\frac{\|m_{\alpha}\|^{2}-\|P^{\overline{U}{}_{\alpha}}\sigma\|^{2}}{1-\|m_{\alpha}\|^{2}}\right|\overset{\eqref{eq: bound squares}}{\le}\frac{8\eta}{\eta^{1/2}}=8\eta^{1/2}.
\]
Since $\left|\sqrt{1+x}-1\right|\le\left|x\right|$ for all $x\ge-1$
we thus have that 
\[
\left|\sqrt{1+\frac{\|m\|^{2}-\|P^{\overline{U}{}_{\alpha}}\sigma\|^{2}}{1-\|m\|^{2}}}-1\right|\le8\eta^{1/2},
\]
and so from (\ref{eq: taking out})
\[
\begin{array}{ccl}
\left|\|P^{\overline{V}{}_{\alpha}}\sigma\|-\sqrt{1-\|m_{\alpha}\|^{2}}\right| & \le & \sqrt{1-\|m_{\alpha}\|^{2}}8\eta^{1/2}\le8\eta^{1/2}.\end{array}
\]
Combing this with (\ref{eq: bound 1}) gives (\ref{eq: claim}).

b) It holds that
\[
\begin{array}{ccl}
\|\sigma-\left(m_{\alpha}+\hat{\tau}_{\alpha}\left(\sigma\right)\right)\| & \overset{\eqref{eq: tau alpha def}}{=} & \|\hat{\sigma}-\sqrt{1-\|m_{\alpha}\|^{2}}\frac{P^{\overline{V}{}_{\alpha}}\sigma}{\|P^{\overline{V}{}_{\alpha}}\sigma\|}\|\\
 & \le & \|\hat{\sigma}-P^{\overline{V}{}_{\alpha}}\sigma\|+\|P^{\overline{V}{}_{\alpha}}\sigma-\sqrt{1-\left|m_{\alpha}\right|^{2}}\frac{P^{\overline{V}{}_{\alpha}}\sigma}{\|P^{\overline{V}{}_{\alpha}}\sigma\|}\|\\
 & \overset{\text{Lemma }\ref{lem: Dalpha is thick slice}}{\le} & 4\eta+\left|\|P^{\overline{V}{}_{\alpha}}\sigma\|-\sqrt{1-\|m_{\alpha}\|^{2}}\right|\\
 & \overset{\eqref{eq: claim}}{\le} & 4\eta+8\eta^{1/4}\le12\eta^{1/4}.
\end{array}
\]
\end{proof}
Let now 
\begin{equation}
E^{\Sigma_{\alpha}}\text{ denote the law of }\hat{\tau}_{\alpha}\left(\sigma\right)\text{\,under }E^{m_{\alpha}}.\label{eq: thin slice law}
\end{equation}
The next lemma bounds the partition function on the ``thick slice''
$\mathcal{E}_{\alpha}$ by that on the ``thin slice'' $m_{\alpha}+\Sigma_{\alpha}$,
using the previous lemma to bound the error made when approximating
$\hat{\sigma}$ by $\hat{\tau}_{\alpha}\left(\hat{\sigma}\right)$.
\begin{lem}[From ``thick slice'' to ``thin slice'']
\label{lem: to thin slice}Let $\beta\ge0,L>0$. If $\sqrt{\xi'''\left(1\right)}\le L$
, $0<\eta\le K^{-1/2}$ and $0<\varepsilon\le\eta^{2}$ then on the
event (\ref{eq: lipschitz event})
\begin{equation}
E^{m_{\alpha}}\left[\exp\left(\beta H_{N}^{m_{\alpha}}\left(\hat{\sigma}\right)\right)\right]\le E^{\Sigma_{\alpha}}\left[\exp\left(\beta H_{N}^{m_{\alpha}}\left(\sigma\right)\right)\right]e^{c\eta^{1/4}\beta LN}\text{ for all }\alpha\in\mathcal{A}_{\varepsilon,\eta}.\label{eq: to thin slice}
\end{equation}
\end{lem}

\begin{proof}
From (\ref{eq: recentered hamil}) it holds for any $\hat{\sigma},\hat{\tau}\in B_{\alpha}:=B_{N}\left(\sqrt{1-\|m_{\alpha}\|^{2}}\right)$
and any $\alpha$ that
\[
H_{N}^{m_{\alpha}}\left(\hat{\sigma}\right)-H_{N}^{m_{\alpha}}\left(\hat{\tau}\right)=H_{N}\left(m_{\alpha}+\hat{\sigma}\right)-H_{N}\left(m_{\alpha}+\hat{\tau}\right)-\nabla H_{N}\left(m_{\alpha}\right)\cdot\left(\hat{\sigma}-\hat{\tau}\right).
\]
Thus on the event (\ref{eq: lipschitz event}) we have
\begin{equation}
\left|H_{N}^{m_{\alpha}}\left(\hat{\sigma}\right)-H_{N}^{m_{\alpha}}\left(\hat{\tau}\right)\right|\le cLN\|\hat{\sigma}-\hat{\tau}\|\text{ for all }\hat{\sigma},\hat{\tau}\in B_{\alpha}.\label{eq: lipschitz recent hamilt}
\end{equation}
Recall from (\ref{eq: PV PU}) that $\hat{\tau}_{\alpha}\left(\sigma\right)\in B_{\alpha}$
for all $\sigma\in\mathcal{E}_{\alpha}$. Thus (\ref{eq: lipschitz recent hamilt})
together with Lemma \ref{lem: projecting back to sphere} b) imply
that on the event (\ref{eq: lipschitz event}) 
\[
E^{m_{\alpha}}\left[\exp\left(\beta H_{N}^{m_{\alpha}}\left(\hat{\sigma}\right)\right)\right]\le E^{m_{\alpha}}\left[\exp\left(\beta H_{N}^{m_{\alpha}}\left(\hat{\tau}_{\alpha}\left(\sigma\right)\right)\right)\right]e^{c\eta^{1/4}\beta LN}.
\]
But by the definition (\ref{eq: thin slice law}) of $E^{\Sigma_{\alpha}}$
we have
\begin{equation}
E^{m_{\alpha}}\left[\exp\left(\beta H_{N}^{m_{\alpha}}\left(\hat{\tau}_{\alpha}\left(\sigma\right)\right)\right)\right]=E^{\Sigma_{\alpha}}\left[\exp\left(\beta H_{N}^{m_{\alpha}}\left(\sigma\right)\right)\right],\label{eq: pushforward}
\end{equation}
which gives the claim (\ref{eq: to thin slice}).
\end{proof}
We have thus reduced the proof of Theorem \ref{thm: main_thm_general}
to bounding the partition function $E^{\Sigma_{\alpha}}[\exp(\beta H_{N}^{m_{\alpha}}(\sigma))]$
on the ``thin slice'' $\Sigma_{\alpha}$. Next we do so uniformly
over all $m_{\alpha},\alpha\in\mathcal{A}_{\varepsilon,\eta}$. These
bounds give rise to the Onsager correction in the TAP free energy.
We obtain the bound by a simple Markov inequality for the partition
function over each slice $\Sigma_{\alpha}$, and since the number
of $m_{\alpha}$'s is small a union bound shows that the upper bound
holds for all $m_{\alpha}$ simultaneously with high probability.
Since the recentered Hamiltonian has no external field the Markov
inequality will give a tight upper for small enough $\beta$. Recall
the notation $q_{\alpha}=\|m_{\alpha}\|^{2}$ from (\ref{eq: q alpha def}).
\begin{prop}[Onsager correction]
\label{prop: Onsager}For any $\varepsilon,\eta>0$ it holds for
$N,\beta,E$ as in Theorem \ref{thm: main_thm_general} and any $\delta>0$
and
\begin{equation}
\mathcal{O}_{N,\delta,\varepsilon,\eta}=\left\{ \begin{array}{c}
E^{\Sigma_{\alpha}}\left[\exp\left(\beta H_{N}^{m_{\alpha}}\left(\hat{\sigma}\right)\right)\right]\le e^{N\frac{\beta^{2}}{2}{\rm On}\left(q_{\alpha}\right)+\delta N}\\
\text{ for all }\alpha\in\mathcal{A}_{\varepsilon,\eta}\text{ s.t. }E\left[\mathcal{E}_{\alpha}\right]>0
\end{array}\right\} ,\label{eq: the event}
\end{equation}
that
\[
\mathbb{P}\left(\mathcal{O}_{N,\delta,\varepsilon,\eta}\right)\le\left|\mathcal{A}_{\varepsilon,\eta}\right|e^{-\delta N}.
\]
\end{prop}

\begin{proof}
Fix $\alpha\in\mathcal{A}$ and consider
\begin{equation}
\mathbb{P}\left(E^{\Sigma_{\alpha}}\left[\exp\left(\beta H_{N}^{m_{\alpha}}\left(\hat{\sigma}\right)\right)\right]\ge e^{\frac{\beta^{2}}{2}\xi_{q_{\alpha}}\left(1-q_{\alpha}\right)+\delta N}|\mathcal{R}_{\alpha}\right),\label{eq: onsager term one B}
\end{equation}
for $\mathcal{R}_{\alpha}$ from (\ref{eq: R alpha def}). To lighten
notation drop the $\alpha$ subscript and write $\Sigma=\Sigma_{\alpha}$,
$m=m_{\alpha}$, $q=q_{\alpha}=\|m_{\alpha}\|^{2}$, $\mathcal{R}=\mathcal{R}_{\alpha}$
and $\overline{V}=\overline{V}_{\alpha}$. Note that $m$ and $\overline{V}$,
and therefore $\Sigma$, are deterministic functions of $u_{\alpha,l,j},1\le l\le k+1,1\le j\le A_{l}$,
are are thus deterministic under the measure $\mathbb{P}\left(\cdot|\mathcal{R}\right)$
(see (\ref{eq: m alpha expansion}) for $m$ and (\ref{eq: Ualpha k in terms of u})
for $\overline{V}$, and (\ref{eq: sigma alpha def}) for $\Sigma=\Sigma_{\alpha}$).

By Markov's inequality
\[
\begin{array}{l}
\mathbb{P}\left(E^{\Sigma_{\alpha}}\left[\exp\left(\beta H_{N}^{m}\left(\hat{\sigma}\right)\right)\right]\ge e^{\frac{\beta^{2}}{2}\xi_{q}\left(1-q\right)+\delta N}|\mathcal{R}\right)\\
\le\mathbb{E}\left[E^{\Sigma_{\alpha}}\left[\exp\left(\beta H_{N}^{m}\left(\hat{\sigma}\right)\right)\right]|\mathcal{R}\right]e^{-N\frac{\beta^{2}}{2}\xi_{q}\left(1-q\right)-\delta N}\\
=E^{\Sigma_{\alpha}}\left[\mathbb{E}\left[\exp\left(\beta H_{N}^{m}\left(\hat{\sigma}\right)\right)|\mathcal{R}\right]\right]e^{-N\frac{\beta^{2}}{2}\xi_{q}\left(1-q\right)-\delta N}.
\end{array}
\]
Lemma \ref{lem: law of recentering alpha} implies that for fixed
$\hat{\sigma}\in\Sigma$ the $\mathbb{P}\left(\cdot|\mathcal{R}\right)$-law
of $H_{N}^{m}\left(\hat{\sigma}\right)$ is that of a centered normal
of variance $\xi_{q}\left(1-q\right)$, so that
\[
\mathbb{E}\left[\exp\left(\beta H_{N}^{m}\left(\hat{\sigma}\right)\right)|\mathcal{R}\right]=e^{N\frac{\beta^{2}}{2}\xi_{q}\left(1-q\right)}\text{ for any }\hat{\sigma}\in\Sigma.
\]
Thus in fact for all $\alpha\in\mathcal{A}$ it holds that
\[
\mathbb{P}\left(E^{\Sigma_{\alpha}}\left[\exp\left(\beta H_{N}^{m_{\alpha}}\left(\hat{\sigma}_{\alpha}\right)\right)\right]\ge e^{\frac{\beta^{2}}{2}\xi^{q_{\alpha}}\left(1-q\right)+\delta N}\right)\le e^{-\delta N}.
\]
A union bound over  $\alpha\in\mathcal{A}_{\varepsilon,\eta}$ and
(\ref{eq: Onsager and recent xi}) completes the proof.
\end{proof}
We are now ready to prove the general TAP upper bound. As already
mentioned we do so by splitting the partition function into integrals
over each set $\mathcal{E}_{\alpha}$ in the cover, normalizing these
integrals and recentering the Hamiltonian in them, and using the previous
results to bound the partition function on the slices by the Onsager
correction. 
\begin{proof}[Proof of Theorem \ref{thm: main_thm_general}]
Assume 
\begin{equation}
0<\eta\le\min\left(K^{-1/2},\frac{1}{2}\right)\text{ and }0<\varepsilon\le\eta^{2}\le\frac{\eta}{2}.\label{eq: eta epsilon assumption}
\end{equation}
We work on the event $\mathcal{L}_{N}\cap\mathcal{O}_{N,\delta/2,\varepsilon,\eta}$,
where $\mathcal{L}_{N}$ is the event from (\ref{eq: lipschitz event})
and $\mathcal{O}_{N,\delta/2,\varepsilon,\eta}$ is the event from
Proposition \ref{prop: Onsager}. By (\ref{eq: the event prob}) and
Proposition \ref{prop: Onsager} we have
\begin{equation}
\mathbb{P}\left(\mathcal{L}_{N}\cap\mathcal{O}_{N,\delta/2,\varepsilon,\eta}\right)\ge1-2\left|\mathcal{A}_{\varepsilon,\eta}\right|e^{-\frac{\delta}{2}N},\label{eq: good event prob to 1}
\end{equation}
(note that $\delta\in\left(0,1\right)$ and $\left|\mathcal{A}_{\varepsilon,\eta}\right|\ge1$).

By Proposition \ref{prop: cover} and the definition of $E^{m_{\alpha}}$
we have
\begin{equation}
\begin{array}{lcl}
E\left[\exp\left(\beta H_{N}^{f}\left(\sigma\right)\right)\right] & \le & \sum_{\alpha\in\mathcal{A}_{\varepsilon,\eta}}E\left[1_{\mathcal{E}_{\alpha}}\exp\left(\beta H_{N}^{f}\left(\sigma\right)\right)\right]\\
 & \overset{\eqref{eq: E equator def}}{=} & \sum_{\alpha\in\mathcal{A}_{\varepsilon,\eta},E\left[\mathcal{E}_{\alpha}\right]>0}E\left[\mathcal{E}_{\alpha}\right]E^{m_{\alpha}}\left[\exp\left(\beta H_{N}^{f}\left(\sigma\right)\right)\right].
\end{array}\label{eq: step 1}
\end{equation}
By Proposition \ref{prop: entropy of Bs} this is, provided $\eta\le\frac{\delta}{6}$
and $N\ge6$ (so that $\delta\le\frac{\delta}{6}N$) bounded by
\begin{equation}
\sum_{\alpha\in\mathcal{A}_{\varepsilon,\eta}}\exp\left(I_{E,\delta}\left(m_{\alpha}\right)+\frac{\delta}{6}N\right)E^{m_{\alpha}}\left[\exp\left(\beta H_{N}^{f}\left(\sigma\right)\right)\right].\label{eq: step 2}
\end{equation}
By Proposition \ref{prop: using decomposition} this is at most
\[
\sum_{\alpha\in\mathcal{A}_{\varepsilon,\eta}}\exp\left(\beta H_{N}^{f}\left(m_{\alpha}\right)+I_{E,\delta}\left(m_{\alpha}\right)+\left(\frac{\delta}{6}+c\eta\beta L\right)N\right)E^{m_{\alpha}}\left[\exp\left(\beta H_{N}^{m_{\alpha}}\left(\hat{\sigma}\right)\right)\right].
\]
By Lemma \ref{lem: to thin slice} this is at most
\begin{equation}
\sum_{\alpha\in\mathcal{A}_{\varepsilon,\eta}}\exp\left(\beta H_{N}^{f}\left(m_{\alpha}\right)+I_{E,\delta}\left(m_{\alpha}\right)+\left(\frac{\delta}{6}+c\eta^{1/4}\beta L\right)N\right)E^{\Sigma_{\alpha}}\left[\exp\left(H_{N}^{m_{\alpha}}\left(\hat{\sigma}\right)\right)\right],\label{eq: step 3}
\end{equation}
on the event $\mathcal{L}_{N}$. On the event $\mathcal{O}_{N,\delta/2,\varepsilon,\eta}$
from Proposition \ref{prop: Onsager} this is bounded by 
\begin{equation}
\begin{array}{cl}
 & {\displaystyle \sum_{\alpha\in\mathcal{A}_{\varepsilon,\eta}}}\exp\left(\beta H_{N}^{f}\left(m_{\alpha}\right)+I_{E,\delta}\left(m_{\alpha}\right)+\frac{\beta^{2}}{2}{\rm On}\left(\|m_{\alpha}\|^{2}\right)+\left(\frac{4\delta}{6}+c\eta^{1/4}\beta L\right)N\right)\\
\overset{\eqref{eq: HTAP measure on subset of sphere}}{=} & {\displaystyle \sum_{\alpha\in\mathcal{A}_{\varepsilon,\eta}}}\exp\left(H_{{\rm TAP}}^{I_{E,\delta}}\left(m_{\alpha}\right)+\left(\frac{4\delta}{6}+c\eta^{1/4}\beta L\right)N\right)\\
\le & \left|\mathcal{A}_{\varepsilon,\eta}\right|\exp\left(\sup_{m\in B_{N}^{\circ}}H_{{\rm TAP}}^{I_{E,\delta}}\left(m\right)+\left(\frac{4\delta}{6}+c\eta^{1/4}\beta L\right)N\right).
\end{array}\label{eq: next next next step}
\end{equation}
Set $\hat{\kappa}=c\max(K,(\beta L/\delta)^{8})\ge1$ for some universal
large $c$, and $\eta=(\hat{\kappa})^{-1/2}$ and $\varepsilon=\eta^{2}$.
We then get that (\ref{eq: eta epsilon assumption}) and all inequalities
above hold if $N\ge6$, and in addition and in addition $c\eta^{1/4}\beta L\le\bar{\kappa}^{-8}\beta L\le\frac{\delta}{6}$
so that the previous line is at most
\begin{equation}
\left|\mathcal{A}_{\varepsilon,\eta}\right|\exp\left(\sup_{m\in B_{N}^{\circ}}H_{{\rm TAP}}^{E,\delta}\left(m\right)+\frac{5\delta}{6}N\right).\label{eq: next times 4 step}
\end{equation}
Furthermore $K\le\eta^{-2}$ so by (\ref{eq: size of cover})
\begin{equation}
\left|\mathcal{A}_{\varepsilon,\eta}\right|\le\left(2\eta^{-2}\right)^{11\eta^{-2}}=\left(2\hat{\kappa}\right)^{11\hat{\kappa}}\le\frac{\delta}{6}N.\label{eq: size of cover kappa}
\end{equation}
for $N\ge\delta^{-1}66\hat{\kappa}\log\left(2\hat{\kappa}\right)\ge6$.
Then (\ref{eq: next times 4 step}) is at most 
\begin{equation}
\exp\left(\sup_{m\in B_{N}^{\circ}}H_{{\rm TAP}}^{E,\delta}\left(m\right)+\delta N\right).\label{eq: final step}
\end{equation}
with probability at least $1-2\left(2\hat{\kappa}\right)^{11\hat{\kappa}}e^{-\frac{\delta}{2}N}$
(recall (\ref{eq: good event prob to 1}) and (\ref{eq: size of cover kappa})),
provided $N\ge\delta^{-1}66\hat{\kappa}\log\left(2\hat{\kappa}\right)$.
If $N\le\delta^{-1}66\hat{\kappa}\log\left(2\hat{\kappa}\right)$
then $\left(2\hat{\kappa}\right)^{\frac{33}{2}\hat{\kappa}}e^{-\frac{\delta}{2}N}\ge1$
so the claim (\ref{eq: general UB}) follows for all $N\ge1$ with
$\bar{\kappa}=\frac{33}{2}\hat{\kappa}$.
\end{proof}

\section{\label{sec: Ising SK proof}Bounding the Ising entropy and proof
of Ising TAP upper bound}

In this section we derive the TAP upper bound for the Ising mixed
SK model from the general TAP upper bound, by explicitly bounding
the entropy function $I_{E,\delta}\left(m\right)$ of (\ref{eq: IN})
by the entropy function $I_{{\rm Ising}}\left(m\right)$ from (\ref{eq: ising entropy I}).

For this it is convenient to extend the definition of $J$ from (\ref{eq: binary entropy def J})
so that
\begin{equation}
J\left(m\right)=\begin{cases}
\frac{1+m}{2}\log\left(1+m\right)+\frac{1-m}{2}\log\left(1-m\right) & \text{ if }\left|m\right|\le1,\\
\log2 & \text{ if }\left|m\right|\ge1.
\end{cases}\label{eq: extended J def}
\end{equation}
We will need the following simple bound.
\begin{lem}
For all $m,\tilde{m}\in\mathbb{R}$
\begin{equation}
\left|J\left(m\right)-J\left(\tilde{m}\right)\right|\le\left|m-\tilde{m}\right|\log\frac{2e}{\left|m-\tilde{m}\right|\wedge1}.\label{eq: J diff}
\end{equation}
\end{lem}

\begin{proof}
Since $0\le J\left(m\right)\le\log2$ for all $m$ we trivially have
\begin{equation}
\left|J\left(m\right)-J\left(\tilde{m}\right)\right|\le\left|m-\tilde{m}\right|\log2\text{ for }\left|m-\tilde{m}\right|\ge1.\label{eq: J diff first bound}
\end{equation}
From the shape of $J'\left(x\right)=\frac{1}{2}\log\frac{1+x}{1-x}$
(for $x\in\left(-1,1\right)$, otherwise $J'\left(x\right)=0$) we
see that if $\left|m-\tilde{m}\right|\le1$ then
\begin{equation}
\begin{array}{ccccl}
\left|J\left(\tilde{m}\right)-J\left(m\right)\right| & \le & \frac{1}{2}\int_{1-\left|\tilde{m}-m\right|}^{1}\log\frac{1+x}{1-x}dx & \le & \frac{1}{2}\int_{0}^{\left|\tilde{m}-m\right|}\log\frac{2}{z}dz\\
 &  &  & = & \frac{1}{2}\left|\tilde{m}-m\right|\log\frac{2e}{\left|\tilde{m}-m\right|}.
\end{array}\label{eq: J diff second bound}
\end{equation}
Combining (\ref{eq: J diff first bound}) and (\ref{eq: J diff second bound})
yields (\ref{eq: J diff}).
\end{proof}
The required bound on the entropy function is the following. Let $d\left(x,A\right)=\inf_{y\in A}\|x-y\|$
for $x\in\mathbb{R}^{N},A\subset\mathbb{R}^{N}$.
\begin{lem}[Entropy for Ising reference measure]
\label{lem: ising entropy lemma} Let $E$ be the uniform measure
on $\left\{ -1,1\right\} ^{N}$ and let $I_{E,\delta}$ be as in (\ref{eq: IN}).
For any $N\ge1,$$\delta\in\left(0,1\right)$ the following holds.
\begin{enumerate}
\item[a)] For all $m\in B_{N}$ 
\begin{equation}
I_{E,\delta}\left(m\right)\le I_{{\rm Ising}}\left(m\right)+cN\delta\log\delta^{-1}.\label{eq: Ising entropy 1}
\end{equation}
\item[b)] For $m\in B_{N}$ such that $d\left(m,\left[-1,1\right]^{N}\right)>\delta$
it holds that
\begin{equation}
I_{E,\delta}\left(m\right)=-\infty.\label{eq: Ising entropy 2}
\end{equation}
\end{enumerate}
\end{lem}

\begin{proof}
a) Note that for all $\lambda\in\mathbb{R},m\in\mathbb{R}$ and $i=1,\ldots,N$
\begin{equation}
\log E\left[\exp\left(\lambda\left(\sigma_{i}-m\right)\right)\right]=\log\cosh\left(\lambda\right)-\lambda m.\label{eq: exp mom of spin}
\end{equation}
Also 
\begin{equation}
\inf_{\lambda\in\mathbb{R}}\left\{ \log\cosh\left(\lambda\right)-\lambda m\right\} =\begin{cases}
-J\left(m\right) & \text{ if }\left|m\right|\le1,\text{ achieved for }\lambda=\text{atanh}\left(m\right),\\
-\infty & \text{ if }\left|m\right|>1,\text{ achieved for }\lambda\to\pm\infty.
\end{cases}\label{eq: exp mom minimizer}
\end{equation}
Now for $m\in B_{N}$ we have using the exponential Chebyshev inequality
\begin{equation}
\begin{array}{ccl}
I_{E,\delta}\left(m\right) & \overset{\eqref{eq: IN}}{=} & \inf_{\lambda\in\mathbb{R}^{N},\left|\lambda\right|=1}\log E\left[\lambda\cdot\left(\sigma-m\right)\ge-\delta\sqrt{N}\right]\\
 & \le & \inf_{r>0}\inf_{\lambda\in\mathbb{R}^{N},\left|\lambda\right|=1}\log\left(E\left[\exp\left(r\lambda\cdot\left(\sigma-m\right)\right)\right]e^{r\delta\sqrt{N}}\right)\\
 & = & \inf_{\lambda\in\mathbb{R}^{N}}\left\{ \log E\left[\exp\left(\lambda\cdot\left(\sigma-m\right)\right)\right]+\left|\lambda\right|\delta\sqrt{N}\right\} \\
 & \overset{\eqref{eq: exp mom of spin}}{=} & \inf_{\lambda\in\mathbb{R}^{N}}\left\{ \sum_{i=1}^{N}\left(\log\cosh\left(\lambda_{i}\right)-\lambda_{i}m_{i}\right)+\left|\lambda\right|\delta\sqrt{N}\right\} .
\end{array}\label{eq: IN bound}
\end{equation}
By (\ref{eq: exp mom minimizer}) the choice $\lambda_{i}=\text{atanh}\left(m_{i}\right)$
or $\lambda_{i}\to\pm\infty$ would be optimal for sum in the $\inf$,
but if some coordinates of $m$ are above or close to $\pm1$ it makes
the term $\left|\lambda\right|\delta\sqrt{N}$ explode. Therefore
we choose 
\begin{equation}
\lambda_{i}=\text{atanh}\left(\tilde{m}_{i}\right)\text{ where }\tilde{m}_{i}=\begin{cases}
1-\delta & \text{ if }m_{i}\ge1-\delta,\\
m_{i} & \text{ if }m_{i}\in\left[-\left(1-\delta\right),1-\delta\right],\\
-\left(1-\delta\right) & \text{\,if }m_{i}\le-\left(1-\delta\right).
\end{cases}\label{eq: m tilde def}
\end{equation}
We have
\[
\left|\lambda\right|\delta\sqrt{N}\le\text{atanh}\left(1-\delta\right)\delta N\le c\delta\log\delta^{-1}N\text{\,for all }\delta\in\left(0,1\right),N\ge1,
\]
and
\[
\log\cosh\left(\lambda_{i}\right)-\lambda_{i}m_{i}\overset{\eqref{eq: m tilde def}}{\le}\log\cosh\left(\lambda_{i}\right)-\lambda_{i}\tilde{m}_{i}\overset{\eqref{eq: exp mom minimizer},\eqref{eq: m tilde def}}{=}-J\left(\tilde{m}_{i}\right).
\]
so we obtain from (\ref{eq: IN bound})
\[
I_{E,\delta}\left(m\right)\le-\sum_{i=1}^{N}J\left(\tilde{m}_{i}\right)+c\delta\log\delta^{-1}N.
\]
Recalling that $I_{{\rm Ising}}\left(m\right)=-\sum_{i=1}^{N}J\left(m_{i}\right)$
and noting that $J\left(\tilde{m}_{i}\right)=J\left(m_{i}\right)$
if $\left|m_{i}\right|\le1-\delta$, and otherwise
\[
\left|J\left(m_{i}\right)-J\left(\tilde{m}_{i}\right)\right|\overset{\eqref{eq: extended J def},\eqref{eq: m tilde def}}{\le}\left|J\left(1\right)-J\left(1-\delta\right)\right|\overset{\eqref{eq: J diff}}{\le}c\delta\log\delta^{-1},
\]
the claim (\ref{eq: Ising entropy 1}) then follows.

b) Since $\left[-1,1\right]^{N}$ is convex, if $d(m,[-1,1]^{N})>\delta$
then there is a hyperplane separating $m$ from $\left[-1,1\right]^{N}$
at $\|\cdot\|$-distance greater than $\delta$ from $m$. This means
that there exists a $\lambda\in S_{N-1}$ such that
\[
\left\langle \lambda,\tilde{m}-m\right\rangle <-\delta\text{ for all }\tilde{m}\in\left[-1,1\right]^{N}.
\]
This in particular holds for any $\sigma\in\left\{ -1,1\right\} ^{N}$
in place of $\tilde{m}$, so 
\[
E\left[\left\langle \lambda,\sigma-m\right\rangle \ge-\delta\right]=0,
\]
so that (\ref{eq: Ising entropy 2}) follows from the definition (\ref{eq: IN})
of $I_{E,\delta}$.
\end{proof}
We also need the following continuity estimate for $I_{{\rm Ising}}$.
\begin{lem}
\label{lem: Iising Lipschitz}It holds for all $m,\tilde{m}\in B_{N}$
that
\[
\left|I_{{\rm Ising}}\left(m\right)-I_{{\rm Ising}}\left(\tilde{m}\right)\right|\le2N\|m-\tilde{m}\|\log\frac{4e}{\|m-\tilde{m}\|}.
\]
\end{lem}

\begin{proof}
By (\ref{eq: J diff}) we have for all $m,\tilde{m}\in\mathbb{R}^{N}$,
\begin{equation}
\left|I_{{\rm Ising}}\left(m\right)-I_{{\rm Ising}}\left(\tilde{m}\right)\right|\le\sum_{i=1}^{N}\left|m_{i}-\tilde{m}_{i}\right|\log\frac{2e}{\left|m_{i}-\tilde{m}_{i}\right|\wedge1}.\label{eq: from J diff}
\end{equation}
Note that
\begin{equation}
\sum_{i=1}^{N}\left|m_{i}-\tilde{m}_{i}\right|\le\sqrt{N}\left|m-\tilde{m}\right|=N\|m-\tilde{m}\|.\label{eq: l1 l2}
\end{equation}
Letting $f\left(x\right)=x\log\frac{2e}{x}$ the RHS of (\ref{eq: from J diff})
is thus bounded by
\[
\sum_{i=1}^{N}f\left(\left|m_{i}-\tilde{m}_{i}\right|\wedge1\right)+N\|m-\tilde{m}\|\log\left(2e\right).
\]
Note that $f\left(x\right)$ is concave on $[0,\infty)$ so 
\[
\sum_{i=1}^{N}f\left(\left|m_{i}-\tilde{m}_{i}\right|\wedge1\right)\le Nf\left(\frac{1}{N}\sum_{i=1}^{N}\left|m_{i}-\tilde{m}_{i}\right|\wedge1\right).
\]
Also $f$ is increasing on $\left[0,2\right]$ and
\[
\frac{1}{N}\sum_{i=1}^{N}\left|m_{i}-\tilde{m}_{i}\right|\wedge1\le\frac{1}{N}\sum_{i=1}^{N}\left|m_{i}-\tilde{m}_{i}\right|\overset{\eqref{eq: l1 l2}}{\le}\|m-\tilde{m}\|\le2,
\]
so that
\[
\sum_{i=1}^{N}f\left(\left|m_{i}-\tilde{m}_{i}\right|\wedge1\right)\le Nf\left(\|m-\tilde{m}\|\right)=N\|m-\tilde{m}\|\log\left(\frac{2e}{\|m-\tilde{m}\|}\right).
\]
Combining these gives the claim.
\end{proof}
From this and the continuity estimates of $H_{N}$ and $f_{N}$ we
obtain the following continuity estimate for $H_{{\rm TAP}}^{{\rm Ising}}\left(m\right)$.
\begin{lem}
\label{lem: HTAP lipschitz}Let $L,\beta,\xi$ and $f_{N}$ be like
in Theorem \ref{thm: main_thm_ising}. There is a constant $c$ such
that on the event $\mathcal{L}_{N}$ from (\ref{eq: lipschitz event})
it holds for all $m,\tilde{m}\in B_{N}^{\circ}$ that
\begin{equation}
\left|H_{{\rm TAP}}^{{\rm Ising}}\left(m\right)-H_{{\rm TAP}}^{{\rm Ising}}\left(\tilde{m}\right)\right|\le c\left(1+L^{3}\right)N\|m-\tilde{m}\|\log\frac{c}{\|m-\tilde{m}\|}\label{eq: HTAP Lipschitz}
\end{equation}
\end{lem}

\begin{proof}
We treat each term of $H_{{\rm TAP}}^{{\rm Ising}}\left(m\right)$
separately. 

Lemma \ref{lem: Iising Lipschitz} gives the sufficient bound for
the entropy term $I_{{\rm Ising}}$.

On the event (\ref{eq: lipschitz event}) we have
\[
\left|\beta H_{N}\left(m\right)-\beta H_{N}\left(\tilde{m}\right)\right|\le cL^{2}N\|m-\tilde{m}\|\text{ for all }m,\tilde{m}\in B_{N}^{\circ}.
\]
By the Lipschitz assumption
\[
\left|\beta f_{N}\left(m\right)-\beta f_{N}\left(\tilde{m}\right)\right|\le L^{2}N\|m-\tilde{m}\|\text{ for all }m,\tilde{m}\in B_{N}.
\]
For the Onsager term, note that from (\ref{eq: Onsager term def})
it follows that ${\rm On}\left(q\right)$ is Lipschitz with constant
bounded by $c\xi''\left(1\right)\le cL$ on $\left[0,1\right]$, so
that since\\
$\left|\|m\|^{2}-\|\tilde{m}\|^{2}\right|\le\left|\|m\|+\|\tilde{m}\right|\left|\|m\|-\|\tilde{m\|}\right|\le2\|m-\tilde{m}\|$
we have
\[
\left|N\frac{\beta^{2}}{2}{\rm On}\left(\|m\|^{2}\right)-N\frac{\beta^{2}}{2}{\rm On}\left(\|\tilde{m}\|^{2}\right)\right|\le cL^{3}N\|m-\tilde{m}\|\text{\,for all }m,\tilde{m}\in B_{N}.
\]
Combing all of the above imply (\ref{eq: HTAP Lipschitz}), since
for a large enough $c$ we have that $\log\frac{c}{\|m-\tilde{m}\|}\ge1$
for $m,\tilde{m}\in B_{N}$.
\end{proof}
Finally we derive the TAP upper bound for the Ising SK model from
the general TAP upper bound.
\begin{proof}[Proof of Theorem \ref{thm: main_thm_ising}]
Let $\tilde{\delta}>0$. By the general bound Theorem \ref{thm: main_thm_general}
with $\tilde{\delta}$ in place of $\delta$ we obtain that
\begin{equation}
\mathbb{P}\left(\log E\left[\exp\left(\beta H_{N}^{f}\left(\sigma\right)\right)\right]\le\sup_{m\in B_{N}}H_{{\rm TAP}}^{E,\tilde{\delta}}\left(m\right)+\tilde{\delta}N\right)\ge1-\bar{\kappa}^{\bar{\kappa}}e^{-\frac{\tilde{\delta}}{2}N},\label{eq: from general bound}
\end{equation}
where $\bar{\kappa}=c\max(K,(\beta L/\tilde{\delta})^{8})\le c\max(K,L^{16}/\tilde{\delta}^{8})$.

By the definitions (\ref{eq: HTAP ising def}), (\ref{eq: HTAP measure on subset of sphere})
of $H_{{\rm TAP}}^{{\rm Ising}}\left(m\right)$ and $H_{{\rm TAP}}^{E,\delta}$
and Lemma \ref{lem: ising entropy lemma} with $\tilde{\delta}$ in
place of $\delta$ 
\begin{equation}
\sup_{m\in B_{N}^{\circ}}H_{{\rm TAP}}^{E,\tilde{\delta}}\left(m\right)\le\sup_{m\in B_{N}^{\circ}:d\left(m,\left(-1,1\right)^{N}\right)\le\tilde{\delta}}H_{{\rm TAP}}^{{\rm Ising}}\left(m\right)+cN\tilde{\delta}\log\tilde{\delta}^{-1}.\label{eq: ising UB first step}
\end{equation}
Using Lemma \ref{lem: HTAP lipschitz} we obtain that on the event
$\mathcal{L}_{N}$
\begin{equation}
\sup_{m\in B_{N}^{\circ}\backslash(-1,1)^{N}:d\left(m,\left(-1,1\right)^{N}\right)\le\tilde{\delta}}H_{{\rm TAP}}^{{\rm Ising}}\left(m\right)\le\sup_{m\in\left(-1,1\right)^{N}}H_{{\rm TAP}}^{{\rm Ising}}\left(m\right)+c\left(1+L^{3}\right)N\tilde{\delta}\log\frac{c}{\tilde{\delta}}.\label{eq: Ising second step}
\end{equation}
By picking $\tilde{\delta}>0$ small enough depending on $\delta,L$,
combing (\ref{eq: from general bound})-(\ref{eq: Ising second step})
and using (\ref{eq: the event prob}) we obtain that
\[
\mathbb{P}\left(\log E\left[\exp\left(\beta H_{N}^{f}\left(\sigma\right)\right)\right]\le\sup_{m\in B_{N}^{\circ}}H_{{\rm TAP}}^{{\rm Ising}}\left(m\right)+\delta N\right)\ge1-\bar{\kappa}^{\bar{\kappa}}e^{-\frac{\tilde{\delta}}{2}N}-e^{-N}.
\]
Now by picking $c_{1}\le\frac{\tilde{\delta}}{4}$ and small enough
so that $(cK)^{cK}\le e^{\frac{\tilde{\delta}}{4}N}$ for all $K\le c_{1}N/\log N$
and $N\ge1$, and $1+(cL^{16}/\tilde{\delta}^{8})^{cL^{16}/\tilde{\delta}^{8}}\le c_{1}^{-1}$,
we get that the right-hand side is at least $1-c_{1}^{-1}e^{-c_{1}N}$,
giving the claim (\ref{eq: Ising TAP UB}).
\end{proof}

\section{\label{sec: Spherical SK proof}Bounding the spherical entropy and
proof of spherical TAP upper bound}

In this section we derive the TAP upper bound for the spherical mixed
SK model from the general TAP upper bound, by explicitly bounding
the entropy function $I_{E,\delta}$ from (\ref{eq: IN}) in terms
of by $I_{{\rm sph}}$ from (\ref{eq: spherical entropy I}).
\begin{lem}[Entropy for spherical reference measure]
\label{lem: IN spherical-1}Let $E$ be the uniform measure on $S_{N-1}$
and let $I_{E,\delta}$ be as in (\ref{eq: IN}). For all $\delta\in\left(0,1\right)$
and all large enough $N$ depending on $\delta$ we have for all $m\in B_{N}^{\circ}$
\begin{equation}
I_{E,\delta}\left(m\right)\le\frac{N}{2}\log\left(1-\|m\|^{2}+2\delta\|m\|\right)+\delta N\le I_{{\rm sph}}\left(m\right)+N\delta\left(1+\frac{\|m\|}{1-\|m\|^{2}}\right).\label{eq: IN spherical}
\end{equation}
\end{lem}

\begin{proof}
Firstly simply by choosing $\lambda=\frac{m}{\|m\|}$ in (\ref{eq: IN})
we have that
\begin{equation}
I_{E,\delta}\left(m\right)\le E\left[\left\langle \sigma,\frac{m}{\|m\|}\right\rangle \ge\|m\|-\delta\right].\label{eq: sphere entropy firstly}
\end{equation}
Next by \cite[(2.8)]{BeliusKistler2spin} (whose $E$ is the uniform
distribution on $\left\{ \sigma\in\mathbb{R}^{N}:\left|\sigma\right|=1\right\} $)
we have for any $u$ with $\|u\|=1$ that
\[
E\left[\left\{ \sigma:\left\langle \sigma,u\right\rangle \ge\alpha\right\} \right]=\int_{\alpha}^{1}\frac{1}{\sqrt{\pi}}\frac{\Gamma\left(\frac{N}{2}\right)}{\Gamma\left(\frac{N-1}{2}\right)}\left(1-x^{2}\right)^{\frac{N-3}{2}}dx\le\sqrt{\frac{N}{2\pi}}\left(1-\alpha^{2}\right)^{\frac{N-3}{2}},
\]
where we used that $\frac{\Gamma\left(\frac{N}{2}\right)}{\Gamma\left(\frac{N-1}{2}\right)}\le\sqrt{\frac{N}{2}}.$
Taking the log of both sides we get that for all $\alpha\le1-\delta$
\[
\log E\left[\left\{ \sigma:\left\langle \sigma,u\right\rangle \ge\alpha\right\} \right]\le\frac{N}{2}\log\left(1-\alpha^{2}\right)+\delta N,
\]
since $\log\sqrt{\frac{N}{2\pi}}\left(1-\alpha^{2}\right)^{\frac{-3}{2}}\le\delta N$
for all $N$ large enough for such $\alpha$. Applying this to the
right-hand side of (\ref{eq: sphere entropy firstly}) with $\alpha=\|m\|-\delta$
gives the first inequality of (\ref{eq: IN spherical}). Recalling
(\ref{eq: spherical entropy I}), the second inequality is elementary.
\end{proof}
We now we give the proof of Theorem \ref{thm: main_thm_spherical}
from the general result Theorem \ref{thm: main_thm_general}. We will
use that by (\ref{eq: bound on energy}) and $\xi\left(1\right)\le\xi'\left(1\right)$
\begin{equation}
\mathbb{P}\left(\sup_{m\in B_{N}}\left|H_{N}\left(m\right)\right|\ge c\sqrt{\xi'\left(1\right)}\right)\le e^{-N}\text{\,for all }N\ge1,\label{eq: max bound}
\end{equation}
for a large enough universal $c$.
\begin{proof}[Proof of Theorem \ref{thm: main_thm_spherical}]
 By the general bound Theorem \ref{thm: main_thm_general} with $\tilde{\delta}$
in place of $\delta$ we obtain that for any $\tilde{\delta}>0$
\begin{equation}
\mathbb{P}\left(\log E\left[\exp\left(\beta H_{N}^{f}\left(\sigma\right)\right)\right]\le\sup_{m\in B_{N}}H_{{\rm TAP}}^{E,\tilde{\delta}}\left(m\right)+\tilde{\delta}N\right)\ge1-\bar{\kappa}^{\bar{\kappa}}e^{-\frac{\tilde{\delta}}{2}N},\label{eq: from general bound-1}
\end{equation}
where $\bar{\kappa}=c\max(K,(\beta L/\tilde{\delta})^{8})\le c\max(K,L^{16}/\tilde{\delta}{}^{8})$.
Let $\gamma\in\left(0,1\right)$ to be fixed later. Lemma \ref{lem: IN spherical-1}
implies that for any $\tilde{\delta}\in\left(0,1\right)$ it holds
for $N\ge c(\tilde{\delta})$ that 
\[
I_{E,\tilde{\delta}}\left(m\right)\le I_{{\rm sph}}\left(m\right)+cN\frac{\tilde{\delta}}{\gamma},\text{ for }m\in B_{N}^{\circ}\text{ such that }1-\|m\|^{2}\ge\gamma.
\]
Thus
\begin{equation}
\sup_{m:\|m\|^{2}\le1-\gamma}H_{{\rm TAP}}^{E,\tilde{\delta}}\left(m\right)\le\sup_{m:\frac{\left|m\right|^{2}}{N}\le1-\gamma}H_{{\rm TAP}}^{{\rm sph}}\left(m\right)+c\frac{\tilde{\delta}}{\gamma}N.\label{eq: spherical tap bulk}
\end{equation}

To deal with the supremum over $m$ such that $1-\|m\|^{2}\le\gamma$
we note that the crude bound 
\begin{equation}
\mathbb{P}\left(\sup_{m\in B_{N}^{\circ}}\left\{ \beta H_{N}^{f}\left(m\right)+N\frac{\beta^{2}}{2}{\rm On}\left(\|m\|^{2}\right)\right\} \le\beta f_{N}\left(0\right)+cL^{2}N\right)\ge1-e^{-N},\label{eq: crude UB}
\end{equation}
holds by the bound (\ref{eq: max bound}) (recall $\xi'\left(1\right)\le\xi'''\left(1\right)$),
the Lipschitz assumption on $f_{N}$ and that fact that ${\rm On}$
is bounded by $\xi\left(1\right)$ (see its definition (\ref{eq: Onsager term def})).
Also for $m$ such that $1-\|m\|^{2}\le\gamma$ it holds by Lemma
\ref{lem: IN spherical-1} for $N\ge c(\tilde{\delta})$ that 
\[
I_{E,\tilde{\delta}}\left(m\right)\le N\left(\frac{1}{2}\log\left(\gamma+2\tilde{\delta}\right)+\tilde{\delta}\right).
\]
Thus first choosing $\gamma,\tilde{\delta}$ small enough depending
on the $cL^{2}$ in (\ref{eq: crude UB}) we have for such $m$
\[
H_{{\rm TAP}}^{E,\tilde{\delta}}\left(m\right)\le\beta f_{N}\left(0\right)+cL^{2}N+\frac{N}{2}\log\left(\gamma+2\tilde{\delta}\right)\le\beta f_{N}\left(0\right)\le H_{{\rm TAP}}^{{\rm sph}}\left(0\right),
\]
on the event in (\ref{eq: crude UB}), which implies
\[
\sup_{m\in B_{N}^{\circ}:\|m\|^{2}>1-\gamma}H_{{\rm TAP}}^{E,\tilde{\delta}}\left(m\right)\le H_{{\rm TAP}}^{{\rm sph}}\left(0\right).
\]
Then picking $\tilde{\delta}$ possibly even smaller depending on
$\delta$ so that $c\frac{\tilde{\delta}}{\gamma}\le\delta$, we obtain
from (\ref{eq: from general bound-1}), (\ref{eq: spherical tap bulk})
and (\ref{eq: crude UB}) that 
\[
\mathbb{P}\left(\log E\left[\exp\left(\beta H_{N}^{f}\left(\sigma\right)\right)\right]\le\sup_{m\in B_{N}}H_{{\rm TAP}}^{{\rm sph}}\left(m\right)+\delta N\right)\ge1-\bar{\kappa}^{\bar{\kappa}}e^{-\frac{\tilde{\delta}}{2}N}-e^{-N}.
\]
Now by picking $\tilde{c}_{2}\le\frac{\tilde{\delta}}{4}$ and small
enough so that $\bar{\kappa}^{\bar{\kappa}}\le(cK)^{cK}\le e^{\frac{\tilde{\delta}}{4}N}$
for all $K\le\tilde{c}_{2}N/\log N$ and $N\ge1$, and also $1+`(cL^{16}/\tilde{\delta}{}^{8})^{cL^{16}/\tilde{\delta}{}^{8}}\le\tilde{c}_{2}^{-1}$
we get that the right-hand side is at least $1-\tilde{c}_{2}^{-1}e^{-\tilde{c}_{2}N}$,
proving the claim (\ref{eq: spherical main result}) with $c_{2}=\tilde{c}_{2}$
for $N\ge c(\delta)$. By possibly making $c_{2}$ even smaller depending
on $\delta$ the claim (\ref{eq: spherical main result}) can be made
to hold for all $N\ge1$.
\end{proof}

\appendix

\section{$ $}

Here we collect some basic properties about the random field $H_{N}$
that follow from the classical theory of Gaussian processes.

Recall that $a\cdot b$ is the standard inner product and $\left|\cdot\right|$
the standard norm on $\mathbb{R}^{N}$. We furthermore use the inner
product $\left\langle a,b\right\rangle =a\cdot b/N$ for $a,b\in\mathbb{R}^{N}$
and the norm $\|\cdot\|=\left|\cdot\right|/\sqrt{N}$. Recall that
$B_{N}\left(r\right)\subset\mathbb{R}^{N}$ is the closed and $B_{N}^{\circ}\left(r\right)\subset\mathbb{R}^{N}$
the open ball of radius $r$ in the $\|\cdot\|$-norm, and $S_{N-1}\left(r\right)$
is the the sphere of $\|\cdot\|$-radius $r$. The first lemma gives
the existence of $H_{N}$.
\begin{lem}
\label{lem: existance-1}Let $r>0$. If $\xi\left(x\right)=\sum_{p\ge0}a_{p}x^{p}$
is a power series with non-negative coefficients $a_{p}\ge0$ such
that $\xi\left(r^{2}\right)<\infty$ and $N\ge1$ then there exists
a centered Gaussian process $\left(H_{N}\left(\sigma\right)\right)_{\sigma\in B_{N}\left(r\right)}$
with covariance
\begin{equation}
\mathbb{E}\left[H_{N}\left(\sigma\right)H_{N}\left(\sigma'\right)\right]=N\xi\left(\left\langle \sigma,\sigma'\right\rangle \right)\text{ for all }\sigma,\sigma'\in B_{N}\left(r\right).\label{eq: covar appendix-2}
\end{equation}
\end{lem}

\begin{proof}
If $\xi\left(r^{2}\right)<\infty$ then $\xi\left(q\right)<\infty$
for all $q\in\left[-r^{2},r^{2}\right]$, so $\xi\left(\left\langle \sigma,\sigma'\right\rangle \right)$
is well-defined for all $\sigma,\sigma'\in B_{N}\left(r\right)$.
By Schoenberg's theorem the function $\left(\sigma,\sigma'\right)\to N\xi\left(\left\langle \sigma,\sigma'\right\rangle \right)$
is positive semi-definite \cite{schoenbergPositiveDefiniteFunctions1942a},
so by standard existence results (e.g. \cite[Chapter 1, Proposition 3.7]{RevuzYor-ContMartAndBM})
one can construct the Gaussian process $H_{N}$.
\end{proof}
In the rest of the appendix we will show that $H_{N}$ is also a smooth
function on $B_{N}^{\circ}\left(r\right)$, and provide useful regularity
estimates. The first one is the following. Let $\|\cdot\|_{L^{2}}$
denote the $L^{2}$ norm on the linear space of random variables.
\begin{lem}
For any $r,\xi,N,H_{N}$ as in Lemma \ref{lem: existance-1} and $0<s\le r$
we have 
\begin{equation}
\xi\left(\|\sigma\|^{2}\right)+\xi\left(\|\sigma'\|^{2}\right)-2\xi\left(\left\langle \sigma,\sigma'\right\rangle \right)\le8s\xi'\left(s^{2}\right)\|\sigma-\sigma'\|,\label{eq: var distance inequality-1}
\end{equation}
and
\begin{equation}
\|H_{N}\left(\sigma\right)-H_{N}\left(\sigma'\right)\|_{L^{2}}^{2}\le8s\xi'\left(s^{2}\right)N\|\sigma-\sigma'\|,\label{eq: L2 distance ineq-1}
\end{equation}
for $\sigma,\sigma'\in B_{N}\left(s\right)$.
\end{lem}

\begin{rem}
When $s=r$ and $\xi'\left(r^{2}\right)=\infty$ we interpret the
RHS of (\ref{eq: var distance inequality-1}) and (\ref{eq: L2 distance ineq-1})
as $\infty$, so that the statements are vacuous. Below we use the
same interpretation in (\ref{eq: expected max}), (\ref{eq: bound on energy})
and (\ref{eq: L2 dist kth partial derivs-2}).
\end{rem}

\begin{proof}
Let $\Delta=\sigma'-\sigma$ and
\[
f\left(\lambda\right)=\xi\left(\|\sigma\|^{2}\right)+\xi\left(\|\sigma+\lambda\Delta\|^{2}\right)-2\xi\left(\left\langle \sigma,\sigma+\lambda\Delta\right\rangle \right),\lambda\in\left[0,1\right].
\]
We have $f\left(0\right)=0$ and
\[
\xi\left(\|\sigma\|^{2}\right)+\xi\left(\|\sigma'\|^{2}\right)-2\xi\left(\left\langle \sigma,\sigma'\right\rangle \right)=f\left(1\right)=\int_{0}^{1}f'\left(\lambda\right)d\lambda.
\]
Furthermore
\[
f'\left(\lambda\right)=\xi'\left(\|\sigma+\lambda\Delta\|^{2}\right)\left(2\left\langle \sigma,\Delta\right\rangle +2\lambda\|\Delta\|^{2}\right)-2\xi'\left(\left\langle \sigma,\sigma+\lambda\Delta\right\rangle \right)\left\langle \sigma,\Delta\right\rangle ,
\]
so that  
\[
\sup_{\lambda\in\left[0,1\right]}\left|f'\left(\lambda\right)\right|\le\xi'\left(s^{2}\right)\left(4\left|\left\langle \sigma,\Delta\right\rangle \right|+2\|\Delta\|^{2}\right).
\]
Since $4\left|\left\langle \sigma,\Delta\right\rangle \right|+2\|\Delta\|^{2}\le4\|\sigma\|\|\Delta\|+4s\|\Delta\|\le8s\|\Delta\|$
the claim (\ref{eq: var distance inequality-1}) follows.

The estimate (\ref{eq: L2 distance ineq-1}) is an immediate consequence
of (\ref{eq: var distance inequality-1}) and (\ref{eq: covar appendix-2}).
\end{proof}
We encapsulate some classical results on the regularity of Gaussian
processes in the following lemma.
\begin{lem}
\label{lem: general gaussian theory}There is a universal constant
$c$ such that the following holds. Let $N\ge1$, $s>0$ and let $T\subset B_{N}\left(s\right)$
be a set. For $a\in\left(0,\infty\right)$ assume that $X_{\sigma},\sigma\in T$,
is a centered Gaussian process such that
\begin{equation}
\|X_{\sigma}-X_{\sigma'}\|_{L^{2}}^{2}\le aN\|\sigma-\sigma'\|\text{ for all }\sigma,\sigma'\in T.\label{eq: l2 dist assump}
\end{equation}
Then $\left(X_{\sigma}\right)_{\sigma\in T}$ is almost surely continuous,
and
\begin{equation}
\mathbb{E}\left(\sup_{\sigma\in T}\left|X_{\sigma}\right|\right)\le cN\sqrt{sa},\label{eq: max of gaussian proc general}
\end{equation}
and
\begin{equation}
\mathbb{P}\left(\sup_{\sigma\in T}\left|X_{\sigma}\right|\ge u\right)\le e^{-\frac{u^{2}}{8\sup_{\sigma\in T}\mathbb{E}\left[X_{\sigma}^{2}\right]}}\text{\,for all }u\ge2c\sqrt{sa}.\label{eq: bound on tail}
\end{equation}
\end{lem}

\begin{proof}
We use Dudley's entropy bound \cite[Theorem 1.18, Theroem 2.10]{azaisLevelSetsExtrema2009}.
Consider the distance $d\left(\sigma,\sigma'\right)=\|X_{\sigma}-X_{\sigma'}\|_{L^{2}}$
on $T$. From (\ref{eq: l2 dist assump}) we have $d\left(\sigma,\sigma'\right)\le\sqrt{a}N^{1/4}\left|\sigma-\sigma'\right|^{1/2}$.
Thus if one covers $B_{N}\left(s\right)$ with Euclidean balls of
$\left|\cdot\right|$-radius $(\varepsilon/(c_{1}N^{1/4}))^{2}$ then
balls of $d$-radius $\varepsilon$ centered at the same points also
cover $B_{N}\left(s\right)\supset T$. We have 
\[
\frac{{\rm Vol}\left(\left\{ \sigma:\left|\sigma\right|\le s\sqrt{N}\right\} \right)}{{\rm Vol}\left(\left\{ \sigma:\left|\sigma\right|\le\frac{1}{2}\left(\varepsilon/\left(\sqrt{a}N^{1/4}\right)\right)^{2}\right\} \right)}=\left(\frac{s\sqrt{N}}{\frac{1}{2}\left(\varepsilon/\left(\sqrt{a}N^{1/4}\right)\right)^{2}}\right)^{N}=\left(c_{1}N/\varepsilon^{2}\right)^{N},
\]
with $c_{1}=2sa$. Thus $B_{N}\left(s\right)$ can be covered with
at most at most $N\left(\varepsilon\right)=\max((c_{1}N/\varepsilon^{2})^{N},1)$
Euclidean balls of $\left|\cdot\right|$-radius $(\varepsilon/(c_{1}N^{1/4}))^{2}$.
We have
\[
\int_{0}^{\infty}\sqrt{\log N\left(\varepsilon\right)}=\sqrt{2N}\int_{0}^{\sqrt{c_{1}N}}\sqrt{\log\frac{\sqrt{c_{1}N}}{\varepsilon}}d\varepsilon=N\sqrt{2c_{1}}\int_{0}^{1}\sqrt{\log\frac{1}{\tilde{\varepsilon}}}d\tilde{\varepsilon}.
\]
This is finite, thus \cite[Theorem 1.18]{azaisLevelSetsExtrema2009}
implies that $X_{\sigma},\sigma\in T$, is continuous. Also \cite[Theroem 2.10]{azaisLevelSetsExtrema2009}
implies (\ref{eq: max of gaussian proc general}). Lastly the Borell-TIS
inequality \cite[Theorem 2.8]{azaisLevelSetsExtrema2009} implies
(\ref{eq: bound on tail}).
\end{proof}
Applying this to $H_{N}$ yields the following.
\begin{lem}
\label{lem: max and cond}There is a universal constant $c$, such
that for any $r,\xi,N,H_{N}$ as in Lemma \ref{lem: existance-1}
we have that $H_{N}\left(\sigma\right)$ is continuous almost surely
for $\sigma\in B_{N}^{\circ}\left(r\right)$ (and if $\xi'\left(r^{2}\right)<\infty$
also for $\sigma\in B_{N}\left(r\right)$) and for $0\le s\le r$
it holds that
\begin{equation}
\mathbb{E}\left[{\displaystyle \sup_{\sigma\in B_{N}\left(s\right)}}\left|H_{N}\left(\sigma\right)\right|\right]\le cs\sqrt{\xi'\left(s^{2}\right)}N,\label{eq: expected max}
\end{equation}
and
\begin{equation}
\mathbb{P}\left(\sup_{\sigma\in B_{N}\left(s\right)}\left|H_{N}\left(\sigma\right)\right|\ge u\right)\le e^{-\frac{u^{2}}{8\xi\left(s^{2}\right)N}}\text{ for all }u\ge2cs\sqrt{\xi'\left(s^{2}\right)}N.\label{eq: bound on energy}
\end{equation}
\end{lem}

\begin{proof}
Provided $\xi'\left(s^{2}\right)<1$ the continuity of $H_{N}\left(\sigma\right),\sigma\in B_{N}\left(s\right)$,
(\ref{eq: expected max}) and (\ref{eq: bound on energy}) follow
by applying Lemma \ref{lem: general gaussian theory} with $T=B_{N}\left(s\right)$,
$X_{\sigma}=H_{N}\left(\sigma\right)$ and $a=8s\xi'\left(s^{2}\right)$,
since we have the bound (\ref{eq: L2 distance ineq-1}). If $\xi'\left(r^{2}\right)<\infty$
then with $s=r$ continuity on $B_{N}\left(r\right)$ follows. If
$\xi'\left(r^{2}\right)=\infty$ then continuity on $B_{N}\left(s\right)$
for all $s<r$ (note $\xi'\left(s^{2}\right)<\infty$ for $s<r$)
implies continuity on $B_{N}^{\circ}\left(r\right)$.
\end{proof}
We now turn to the derivatives of $H_{N}$. For a multi-index $\alpha\in\left\{ 1,\ldots,N\right\} ^{l}$
let $\partial_{\alpha}$ denote the corresponding partial derivative,
and write $\partial_{\alpha}^{x}$ for the partial derivative with
respect to a variable $x$. Let $e_{1},\ldots,e_{N}$ denote the
standard basis vectors of $\mathbb{R}^{N}$. The next lemma will be
applied with $C$ the covariance of a Gaussian process to prove differentiability.
\begin{lem}
\label{lem: taylor lemma}If $0<s<r$ and $C:B_{N}^{\circ}\left(r\right)\times B_{N}^{\circ}\left(r\right)\to\mathbb{R}$
has derivatives up to fourth order that are Lipschitz functions and
$i,j\in\left\{ 1,\ldots,N\right\} $, then there exists a Lipschitz
function $R:B_{N}^{\circ}\left(s\right)\times B_{N}^{\circ}\left(s\right)\times(0,r-s)^{2}\to\mathbb{R}$
that satisfies $R\left(\sigma,\sigma',\eta_{1},\eta_{2}\right)=O\left(\left|\eta_{1}\right|+\left|\eta_{2}\right|\right)$
such that for all $\sigma,\sigma'\in B_{N}^{\circ}\left(s\right)$
and $0<\eta,\eta'\le r-s$ 
\begin{equation}
\frac{1}{\eta_{1}\eta_{2}}\sum_{s_{1},s_{2}\in\left\{ 0,1\right\} }\left(-1\right)^{s_{1}+s_{2}}C\left(\sigma+s_{1}\eta_{1}e_{i},\sigma+s_{2}\eta_{2}e_{j}\right)=\partial_{i}^{\sigma}\partial_{j}^{\sigma'}C\left(\sigma,\sigma'\right)+R\left(\sigma,\sigma',\eta,\eta'\right),\label{eq: taylor lemma}
\end{equation}
where the constant in the $O$-term, the Lipschitz constant of $R$
and $R$ itself depend on $C,i,j$.
\end{lem}

\begin{proof}
The LHS can be written as
\[
\frac{1}{\eta_{2}}\sum_{s_{2}\in\left\{ 0,1\right\} }\left(-1\right)^{s_{2}}\frac{C\left(\sigma+\eta_{1}e_{i},\sigma+s_{2}\eta_{2}e_{j}\right)-C\left(\sigma,\sigma+s_{2}\eta_{2}e_{j}\right)}{\eta_{1}}.
\]
By Taylor's theorem with integral remainder in the $\sigma$ variable
this equals
\[
\begin{array}{l}
{\displaystyle \frac{1}{\eta_{2}}\sum_{s_{2}\in\left\{ 0,1\right\} }\left(-1\right)^{s_{2}}\left(\partial_{i}^{\sigma}C\left(\sigma,x+s_{2}\eta_{2}e_{j}\right)+\eta_{1}\int_{0}^{1}\left(1-t\right)\left(\partial_{i}^{\sigma}\right)^{2}C\left(\sigma+t_{1}\eta_{1}e_{i},\sigma'+s_{2}\eta_{2}e_{j}\right)dt_{1}\right)}\\
=\frac{1}{\eta_{2}}\sum_{s_{2}\in\left\{ 0,1\right\} }\left(-1\right)^{s_{2}}\sum_{g_{1}\in\left\{ 0,1\right\} }\eta_{1}^{g_{1}}\int_{0}^{1}\left(\partial_{i}^{\sigma}\right)^{1+g_{1}}C\left(\sigma+t_{1}\eta_{1}e_{i},\sigma'+s_{2}\eta_{2}e_{j}\right)dt_{1}.
\end{array}
\]
Next applying the same theorem in $\sigma'$ this equals
\[
\sum_{g_{1},g_{2}\in\left\{ 0,1\right\} }\eta_{1}^{g_{1}}\eta_{2}^{g_{2}}\int_{0}^{1}\int_{0}^{1}(\partial_{j}^{\sigma'})^{1+g_{2}}(\partial_{i}^{\sigma})^{1+g_{1}}\left(1-t_{1}\right)\left(1-t_{2}\right)C\left(\sigma+t_{1}\eta_{1}e_{i},\sigma'+t_{2}\eta_{2}e_{j}\right)dt_{2}dt_{2}.
\]
The summand when $g_{1}=g_{2}=0$ is $\partial_{i}^{\sigma}\partial_{j}^{\sigma'}C\left(\sigma,\sigma'\right)$,
and the other summands are $O\left(\left|\eta_{1}\right|+\left|\eta_{2}\right|\right)$
and Lipschitz in $\eta_{1},\eta_{2},\sigma,\sigma'$ by the assumption
on $C$.
\end{proof}

With this lemma we can prove the existence of the derivatives of $H_{N}$.
\begin{lem}
\label{lem: grad existance}Let $r,\xi,N,H_{N}$ be as in Lemma \ref{lem: existance-1}.
Then $H_{N}$ is almost surely smooth on $B_{N}^{o}\left(r\right)$
and $\left(\partial_{\alpha}H_{N}\left(\sigma\right)\right)_{\sigma\in B_{N}^{o}\left(r\right),\alpha\in\cup_{l=0}^{\infty}\left\{ 1,\ldots,N\right\} ^{l}}$
is a centered Gaussian process with covariance
\begin{equation}
\mathbb{E}\left[\partial_{\alpha}H_{N}\left(\sigma\right)\partial_{\alpha'}H_{N}\left(\sigma'\right)\right]=N\partial_{\alpha}^{\sigma}\partial_{\alpha'}^{\sigma'}\xi\left(\left\langle \sigma,\sigma'\right\rangle \right).\label{eq: covar of as partial deriv}
\end{equation}
\end{lem}

\begin{proof}
Since $\xi$ is a convergent power series in $B_{N}^{\circ}\left(r\right)$,
the function $\left(\sigma,\sigma'\right)\to\xi\left(\left\langle \sigma,\sigma'\right\rangle \right)$
is infinitely differentiable. Consider the statement
\begin{equation}
\begin{array}{c}
\partial_{\alpha}H_{N}\left(\sigma\right)\text{ exists and is continuous for all }\sigma\in B_{N}^{\circ}\left(r\right),\left|\alpha\right|\le k,\\
\text{and is a centered Gaussian process with covariance }\eqref{eq: covar of as partial deriv}.
\end{array}\label{eq: induc step}
\end{equation}
If this holds for all $k$ then it implies the claim of the lemma,
since existence and continuity of partial derivatives implies differentiability.

To prove (\ref{eq: induc step}) we use induction on $k$. The case
$k=0$ follows since $H_{N}$ is continuous on $B_{N}^{\circ}\left(r\right)$
by Lemma \ref{lem: max and cond}. Assume (\ref{eq: induc step})
holds for $k\le l$. Fix an $s\in\left(0,r\right)$. Consider for
any $\alpha$ with $\left|\alpha\right|=l$ and any $i=1,\ldots,N,0<\eta<r-s,\sigma\in B_{N}^{o}\left(s\right)$
\[
\Delta_{\alpha,i}\left(\sigma,\eta\right)=\frac{\partial_{\alpha}H_{N}\left(\sigma+\eta e_{i}\right)-\partial_{\alpha}H_{N}\left(\sigma\right)}{\eta}.
\]
Note that for $0<\eta,\eta'\le r-s,$ $\sigma,\sigma'\in B_{N}^{o}\left(s\right)$
and $\alpha,\alpha'$ with $\left|\alpha\right|,\left|\alpha'\right|\le l$
the covariance $N^{-1}\mathbb{E}\left[\Delta_{\alpha,i}\left(\sigma,\eta_{1}\right)\Delta_{\alpha',j}\left(\sigma',\eta_{2}\right)\right]$
equals the LHS of (\ref{eq: taylor lemma}) with $C\left(\sigma,\sigma'\right)=C_{\alpha,\alpha'}\left(\sigma,\sigma'\right)=\partial_{\alpha}^{\sigma}\partial_{\alpha'}^{\sigma'}\xi\left(\left\langle \sigma,\sigma'\right\rangle \right)$,
so by Lemma \ref{lem: taylor lemma}
\begin{equation}
\mathbb{E}\left[\Delta_{\alpha,i}\left(\sigma,\eta_{1}\right)\Delta_{\alpha',j}\left(\sigma',\eta_{2}\right)\right]=N\partial_{i}^{\sigma}\partial_{j}^{\sigma'}C\left(\sigma,\sigma'\right)+R\left(\sigma,\sigma',\eta,\eta'\right),\label{eq: covar of delta}
\end{equation}
where $R\left(\sigma,\sigma',\eta,\eta'\right)$ is $O\left(\left|\eta\right|+\left|\eta'\right|\right)$
and Lipschitz. This implies that $\|\Delta_{\alpha,i}\left(\sigma,\eta\right)-\Delta_{\alpha,i}\left(\sigma,\eta'\right)\|_{L^{2}}^{2}=O\left(\left|\eta\right|+\left|\eta'\right|\right)$.
Thus $\Delta_{\alpha,i}\left(\sigma,\eta\right)$ for $\eta\downarrow0$
is a $L^{2}$-Cauchy sequence, and there exists a random variable
$D^{i}\partial_{\alpha}H_{N}\left(\sigma\right)$ such that $\Delta_{\alpha,i}\left(\sigma,\eta\right)\to D^{i}\partial H_{N}\left(\sigma\right)$
in $L^{2}$, as $\eta\to0$. Also $D^{i}\partial H_{N}\left(\sigma\right)$
is a centered Gaussian since it is the limit of centered Gaussians,
and jointly Gaussian with $\partial_{\alpha}H_{N}\left(\sigma\right),\sigma\in B_{N}^{\circ}\left(s\right),0<\eta<r-s$.
Next define 
\[
g_{i}\left(\sigma,\eta\right)=\begin{cases}
\Delta_{\alpha,i}\left(\sigma,\eta\right) & \text{ if }\eta\ne0,\\
D^{i}\partial H_{N}\left(\sigma\right) & \text{ if }\eta=0.
\end{cases}
\]
Then $g_{i}\left(\sigma,\eta\right)$ is a Gaussian process on $T=B_{N}^{o}\left(s\right)\times\left(-\left(r-s\right),r-s\right)\subset B_{N+1}\left(\sqrt{s^{2}+r^{2}}\right)$.
By (\ref{eq: covar of delta}) and the Lipschitz property of $R$
\[
\|g_{i}\left(\sigma,\eta\right)-g_{i}\left(\sigma',\eta'\right)\|_{L^{2}}^{2}\le c\sqrt{\left|\sigma-\sigma'\right|^{2}+\left|\eta-\eta'\right|^{2}},
\]
first for $\eta,\eta\ne0$ and by the $L^{2}$ convergence to $D^{i}\partial H_{N}\left(\sigma\right)$
resp. $D^{i}\partial H_{N}\left(\sigma'\right)$ for all $\eta,\eta'$.
Therefore by Lemma \ref{lem: general gaussian theory} the process
$g_{i}\left(\sigma,\eta\right)$ is almost surely continuous. Thus
the limit $\lim_{\eta\downarrow0}\Delta_{\alpha,i}\left(\sigma,\eta\right)$
exists and equals $D^{i}\partial H_{N}\left(\sigma\right)$ for all
$\sigma$, almost surely. Then for all $i$ and $\sigma\in B_{N}^{\circ}\left(s\right)$
the derivative $\partial_{i}\partial_{\alpha}H_{N}\left(\sigma\right)=g_{i}\left(\sigma,0\right)$
exists and is continuous. Since this holds for all $s<r$ it also
holds for $s=r$, proving all of (\ref{eq: induc step}) for $l=k+1$
except for the formula for the covariance. The latter then follows
since for any $\sigma,\sigma'\in B_{N}^{o}\left(r\right),i,j=1,\ldots,N$,
$\alpha,\alpha'$ with $\left|\alpha\right|=\left|\alpha'\right|\le l$
and $\eta,\eta'$ small enough we have 
\[
\mathbb{E}\left[\Delta_{\alpha,i}\left(\sigma,\eta\right)\Delta_{\alpha',j}\left(\sigma',\eta\right)\right]\overset{\eqref{eq: covar of delta}}{\to}N\partial_{i}^{\sigma}\partial_{j}^{\sigma'}C\left(\left\langle \sigma,\sigma'\right\rangle \right)=\partial_{i}^{\sigma}\partial_{\alpha}^{\sigma}\partial_{j}^{\sigma'}\partial_{\alpha'}^{\sigma'}\xi\left(\left\langle \sigma,\sigma'\right\rangle \right)\text{ as }\eta\to0.
\]
\end{proof}
Lastly we give a basic regularity estimate for the derivatives of
$H_{N}$. The \emph{spectral norm }of $\nabla^{k}H_{N}\left(\sigma\right)=\left(\partial_{i_{1}\ldots i_{k}}H_{N}\left(\sigma\right)\right)_{i_{1},\ldots,i_{n}=1,\ldots,N}$
viewed as a tensor is $\sup_{v:\|v\|=1}\left|\nabla^{k}H_{N}\left(\sigma\right)v^{\otimes k}\right|$
where
\[
\nabla^{k}H_{N}\left(\sigma\right)v^{\otimes k}=\sum_{i_{1},\ldots,i_{p}=1}^{N}\partial_{i_{1}\ldots i_{p}}H_{N}\left(\sigma\right)v_{i_{1}}\ldots v_{i_{p}}.
\]
We have the following.
\begin{lem}
Let $r,\xi,N,H_{N}$ be as in Lemma \ref{lem: existance-1}, and $s<r$
and $k\ge0$. We have for all $a,b\in B_{N}\left(s\right)$ and $v,w\in S_{N}\left(s\right)$
\begin{equation}
\begin{array}{l}
\|\nabla^{k}H_{N}\left(a\right)v^{\otimes k}-\nabla^{k}H_{N}\left(b\right)w{}^{\otimes k}\|_{L^{2}}^{2}\\
\le c_{k}Ns^{4k+1}\xi^{\left(2k+1\right)}\left(s^{2}\right)\sqrt{\|v-w\|^{2}+\|a-b\|^{2}},
\end{array}\label{eq: L2 dist kth partial derivs-2}
\end{equation}
where $c_{k}=c\left(k+1\right)!2^{k},$and
\begin{equation}
\sup_{v\in S_{N}\left(s\right)}\mathbb{E}\left[\left(\nabla^{k}H_{N}\left(a\right)v^{\otimes k}\right)^{2}\right]=Ns^{2k}\sum_{l=0}^{k}{k \choose l}\left(k-l\right)!s^{2l}\xi^{\left(k+l\right)}\left(\|a\|^{2}\right).\label{eq: variance}
\end{equation}
\end{lem}

\begin{proof}
By Lemma \ref{lem: grad existance}
\begin{equation}
\begin{array}{l}
\mathbb{E}\left[\left(\nabla^{k}H_{N}\left(a\right)v^{\otimes k}\right)\left(\nabla^{k}H_{N}\left(b\right)w{}^{\otimes k}\right)\right]\\
=N\sum_{i_{1}\ldots i_{k}j_{1}\ldots j_{k}}\partial_{i_{1}}^{a}\ldots\partial_{i_{k}}^{a}\partial_{j_{1}}^{b}\ldots\partial_{j_{k}}^{b}\xi\left(\left\langle a,b\right\rangle \right)v_{i_{1}}\ldots v_{i_{k}}w_{j_{1}}\ldots w_{j_{k}}.
\end{array}\label{eq: tensor covar}
\end{equation}
We have
\[
\partial_{j_{1}}^{b}\ldots\partial_{j_{k}}^{b}\xi\left(\left\langle a,b\right\rangle \right)=N^{-k}a_{j_{1}}\ldots a_{j_{k}}\xi^{\left(k\right)}\left(\left\langle a,b\right\rangle \right),
\]
and by the product rule
\[
\begin{array}{l}
\partial_{i_{k}}^{a}a_{j_{1}}\ldots a_{j_{k}}\xi^{\left(k\right)}\left(\left\langle a,b\right\rangle \right)\\
=\sum_{r=1}^{k}\delta_{i_{k}j_{r}}\prod_{l\ne r}a_{j_{l}}\xi^{\left(k\right)}\left(\left\langle a,b\right\rangle \right)+N^{-1}a_{j_{1}}\ldots a_{j_{k}}b_{i_{k}}\xi^{\left(k+1\right)}\left(\left\langle a,b\right\rangle \right).
\end{array}
\]
Applying also the other derivatives $\partial_{i_{1}}^{a}\ldots\partial_{i_{k-1}}^{a}$
we get a large sum, which can be expressed as follows. Letting $I$
denote the set of indices $r$ of $i_{r}$ where the derivative $\partial_{i_{r}}^{a}$
is applied to one of $a_{j_{1}},\ldots,a_{j_{k}}$, and letting $\pi\left(r\right)$
denote the index of the factor it is applied to, we get that
\[
\begin{array}{l}
N\partial_{i_{1}}^{a}\ldots\partial_{i_{k}}^{a}\partial_{j_{1}}^{b}\ldots\partial_{j_{k}}^{b}\xi\left(\left\langle a,b\right\rangle \right)=N^{1-k}\partial_{i_{1}}^{a}\ldots\partial_{i_{k}}^{a}a_{j_{1}}\ldots a_{j_{k}}\xi\left(\left\langle a,b\right\rangle \right)\\
=\sum_{I}\sum_{\pi:I\to\left\{ 1,\ldots,k\right\} }\left(\prod_{r\in I}\delta_{i_{r}j{}_{\pi\left(r\right)}}\right)\left(\prod_{r\notin I}b_{i_{r}}\right)\left(\prod_{r\notin\pi\left(I\right)}a_{j_{r}}\right)N^{1-\left(2k-\left|I\right|\right)}\xi^{\left(2k-\left|I\right|\right)}\left(\left\langle a,b\right\rangle \right),
\end{array}
\]
where $I$ is summed over all subsets of $\left\{ 1,\ldots,k\right\} $
and $\pi$ over all injective maps from $I$ to $\left\{ 1,\ldots,k\right\} $.
Applying the sums over $i_{1},\ldots,i_{k},j_{1},\ldots,j_{k}$ from
(\ref{eq: tensor covar}) we get
\[
\mathbb{E}\left[\left(\nabla^{k}H_{N}\left(a\right)v^{\otimes k}\right)\left(\nabla^{k}H_{N}\left(b\right)w{}^{\otimes k}\right)\right]=N\sum_{I}\sum_{\pi:I\to\left\{ 1,\ldots,k\right\} }G_{\left|I\right|}\left(a,v,b,w\right),
\]
where
\[
G_{n}\left(a,v,b,w\right)=\left\langle v,w\right\rangle ^{n}\left\langle a,w\right\rangle ^{k-n}\left\langle b,v\right\rangle ^{k-n}\xi^{\left(2k-n\right)}\left(\left\langle a,b\right\rangle \right).
\]
From this (\ref{eq: variance}) follows. It also implies that
\begin{equation}
\begin{array}{l}
\|\nabla^{k}H_{N}\left(a\right)v^{\otimes k}-\nabla^{k}H_{N}\left(b\right)w{}^{\otimes k}\|_{L^{2}}^{2}\\
=N\sum_{I}\sum_{\pi:I\to\left\{ 1,\ldots,k\right\} }\left\{ G_{\left|I\right|}\left(a,v,a,v\right)+G_{\left|I\right|}\left(b,w,b,w\right)-2G_{\left|I\right|}\left(a,v,b,w\right)\right\} .
\end{array}\label{eq: l2 diff}
\end{equation}
Using that $\|a\|,\|b\|,\|v\|,\|w\|\le s\sqrt{N}$ and $\left|x^{l}-y^{l}\right|\le\left|x-y\right|l\max\left(\left|x\right|,\left|y\right|\right)^{l-1}$
we get
\[
\begin{array}{rcccl}
\left|\left\langle v,v\right\rangle ^{n}-\left\langle v,w\right\rangle ^{n}\right| & \le & \left|\left\langle v,v\right\rangle -\left\langle v,w\right\rangle \right|ks^{2\left(n-1\right)} & \le & k\|v-w\|s^{2n-1},\\
\left|\left\langle a,w\right\rangle ^{k-n}-\left\langle a,v\right\rangle ^{k-n}\right| & \le & \left|\left\langle a,v\right\rangle -\left\langle a,w\right\rangle \right|ks^{2\left(k-n-1\right)} & \le & k\|v-w\|s^{2k-2n-1},\\
\left|\left\langle a,v\right\rangle ^{k-n}-\left\langle b,v\right\rangle ^{k-n}\right| & \le & \left|\left\langle a,v\right\rangle -\left\langle b,v\right\rangle \right|ks^{2\left(k-n-1\right)} & \le & k\|a-b\|s^{2k-2n-1},\\
\left|\left\langle a,v\right\rangle ^{k-n}-\left\langle b,w\right\rangle ^{k-n}\right| & \le & \left|\left\langle a,v\right\rangle -\left\langle b,w\right\rangle \right|ks^{2\left(k-n-1\right)} & \le & k\left(\|a-b\|+\|v-w\|\right)s^{2k-2n-1}.
\end{array}
\]
From these we obtain
\[
\begin{array}{l}
G_{n}\left(a,v,a,v\right)+G_{n}\left(b,w,b,w\right)-2G_{n}\left(a,v,b,w\right)\\
\le ck\left(\|a-b\|+\|v-w\|\right)s^{4k-2n-1}\xi^{\left(2k-n\right)}\left(s^{2}\right)\\
+s^{4k-2n}\left\{ \xi^{\left(2k-n\right)}\left(\left\langle a,a\right\rangle \right)+\xi^{\left(2k-n\right)}\left(\left\langle b,b\right\rangle \right)-2\xi^{\left(2k-n\right)}\left(\left\langle a,b\right\rangle \right)\right\} .
\end{array}
\]
Next by (\ref{eq: var distance inequality-1}) and the inequality
$\xi^{\left(l\right)}\left(s^{2}\right)\le s^{2m}\xi^{\left(l+m\right)}\left(s^{2}\right)$
this is at most
\[
ck\left(\|a-b\|+\|v-w\|\right)s^{4k+1}\xi^{\left(2k+1\right)}\left(s^{2}\right).
\]
Thus from (\ref{eq: l2 diff}) and $\sum_{I}\sum_{\pi:I\to\left\{ 1,\ldots,k\right\} }1\le2^{k}k!$
the claim (\ref{eq: L2 dist kth partial derivs-2}) follows.
\end{proof}
From this we derive the following estimates for the spectral norm
of $\nabla^{k}H_{N}\left(\sigma\right)$.
\begin{lem}
\label{lem: spectral norm bound}There is a universal constant $c$
such that the following holds. Let $r,\xi,N,H_{N}$ be as in Lemma
\ref{lem: existance-1}, and $0<s<r$. Then with $c_{k}=c\left(k+1\right)!2^{k}$
\begin{equation}
\mathbb{E}\left[\sup_{\sigma\in B_{N}\left(s\right)}\sup_{v:\|v\|=1}\left|\nabla^{k}H_{N}\left(\sigma\right)v^{\otimes k}\right|\right]\le\sqrt{c_{k}}Ns^{k+1}\sqrt{\xi^{\left(2k+1\right)}\left(s^{2}\right)},\label{eq: tensor max expectaqtion}
\end{equation}
and letting $w^{2}=s^{2k}\sum_{l=0}^{k}{k \choose l}\left(k-l\right)!s^{2l}\xi^{\left(k+l\right)}\left(s^{2}\right)$
also
\begin{equation}
\mathbb{P}\left(\sup_{\sigma\in B_{N}\left(s\right)}\sup_{v:\|v\|=1}\left|\nabla^{k}H_{N}\left(\sigma\right)v^{\otimes k}\right|\ge u\right)\le e^{-\frac{u^{2}}{8w^{2}N}}\text{ for }u\ge2\sqrt{c_{k}}Ns^{k+1}\sqrt{\xi^{\left(2k+1\right)}\left(s^{2}\right)}.\label{eq: tensor tail}
\end{equation}
\end{lem}

\begin{proof}
Let $T=B_{N}\left(s\right)\times S_{N-1}\left(s\right)$. It follows
from Lemma \ref{lem: grad existance} that $\nabla^{k}H_{N}\left(\sigma\right)v^{\otimes k},\left(\sigma,v\right)\in T$,
is a centered Gaussian process. The claims then follow from (\ref{eq: L2 dist kth partial derivs-2})
and Lemma \ref{lem: general gaussian theory} with $2N$ in place
of $N$, $2s$ in place of $s$, $T\subset B_{2N}\left(\sqrt{2}s\right)$
and $a=c_{k}Ns^{4k+1}\xi^{\left(2k+1\right)}\left(s^{2}\right)$,
after dividing by $s^{k}$ to normalize $v$. 
\end{proof}
In this article we use the following special cases. Recall that
$B_{N}=B_{N}\left(1\right),B_{N}^{\circ}=B_{N}^{\circ}\left(1\right)$.

\begin{lem}
\label{lem: energy and first deriv covar}If $\xi\left(x\right)=\sum_{p\ge0}a_{p}x^{p}$
is a power series with non-negative coefficients $a_{p}\ge0$ such
that $\xi\left(1\right)<\infty$ and $N\ge1$, then there exists a
centered Gaussian process $\left(H_{N}\left(\sigma\right)\right)_{\sigma\in B_{N}}$
with covariance 
\[
\mathbb{E}\left[H_{N}\left(\sigma\right)H_{N}\left(\sigma'\right)\right]=N\xi\left(\left\langle \sigma,\sigma'\right\rangle \right)\text{ for all }\sigma,\sigma'\in B_{N},
\]
which is almost surely differentiable in $B_{N}^{o}$. Also $\left(H_{N}\left(\sigma\right),\nabla H_{N}\left(\sigma\right)\right)$
is a centered Gaussian process satisfying for all $i,j=1,\ldots,N$
and $\sigma,\sigma'\in B_{N}^{o}$
\begin{equation}
\mathbb{E}\left[\partial_{i}H_{N}\left(\sigma\right)\partial_{j}H_{N}\left(\sigma'\right)\right]=\delta_{ij}\xi'\left(\left\langle \sigma,\sigma'\right\rangle \right)+\frac{\sigma_{j}\sigma_{i}^{'}}{N}\xi''\left(\left\langle \sigma,\sigma'\right\rangle \right)\label{eq: grad grad covar}
\end{equation}
and
\begin{equation}
\mathbb{E}\left[H_{N}\left(\sigma\right)\partial_{i}H_{N}\left(\sigma'\right)\right]=\sigma_{i}\xi'\left(\left\langle \sigma,\sigma'\right\rangle \right).\label{eq: energy grad covar}
\end{equation}
\end{lem}

\begin{proof}
These are a special case of Lemmas \ref{lem: existance-1}, \ref{lem: grad existance}.
\end{proof}
\begin{lem}
\label{lem: grad estimate}For $\xi,N,H_{N}$ as in Lemma \ref{lem: energy and first deriv covar}
with $\xi'\left(1\right)<\infty$ it holds that
\begin{equation}
\mathbb{P}\left(\sup_{\sigma\in B_{N}}\left|H_{N}\left(\sigma\right)\right|\ge u\right)\le e^{-\frac{u^{2}}{8\xi\left(1\right)N}}\text{ for all }u\ge c\sqrt{\xi'\left(1\right)}N,\label{eq: max bound special case}
\end{equation}
and if also $\xi'''\left(1\right)<\infty$ it holds for all $u\ge c\xi'''\left(1\right)$
that

\begin{equation}
\mathbb{P}\left(\sup_{\sigma\in B_{N}^{\circ}}\|\nabla H_{N}\left(\sigma\right)\|\ge u\right)\le e^{-\frac{u^{2}}{8\left(\xi''\left(1\right)+\xi'\left(1\right)\right)}N},\label{eq: bound on grad}
\end{equation}
and
\begin{equation}
\mathbb{P}\left(\exists\sigma,\sigma'\in B_{N}^{\circ}\text{ s.t. }\left|H_{N}\left(\sigma\right)-H_{N}\left(\sigma'\right)\right|\ge uN\|\sigma-\sigma'\|\right)\le e^{-\frac{u^{2}}{8\left(\xi''\left(1\right)+\xi'\left(1\right)\right)}N}.\label{eq: lipschitz}
\end{equation}
\end{lem}

\begin{proof}
The bound (\ref{eq: max bound special case}) is a special case of
(\ref{eq: bound on energy}), and (\ref{eq: bound on grad}) is the
case $k=1$ of Lemma \ref{lem: spectral norm bound} and implies (\ref{eq: lipschitz})
via the mean value theorem.
\end{proof}

\printbibliography

\end{document}